\pdfoutput=1
\batchmode 
\AtBeginDocument{\scrollmode}

\documentclass[11pt,reqno]{amsart}

\usepackage[T1]{fontenc}
\usepackage[utf8]{inputenc}

\usepackage[english]{babel}

\usepackage{microtype} 
\usepackage{geometry}

\usepackage[lowtilde]{url}
\usepackage[unicode]{hyperref}

\usepackage[capitalise]{cleveref}
\crefformat{section}{#2§#1#3}
\Crefformat{section}{#2§#1#3}
\crefrangeformat{section}{#3§§#1#4--#5#2#6}
\Crefrangeformat{section}{#3§§#1#4--#5#2#6}

\crefformat{sequence}{Seq.~(#2#1#3)}

\AddToHook{env/theorem/begin}{\crefalias{theoremcounter}{theorem}}
\crefname{theorem}{Theorem}{Theorems}
\crefformat{theorem}{Theorem~#2#1#3} 
\Crefformat{theorem}{Theorem~#2#1#3}
\AddToHook{env/conjecture/begin}{\crefalias{theoremcounter}{conjecture}}
\crefname{conjecture}{Conjecture}{Conjectures}
\crefformat{conjecture}{Conjecture~#2#1#3} 
\Crefformat{conjecture}{Conjecture~#2#1#3}
\AddToHook{env/corollary/begin}{\crefalias{theoremcounter}{corollary}}
\crefname{corollary}{Corollary}{Corollaries}
\crefformat{corollary}{Corollary~#2#1#3} 
\Crefformat{corollary}{Corollary~#2#1#3}
\AddToHook{env/lemma/begin}{\crefalias{theoremcounter}{lemma}}
\crefname{lemma}{Lemma}{Lemmas}
\crefformat{lemma}{Lemma~#2#1#3} 
\Crefformat{lemma}{Lemma~#2#1#3}
\AddToHook{env/proposition/begin}{\crefalias{theoremcounter}{proposition}}
\crefname{proposition}{Proposition}{Propositions}
\crefformat{proposition}{Proposition~#2#1#3} 
\Crefformat{proposition}{Proposition~#2#1#3}
\AddToHook{env/desideratum/begin}{\crefalias{theoremcounter}{desideratum}}
\crefname{desideratum}{Desideratum}{Desideratums}
\crefformat{desideratum}{Desideratum~#2#1#3} 
\Crefformat{desideratum}{Desideratum~#2#1#3}
\AddToHook{env/remark/begin}{\crefalias{theoremcounter}{remark}}
\crefname{remark}{Remark}{Remarks}
\crefformat{remark}{Remark~#2#1#3} 
\Crefformat{remark}{Remark~#2#1#3}
\AddToHook{env/question/begin}{\crefalias{theoremcounter}{question}}
\crefname{question}{Question}{Questions}
\crefformat{question}{Question~#2#1#3} 
\Crefformat{question}{Question~#2#1#3}
\AddToHook{env/definition/begin}{\crefalias{theoremcounter}{definition}}
\crefname{definition}{Definition}{Definitions}
\crefformat{definition}{Definition~#2#1#3} 
\Crefformat{definition}{Definition~#2#1#3}
\AddToHook{env/notation/begin}{\crefalias{theoremcounter}{notation}}
\crefname{notation}{Notation}{Notations}
\crefformat{notation}{Notation~#2#1#3} 
\Crefformat{notation}{Notation~#2#1#3}
\AddToHook{env/example/begin}{\crefalias{theoremcounter}{example}}
\crefname{example}{Example}{Examples}
\crefformat{example}{Example~#2#1#3} 
\Crefformat{example}{Example~#2#1#3}
\AddToHook{env/situation/begin}{\crefalias{theoremcounter}{situation}}
\crefname{situation}{Situation}{Situations}
\crefformat{situation}{Situation~#2#1#3} 
\Crefformat{situation}{Situation~#2#1#3}
\AddToHook{env/assumption/begin}{\crefalias{theoremcounter}{assumption}}
\crefname{assumption}{Assumption}{Assumptions}
\crefformat{assumption}{Assumption~#2#1#3} 
\Crefformat{assumption}{Assumption~#2#1#3}
\AddToHook{env/passage/begin}{\crefalias{theoremcounter}{passage}}
\crefname{passage}{Paragraph}{Paragraphs}
\crefformat{passage}{#2\P#1#3} 
\Crefformat{passage}{#2\P#1#3}

\PassOptionsToPackage{%
  backend=biber,hyperref,
  bibencoding=auto,style=alphabetic,citestyle=alphabetic,
  doi=true,url=true,isbn=false,
  maxbibnames=99,giveninits=true
  }{biblatex}

\usepackage{biblatex} 

\DeclareSourcemap{\maps[datatype=bibtex]{
  \map[overwrite]{
    \step[fieldsource=doi, final]
    \step[fieldset=url, null]
    \step[fieldset=eprint, null]
  }
  \map[overwrite]{
    \step[fieldset=language, null]
  }
  \map[overwrite]{
    \pertype{book}
    \step[fieldset=pages, null]
  }
}}

\renewbibmacro{in:}{%
  \ifentrytype{article}{}{\printtext{\bibstring{in}\intitlepunct}}}

\usepackage{amsmath}
\usepackage{mathtools}
\usepackage{amsthm}
\usepackage{amssymb}

\usepackage{tikz-cd}

\usepackage{fixmath} 

\usepackage[inline]{enumitem}

\usepackage{empheq} 

\usepackage{etoolbox} 

\makeatletter
\providecommand{\phantomsection}{} 
\newenvironment{acknowledgements}{%
  \@startsection{paragraph}{5}%
  {0pt}{.5\linespacing\@plus.7\linespacing}{-.5em}%
  {\normalfont\normalsize\bfseries}{Acknowledgements}%
  \phantomsection%
  }{}
\makeatother

\setlist[enumerate]{nosep,label=(\roman*),font=\normalfont,leftmargin=2.8em}
\AtBeginEnvironment{enumerate}{~}

\usepackage{chngcntr} 

\usetikzlibrary{decorations.markings}
\tikzcdset{
  open/.code     = {\tikzcdset{hook, circled};},
  closed/.code   = {\tikzcdset{hook, slashed};},
  open'/.code    = {\tikzcdset{hook', circled};},
  closed'/.code  = {\tikzcdset{hook', slashed};},
  circled/.code  = {\tikzcdset{markwith = {\draw (0,0) circle (.375ex);}};},
  slashed/.code  = {\tikzcdset{markwith = {\draw[-] (-.4ex,-.4ex) -- (.4ex,.4ex);}};},
  markwith/.code = {\pgfkeysalso{/tikz/postaction = {
      /tikz/decorate,
      /tikz/decoration={markings, mark = at position 0.5 with {#1}}}
  }},
}

\usepackage{amsmath}
\usepackage{amsthm}
\usepackage{amssymb}
\usepackage{mathrsfs}
\usepackage{fixmath}
\usepackage{upgreek}
\usepackage{mathtools} 
\usepackage{stmaryrd} 

\newcounter{theoremcounter}
\makeatletter
\numberwithin{theoremcounter}{section}
\let\c@theoremcounter\c@subsection
\makeatother
\numberwithin{theoremcounter}{subsection}

\numberwithin{equation}{subsection}

\theoremstyle{plain}
\newtheorem{theorem}[theoremcounter]{Theorem}
\newtheorem*{theorem*}{Theorem}

\newtheorem{corollary}[theoremcounter]{Corollary}
\newtheorem{lemma}[theoremcounter]{Lemma}
\newtheorem{proposition}[theoremcounter]{Proposition}

\theoremstyle{remark}
\newtheorem{remark}[theoremcounter]{Remark}
\newtheorem{question}[theoremcounter]{Question}

\theoremstyle{definition}
\newtheorem{definition}[theoremcounter]{Definition}

\newtheorem{situation}[theoremcounter]{Situation}

\theoremstyle{definition}
\newtheorem{passage*}[theoremcounter]{}
\newenvironment{passage}[1][]
  {\begin{passage*}%
  \if\relax\detokenize{#1}\relax\else
    \textbf{#1.}%
  \fi}
  {\end{passage*}}
  
\newcounter{char}
\loop\ifnum\value{char}<27
  \expandafter\edef\csname\Alph{char}bb\endcsname{\noexpand\mathbb{\Alph{char}}}
  \expandafter\edef\csname\Alph{char}ca\endcsname{\noexpand\mathcal{\Alph{char}}}
  \expandafter\edef\csname\Alph{char}sc\endcsname{\noexpand\mathscr{\Alph{char}}}
  \expandafter\edef\csname\Alph{char}fr\endcsname{\noexpand\mathfrak{\Alph{char}}}
  \expandafter\edef\csname\alph{char}fr\endcsname{\noexpand\mathfrak{\alph{char}}}
  \expandafter\edef\csname\Alph{char}bf\endcsname{\noexpand\mathbf{\Alph{char}}}
  \expandafter\edef\csname\alph{char}bf\endcsname{\noexpand\mathbf{\alph{char}}}
  \expandafter\edef\csname\Alph{char}rm\endcsname{\noexpand\mathrm{\Alph{char}}}
  \expandafter\edef\csname\alph{char}rm\endcsname{\noexpand\mathrm{\alph{char}}}
  \ifnum\value{char}=1\else
    \expandafter\edef\csname\Alph{char}\Alph{char}\endcsname{\noexpand\mathbb{\Alph{char}}}
  \fi
\addtocounter{char}{1}
\repeat

\let\oldAA\AA
\renewcommand*{\AA}{\ifmmode\mathbb{A}\else\oldAA\fi}

\DeclareFontFamily{U}{wncy}{}
\DeclareFontShape{U}{wncy}{m}{n}{<->wncyr10}{}
\DeclareSymbolFont{mcy}{U}{wncy}{m}{n}
\DeclareMathSymbol{\Sha}{\mathord}{mcy}{"58}

\DeclareMathOperator{\Alb}{Alb}
\DeclareMathOperator{\AlterGrp}{A}
\DeclareMathOperator{\Aut}{Aut}
\DeclareMathOperator{\Bir}{Bir}
\DeclareMathOperator{\Bl}{Bl}

\DeclareMathOperator{\Center}{Z}

\DeclareMathOperator{\codim}{codim}

\DeclareMathOperator{\coker}{coker}

\DeclareMathOperator{\const}{const}
\DeclareMathOperator{\cor}{cor}

\DeclareMathOperator{\diag}{diag}
\DeclareMathOperator{\disc}{disc}

\DeclareMathOperator{\Div}{Div}
\DeclareMathOperator{\End}{End}
\DeclareMathOperator{\eval}{ev}
\DeclareMathOperator{\Eq}{Eq}
\DeclareMathOperator{\Ext}{Ext}

\DeclareMathOperator{\FM}{FM}

\DeclareMathOperator{\GL}{GL}

\DeclareMathOperator{\Hilb}{Hilb}
\DeclareMathOperator{\Ho}{H}

\DeclareMathOperator{\Hom}{Hom}

\DeclareMathOperator{\im}{im}

\DeclareMathOperator{\Ind}{Ind}
\DeclareMathOperator{\inflation}{inf}
\DeclareMathOperator{\Isom}{Isom}

\DeclareMathOperator{\kronecker}{\updelta}

\DeclareMathOperator{\Kummer}{Kum}

\DeclareMathOperator{\Map}{Map}
\DeclareMathOperator{\Mat}{Mat}

\DeclareMathOperator{\Mov}{Mov}
\DeclareMathOperator{\NS}{NS}

\DeclareMathOperator{\PGL}{PGL}
\DeclareMathOperator{\Pic}{Pic}

\DeclareMathOperator{\rank}{rk}
\DeclareMathOperator{\res}{res}
\DeclareMathOperator{\Res}{Res}

\DeclareMathOperator{\sgn}{sgn}

\DeclareMathOperator{\SL}{SL}

\DeclareMathOperator{\Sp}{Sp}

\DeclareMathOperator{\Stab}{Stab}
\DeclareMathOperator{\Sym}{Sym}
\newcommand*{\SymGrp}{\mathrm{S}}

\newcommand*{\isoarr}{\xrightarrow{\smash{\raisebox{-0.5ex}{\ensuremath{\scriptstyle\sim}}}}}
\newcommand*{\isodasharrow}{\overset{\smash{\raisebox{-0.5ex}{\ensuremath{\scriptstyle\sim}}}}{\dasharrow}}

\newcommand*{\notdivide}{\nmid}

\DeclarePairedDelimiter{\abs}{\lvert}{\rvert}
\newcommand*{\GrpIndex}[2]{\left({#1}:{#2}\right)}

\newcommand*{\ab}{\mathrm{ab}}

\newcommand*{\inv}{{-1}}

\newcommand*{\restr}[2]{\left.#1\right|_{#2}}

\newcommand*{\canosheaf}{\omega}
\newcommand*{\Diagonal}{\mathrm{\Delta}}

\newcommand*{\id}{\operatorname{id}}
\newcommand*{\pr}{\operatorname{pr}}

\newcommand*{\ROU}{\mu}

\newcommand*{\Category}[1]{\operatorname{\mathbf{#1}}}

\DeclareMathOperator{\CatCoh}{\Category{Coh}}
\DeclareMathOperator{\Db}{\Category{D}^{\mathrm{b}}}

\newcommand*{\kbar}{\overline{k}}

\newcommand*{\QQbar}{\overline{\QQ}}

\newcommand*{\DbG}[1]{\operatorname{\Category{D}}^{\mathrm{b}}_{#1}}
\newcommand*{\dualisog}{\widehat}
\newcommand*{\dualpol}[1]{#1^{\delta}}

\newcommand*{\forget}{\mathrm{for}}
\newcommand*{\HeckeCong}[1]{\mathrm{\Gamma}_{0}{(#1)}}
\DeclareMathOperator{\HilbChow}{HC}
\newcommand*{\SigmaMukai}{\mathrm{\Sigma}}
\newcommand*{\otimesZ}{\otimes_{\ZZ}}
\newcommand*{\shapiro}{\operatorname{sh}}
\newcommand*{\Sn}{\SymGrp_{n}}
\newcommand*{\StdRep}{\mathrm{\Gamma}}

\newcommand*{\Ulattice}{\operatorname{U}}

\usepackage{xspace}
\usepackage{xpunctuate}

\newcommand*{\eg}{e.g.\@\xspace}
\newcommand*{\ie}{i.e.\@\xspace}
\newcommand*{\cf}{cf.\@\xspace}
\newcommand*{\Cf}{See\xspace}
\newcommand*{\loccit}{loc.~cit\xperiod}

\let\oldiff\iff
\renewcommand*{\iff}{\ifmmode\oldiff\else{if and only if}\xspace\fi}
\newcommand*{\resp}{respectively\xspace}

\newcommand*{\rom}[1]{\textrm{(\romannumeral #1)}}

\usepackage{etoolbox} 
\newcounter{itemnumcounter}[theoremcounter] 
\AtBeginEnvironment{proof}{\setcounter{itemnumcounter}{0}}
\newcommand*{\itemnum}{\refstepcounter{itemnumcounter}\textnormal{(\roman{itemnumcounter})}\xspace}

\usepackage{eucal}

\hypersetup{colorlinks=true,linktocpage}
\hypersetup{citecolor=magenta,linkcolor=cyan,urlcolor=magenta}

\graphicspath{{./graphics/}}

\title[Derived equivalences of generalized Kummer varieties]{Derived equivalences of generalized Kummer varieties}

\author{Pablo Magni}
\address{}
\email{pmagni@math.uni-bielefeld.de}

\hypersetup{
  pdfauthor={Pablo Magni},
  pdftitle={Derived equivalences of generalized Kummer varieties},
  pdfcreator=pdflatex
}

\begin{document}

\begin{abstract}
  In this article we study derived (auto)equivalences of generalized Kummer varieties~$\Kummer^n(A)$.
  We provide an answer to a question raised by Namikawa by showing that the generalized Kummer varieties $\Kummer^n(A)$ and $\Kummer^n(A^\vee)$ are derived equivalent as long as $n$ is even and the abelian surface $A$ admits a polarization whose exponent is coprime to $n+1$.
  Furthermore we obtain exact sequences involving groups of autoequivalences in the style of Orlov's short exact sequence for autoequivalences of abelian varieties.
\end{abstract}

\maketitle


\section*{Introduction}

\counterwithout{theoremcounter}{section}
\preto\section{\counterwithin{theoremcounter}{section}}

In this article we are concerned with derived equivalences of certain hyperkähler varieties, the generalized Kummer varieties.\footnotemark\
\footnotetext{We work over an algebraically closed field $\Bbbk$ of characteristic zero in the introduction.
By (derived) equivalences and the notation ``$\Eq(\Db(-),\Db(-))$'' we mean in this text exact $\Bbbk$-linear equivalences of bounded derived categories of coherent sheaves on smooth projective varieties over $\Bbbk$, \ie Fourier--Mukai transforms, \cf\cite{Orl:97,Bal:09}.
}
By Bondal--Orlov~\cite{BO:01}, if two varieties $X$ and $Y$ with ample or anti-ample canonical sheaf $\canosheaf_X$ are derived equivalent, then they are already isomorphic, so it is natural to investigate the contrasting case of a trivial canonical sheaf $\canosheaf_X\simeq\Oca_X$.
By the Beauville–-Bogomolov decomposition theorem~\cite{Bog:74, Bea:83} there are three building blocks for these varieties: 1) abelian varieties, 2) strict Calabi--Yau varieties, and 3) hyperkähler varieties.

The case of derived equivalences of abelian varieties is well-understood by work of Mukai, Polishchuk, and Orlov \cite{Muk:81, Pol:96, Orl:02}.
In particular an abelian variety $A$ and its dual $A^\vee$ are derived equivalent, while usually not isomorphic, \cf\cite[Thm.~2.2]{Muk:81}.
They explain \mbox{(auto-)}equivalences between two abelian varieties $A$ and $B$ in terms of `symplectic isomorphisms', which are certain isomorphism $A\times A^{\vee}\to B\times B^{\vee}$, \cf\cref{def:symplectic-isomorphism}.
We denote the set of symplectic isomorphisms by $\Sp(A,B)$, or $\Sp(A)$ when $A=B$.
The following theorem summarizes the situation.

\begin{theorem*}[Orlov]
   We have a short exact sequence of groups
   \[
    0\to\ZZ\times A\times A^{\vee}\to \Aut(\Db(A))\to \Sp(A)\to0,
  \]
  and a surjective map
  \[
    \Eq(\Db(A),\Db(B))\twoheadrightarrow\Sp(A,B).
  \]
\end{theorem*}

In contrast, not much is known in the case of strict Calabi--Yau varieties.
By \cite{Bri:02} two strict Calabi--Yau threefolds which are birationally equivalent are also derived equivalent.
This establishes a special case of a conjecture of Bondal--Orlov~\cite{BO:95}, which predicts that K-equivalent varieties should be derived equivalent.

The hyperkähler case remains tractable yet interesting.
Much has been done in dimension~2, \ie about K3 surfaces, see \cite[Ch.~16]{Huy:K3} for an overview.
In higher dimensions, there are four known types of hyperkähler varieties up to deformations, namely Hilbert schemes of points $\Hilb^n(S)$ where $S$ is a K3 surface, generalized Kummer varieties $\Kummer^n(A)$ where $A$ is an abelian surface, and O'Grady's exceptional examples of dimension 6 and of dimension 10, \cf\cite{Bea:83,OGr:99,OGr:03}.
Ploog considers Hilbert schemes of points in \cite{Plo:07} and shows that a derived equivalence of smooth projective surfaces $S$ and $S'$ induces a derived equivalence $\Db(\Hilb^n(S))\simeq\Db(\Hilb^n(S'))$.
In particular, since an abelian surface $A$ and its dual abelian surface $A^{\vee}$ are derived equivalent, we have $\Db(\Hilb^n(A))\simeq\Db(\Hilb^n(A^{\vee}))$.
In \cite{Nam:02, Nam:02b} Namikawa raised the question whether the generalized Kummer varieties $\Kummer^n(A)$ and $\Kummer^n(A^{\vee})$ are derived equivalent.
In the case of Kummer surfaces, \ie $n=1$, they are in fact isomorphic, \cf\cite{HLOY:03}.
In this article we concern ourself with this question and reach the following answer:

\begin{theorem}[\cref{thm:main-theorem-generalized-kummer-derived-equivalence,thm:main-theorem-generalized-kummer-derived-equivalence-part2}]\label{thm:main_kum_two_derived_eq}
  Let $m\in\NN$ be even and assume there exists a polarization $\lambda\colon A\to A^\vee$ whose exponent~$\erm(\lambda)$ satisfies $\gcd(m+1,\erm(\lambda))=1$.
  Then there exist a derived equivalence
  \[
    \Db(\Kummer^m(A))\simeq\Db(\Kummer^m(A^{\vee})).
  \]
\end{theorem}

When $\End(A)=\ZZ$, the assumption reduces to $\gcd(m+1,\deg(\lambda))=1$, where $\lambda\colon A\to A^\vee$ is the polarization of minimal degree.
In \cref{rmk:main-thm-sharp} we explain that this is sharp in a certain sense.
%
Since generalized Kummer varieties associated to dual abelian varieties as in the theorem are in general not birational, the theorem provides examples of derived equivalent non-birational hyperkähler varieties.

We also study autoequivalences of generalized Kummer varieties in the style of Orlov's and Ploog's short exact sequences for abelian varieties \cite{Orl:02}, and Kummer surfaces \cite{Plo:07}, \resp.

We recall the definition of a generalized Kummer variety $\Kummer^{n-1}(A)$, where $A$ is an abelian surface.
Beauville \cite{Bea:83} defines them as the fiber of the albanese map $\Sigma\circ\mathrm{HC}$, which is an isotrivial fibration, where $\mathrm{HC}\colon\Hilb^{n}(A)\to\Sym^{n}(A)$ is the Hilbert--Chow morphism and $\Sigma\colon\Sym^{n}(A)\to A$ denotes the summation map,
and checks that they are hyperkähler varieties.
The following alternative description is more suitable for our investigations.
Consider the kernel $A\otimes\StdRep_{n}$ of the summation map $\Sigma\colon A^{n}\to A$, see~\cref{sec:standard_rep} for an explanation of the notation.
The symmetric group $\Sn$ acts by coordinate permutations on $A^n$ and trivially on $A$, so we get a diagram of fiber sequences
\[\begin{tikzcd}
  A\otimes\StdRep_{n} \arrow[r, hook] \arrow[d] & A^n \arrow[r, two heads, "\Sigma"] \arrow[d] & A \arrow[d] \\
  (A\otimes\StdRep_{n})/\Sn \arrow[r, hook] & \Sym^{n}(A) \arrow[r, two heads, "\Sigma"] & A.
\end{tikzcd}\]
Now $\Sym^{n}(A)$ is singular, but it admits a crepant resolution in form of the Hilbert--Chow morphism $\Hilb^n(A)\to\Sym^{n}(A)$.
Haiman \cite{Hai:01} provides an identification $\Hilb^n(A)\simeq\Hilb^{\Sn}(A^n)$ of the usual Hilbert scheme of points with  Nakamura's equivariant Hilbert scheme of clusters, \cf~\cite{IN:96,Rei:97}.
Similarly $(A\otimes\StdRep_{n})/\Sn$ is singular, but it admits a crepant resolution of singularities 
\[
  \Kummer^{n-1}(A)\coloneqq\Hilb^{\Sn}(A\otimes\StdRep_n)\to(A\otimes\StdRep_{n})/\Sn,
\]
which is once again the generalized Kummer variety associated to $A$.

To study the derived category of $\Kummer^{n-1}(A)$, the derived McKay correspondence \cite{BKR:01} comes into play, which in this case says that we have equivalences of derived categories
\[
  \Db(\Hilb^n(A))\simeq\Db(\Hilb^{\Sn}(A^n))\simeq\DbG{\Sn}(A^n)
\]
and
\[
  \Db(\Kummer^{n-1}(A))\simeq\Db(\Hilb^{\Sn}(A\otimes\StdRep_n))\simeq\DbG{\Sn}(A\otimes\StdRep_n)
\]
\resp, \cf\cite[Cor.~1.3]{BKR:01}; here $\DbG{\Sn}(-)$ denotes the derived category of $\Sn$-equivariant coherent sheaves.
See \cite[§1]{Plo:07} for an overview of the equivariant setup.
At this point the reader might wonder why the derived equivalence of $A$ and $A^{\vee}$ (induced by the Poincaré bundle) does not immediately yield a derived equivalence of generalized Kummer varieties $\Kummer^{n-1}(A)$ and $\Kummer^{n-1}(A^{\vee})$.
Indeed, while $(A\otimes\StdRep_n)^{\vee}$ is isomorphic as an abelian variety to $A^{\vee}\otimes\StdRep_n$, the $\Sn$-actions are different.
Instead, we have $(A\otimes\StdRep_n)^{\vee}\simeq A^{\vee}\otimes\StdRep_{n}^{\vee}$, and $\StdRep_{n}\not\simeq\StdRep_{n}^{\vee}$ are non-isomorphic $\ZZ[\Sn]$-modules for $n\geq3$.
\Cf\cite{Nam:02b} and \cite[§4.4]{Plo:Thesis} for related discussions.

We are lead to consider $\Sn$-invariant equivalences as a first step towards equivalences of $\Sn$-equivariant derived categories.
Recall that a derived equivalence $\FM_\Pca$ is called $\Sn$-invariant if its kernel~$\Pca$ satisfies $\sigma^{*}\Pca\simeq\Pca$ for every $\sigma\in\Sn$.

\begin{theorem}\label{thm:intro-main-thm-ses}
  Assume $n$ is odd, then we have a short exact sequence
  \[
    0\to \ZZ\times A[n]\to \Aut(\Db(A\otimes\StdRep_n))^{\Sn}\to \Sp(A\otimes\StdRep_n)^{\Sn}\to 0.
  \]
  If $n\neq2,4$ is even, we have an exact sequence of pointed sets
  \[
    0\to\ZZ\times A[n]\to \Aut(\Db(A\otimesZ\StdRep_n))^{\Sn} \to\Sp(A\otimes\StdRep_n)^{\Sn}\xrightarrow{\delta} A[2].
  \]
\end{theorem}

For the full statement, including sequences for $n=2,4$, see~\cref{thm:fmaut_invariants_seqs}.
We invite the reader to contrast this with the analogous sequence for Kummer surfaces.
Furthermore, we calculate groups of $\Sn$-invariant symplectic isomorphisms as follows.

%

\begin{theorem}\label{prop:main_invariant_symplectic_isos}
  \begin{enumerate}
  \item 
  For every polarization $\lambda\colon A\to A^\vee$ of exponent $e=\erm(\lambda)$ we have an inclusion
  \begin{equation}\label{eq:thm3inclusion}
     \HeckeCong{n e}\subset\Sp(A\otimes\StdRep_n)^{\Sn},
  \end{equation}
  where $\HeckeCong{n e}$ denotes the Hecke congruence subgroup of level $ne$.
  \item\label{item:thm3ii}
  If $\gcd(n,e)=1$ we have
  \[
    \Sp(A\otimes\StdRep_{n},A^{\vee}\otimes\StdRep_{n})^{\Sn}\neq\emptyset,
  \]
  and the latter becomes a right-torsor under $\Sp(A\otimes\StdRep_{n})^{\Sn}$.
  \item\label{item:thm3iii}
  If $\End(A)=\ZZ$ we can consider a polarization $\lambda_0$ of minimal degree $d=\erm(\lambda_0)^2$.
  Then the inclusion in (\ref{eq:thm3inclusion}) becomes an equality, and the condition $\gcd(n,d)=1$ in \rom{2} becomes necessary in addition to being sufficient.
  \end{enumerate}
\end{theorem}

We prove the preceding statements via a systematic study of the standard representation $\StdRep_n$ of the symmetric group $\Sn$.
In particular we calculate (in the stable range) the group cohomology $\Ho^{\bullet}(\Sn,\StdRep_n\otimesZ A)$ with arbitrary coefficients in an abelian group $A$ in terms of the group cohomology $\Ho^{\bullet}(\Sn,A)$ of the symmetric group, \cf\cref{prop:stdrep_coho_ses} and~\cref{prop:stdrep_dual_coho_vanishing_and_seq}.
For example, we get $\Ho^1(\SymGrp_n,\StdRep_{n}\otimesZ A)=0$ and $\Ho^1(\SymGrp_n,\StdRep_{n}^{\vee}\otimesZ A)=0$ when $n$ is odd and $A$ is $n$-divisible.
This usually falls into the realm of integral/modular representation theory.
For related considerations see~\cite{KS:99, Shc:04, Hem:09, CHN:10}.
With the information gained in these calculations, we can apply non-abelian group cohomology to Orlov's short exact sequence.
Taking the viewpoint that the non-abelian group cohomology group $\Ho^1(G,\Gamma)$ classifies $G$-equivariant $\Gamma$-torsors, see~\cref{sec:equivariant_torsors}, we can prove the following theorem.

\begin{theorem}[\cref{thm:invariant-equivalence}]\label{thm:main_invariant_derived_eq}
  Assume that $n$ is odd, and let $\lambda\colon A\to A^\vee$ be some polarization of exponent $\erm(\lambda)$.
  If $\gcd(n,\erm(\lambda))=1$, then
  \[
    \Eq(\Db(A\otimes\StdRep_n),\Db(A^{\vee}\otimes\StdRep_n))^{\Sn}\neq\emptyset
  \]
  and it is a right-torsor under $\Aut(\Db(A\otimes\StdRep_n))^{\Sn}$.
  If $\End(A)=\ZZ$, the converse is true when we take $\lambda$ to be the polarization of minimal degree.
\end{theorem}



To conclude our main theorem from this, we need to go from $\Sn$-invariant equivalences to equivalences of equivariant derived categories.
We accomplish this using the techniques in \cite{Plo:07} which we recall in \cref{sec:ploogs-method}.
In particular, this lets us enhance an invariant equivalence to an equivalence of equivariant derived categories as long as we have control over the obstructions in the Schur multiplier $\Ho^2(\Sn,\Bbbk^{\times})$.
Luckily, for $n=3$ we have $\Ho^2(\SymGrp_{3},\Bbbk^{\times})=0$.
To handle a general odd integer~$n$, we need to delve into the constructions of Orlov~\cite[§4]{Orl:02} and Mukai~\cite{Muk:78} in \cref{sec:equivariant-semi-homogeneous}.

\begin{remark}\label{rmk:main-thm-sharp}
  \cref{thm:main_kum_two_derived_eq} is sharp in the following sense:
  Let $G$ be a finite group acting on two varieties $X$ and $Y$.
  A derived equivalence $\DbG{G}(X)\isoarr\DbG{G}(Y)$ is called \emph{inflated} if its Fourier--Mukai kernel lies in the image of the inflation map $\inflation_{\Diagonal G}^{G\times G}$.
  Assume $n\geq3$ and $\End(A)=\ZZ$, and let $d$ denote the minimal degree of a polarization of $A$.
  Then $\Kummer^{n-1}(A)$ and $\Kummer^{n-1}(A^{\vee})$ cannot be derived equivalent via an inflated equivalence unless $\gcd(n,d)=1$.
  This is because the relevant set of invariant symplectic isomorphisms in \cref{prop:main_invariant_symplectic_isos}.\ref{item:thm3ii} needs to be non-empty, which by \cref{prop:main_invariant_symplectic_isos}.\ref{item:thm3iii} is equivalent to $\gcd(n,d)=1$.
\end{remark}

Furthermore, Frei--Honigs \cite[Cor.~1.2]{FH:22} announced an example of a  generalized Kummer fourfold $\Kummer^2(A)$ over a number field $K$ which cannot be derived equivalent to $\Kummer^2(A^{\vee})$ over the field $K$.
The abelian variety in this example cannot carry a polarization defined over $K$ whose degree is coprime to $3$, since \cite{FH:22} requires that the Galois modules $A[3]$ and $A^{\vee}[3]$ are not isomorphic.
This gives further information towards the following question.

\begin{question}
  Is it necessary for $\Kummer^m(A)$ and $\Kummer^m(A^{\vee})$ to be derived equivalent that there exists an isogeny $\lambda\colon A\to A^{\vee}$ of exponent coprime to $m+1$?
\end{question}

%
%
%

In the last section \cref{sec:non-birat}, we exhibit cases of generalized Kummer fourfolds which are not birationally equivalent, extending a result of \cite{Nam:02}, but which are nonetheless derived equivalent according to our results.

\begin{theorem}[\cref{thm:gen-kummer-non-birat,thm:auto-group-of-gen-kummer}]
  Let $A$ be an abelian surface over $\CC$ satisfying $\End(A)=\ZZ$, and let $\deg(\lambda)=d^2$ be the degree of the polarization~$\lambda\colon A\to A^\vee$ of minimal degree.
  Assume that $d\neq1$ and either
  \begin{enumerate}
    \item
    $d/3$ is a perfect square, in the case that $3$ divides $d$, or
    \item
    the Pell equation $x^2-3d y^2=1$ has some solution with odd~$y$, in the case that $3$ does not divide $d$.
  \end{enumerate}
  Then the generalized Kummer fourfolds $\Kummer^2(A)$ and $\Kummer^2(A^\vee)$ are \emph{not} birationally equivalent.
  Furthermore we have isomorphisms \[\Bir(\Kummer^2(A))\simeq\Aut(\Kummer^2(A))\simeq A[3]\rtimes\Aut_{\mathrm{AV}}(A),\] where $\Bir(\Kummer^2(A))$ is the group of birational autoequivalences.
\end{theorem}

\begin{acknowledgements}
  I want to genuinely thank Lenny Taelman for suggesting the questions that sparked the research in this article and for offering his time in many discussions related to it.
  I also want to thank Lie Fu for discussions about hyperkähler varieties and related topics and his valuable feedback on drafts of this article.
  During the research presented here, the computer algebra system GAP~\cite{GAP4} was very helpful for experimentation with the computations in this work, hence we would like to thank the authors and maintainers of GAP for making the system available.
  
  Research supported by the Netherlands Organization for Scientific Research (NWO) under project number 613.001.752.
  Research funded by the Deutsche Forschungsgemeinschaft (DFG, German Research Foundation) – Project-ID 491392403 – TRR 358.
\end{acknowledgements}


\section{The integral standard representation of $\Sn$}
\label{sec:standard_rep}

In this section we discuss the standard representation of the symmetric group with integral coefficients, as well as its dual representation.
This serves as preliminaries for calculations in later sections.

\begin{definition}
  Define the abelian group $\StdRep_n$ as the kernel of the summation map $\Sigma\colon\ZZ^n\to\ZZ$.
  Now the symmetric group $\Sn$ acts on $\ZZ^n$ by permuting the factors;
  explicitly $\sigma.\erm_i\coloneqq\erm_{\sigma(i)}$, where $\erm_i$ denotes the $i$-th standard basis vector.
  Since the morphism $\Sigma$ is equivariant when we endow $\ZZ$ with the trivial action, this induces an action of $\Sn$ on $\StdRep_n$.
  The $\ZZ[\Sn]$-module $\StdRep_n$ is called the \emph{standard representation} of $\Sn$.
\end{definition}

\begin{passage}
  Note that, there is an isomorphism of abelian groups $\ZZ^{n-1}\isoarr\StdRep_n$ given by
  \[
    (a_1,\dots,a_{n-1})\mapsto(a_1,\dots,a_{n-1},-a_1-\dots -a_{n-1}).
  \]
  So, when $A$ is some abelian group, we get a short exact sequence
  \begin{equation}\label{ses:abvar_stdrep_def}
    0\to A\otimesZ\StdRep_n\to A\otimesZ\ZZ^n\xrightarrow{\Sigma}A\otimesZ\ZZ\to0.
  \end{equation}
  This identifies $A\otimesZ\StdRep_n$ with the kernel of the morphism $\Sigma\colon A^n\to A$, when we identify $A\otimesZ\ZZ^n$ with $A^n$.
  Note that the latter identification is clearly $\Sn$-equivariant, so we get an isomorphism of $\ZZ[\Sn]$-modules.
\end{passage}


%

\begin{definition}\label{par:stdrep_dual_definition}
  Define the abelian group ${\StdRep}^{\vee}_n\coloneqq\coker(\Delta\colon\ZZ\to\ZZ^n)$, where $\Delta$ is the diagonal map.
  Then, similar to above, ${\StdRep}^{\vee}_n$ becomes a $\ZZ[\Sn]$-module.  
\end{definition}

\begin{passage}\label{sec:diag-is-dual-to-sum}
  The notation ${\StdRep}^{\vee}_n$ is justified since $\Delta\colon\ZZ\to\ZZ^n$ is the dual homomorphism to $\Sigma\colon\ZZ^n\to\ZZ$, under the identifications $\Hom_{\ZZ}(\ZZ,\ZZ)\isoarr\ZZ$, $f\mapsto f(1)$, and $\Hom_{\ZZ}(\ZZ^n,\ZZ)\isoarr\ZZ^n$, $f\mapsto (f(\erm_i))_i$.
  So we can identify $\coker(\Delta\colon\ZZ\to\ZZ^n)$ with the dual abelian group $\Hom_{\ZZ}(\StdRep_n,\ZZ)$.
  That is, we have an isomorphism of short exact sequences
  \[\begin{tikzcd}[rotate90south/.style={anchor=south, rotate=270}]
    0 \arrow[r] & \ZZ^{\vee} \arrow[r, "\Sigma^{\vee}"] \arrow[d, "\sim" rotate90south] & (\ZZ^{n})^{\vee} \arrow[r] \arrow[d, "\sim" rotate90south] & \Hom_{\ZZ}(\StdRep_n,\ZZ) \arrow[r] \arrow[d, dotted, "\sim" rotate90south] & 0 \\
    0 \arrow[r] & \ZZ \arrow[r, "\Delta"] & \ZZ^n \arrow[r] & {\StdRep}^{\vee}_{n} \arrow[r] & 0
  \end{tikzcd}\]
  Again, this isomorphism is $\Sn$-equivariant.
  Indeed, when $\Sn$ acts on a abelian group~$A$, then $\Sn$ acts on $\Hom_{\ZZ}(A,\ZZ)$ via $\sigma.f=x\mapsto f(\sigma^{-1}.x)$, where we have endowed $\ZZ$ with the trivial $\Sn$-action.
  Now the equalities
  \[
    \sigma.\erm_i \coloneqq ((\sigma.\erm_i^{\vee})(\erm_j))_j=(\erm_i^{\vee}(\sigma^{-1}.\erm_j))_j=(\erm_i^{\vee}(\erm_{\sigma^{-1}(j)}))_j=(\kronecker_{i,\sigma^{-1}(j)})_j=(\kronecker_{\sigma(i),j})_j=\erm_{\sigma(i)}
  \]
  show that the identifications above are equivariant.
\end{passage}

\begin{definition}
  We call the composition $\phi_0\colon\StdRep_n\hookrightarrow\ZZ^n\twoheadrightarrow\StdRep^{\vee}_n$ the \emph{canonical map}.
\end{definition}

\begin{proposition}\label{par:stdrep_canonical_map}
  The canonical map $\phi_0$ induces a short exact sequence of $\ZZ[\Sn]$-modules
  \begin{equation}\label{ses:stdrep_canonical_map}
    0\to\StdRep_n\to\StdRep^{\vee}_n\to\ZZ/n\ZZ\to0,
  \end{equation}
  where the action of $\Sn$ on $\ZZ/n\ZZ$ is trivial.
\end{proposition}

\begin{proof}
  We have $\Delta(k)\in\StdRep_n$ for some $k\in\ZZ$ if and only if $n k=0$, so the canonical map~$\phi_0$ is injective.
  An element $[(k_1,\dots,k_n)]\in\StdRep^{\vee}_n$ is in the image of $\phi_0$ if and only if $(k_1,\dots,k_n)+\Delta(a)\in\StdRep_n$ for some $a\in\ZZ$.
  This is equivalent to $\sum{k_i}=-n a$, \ie to $n$ divides $\sum{k_i}$.
  So $\phi_0\colon\StdRep_n\hookrightarrow\StdRep^{\vee}_n$ is an injection of index $n$; the quotient is generated by $(k,0,\dots,0)$ for $k=0,\dots,n-1$, and we see that Seq.~\eqref{ses:stdrep_canonical_map} is an exact sequence.
  
  Finally, the action of $\Sn$ on $\ZZ/n\ZZ$ is trivial:
  Consider a transposition $\tau\in\Sn$ and $k\in\ZZ$, then
  \begin{multline*}
    \tau.(k,0,\dots,0)=(0,\dots,k,\dots,0) \\
    \equiv(0,\dots,k,\dots,0)+(k,0,\dots,-k,\dots,0)=(k,0,\dots,0).
    \qedhere
  \end{multline*}
\end{proof}

\begin{proposition}\label{prop:stdrep_to_stdrep_equivariant_maps}
  We have
  \begin{enumerate}
    \item $\Hom_{\ZZ[\Sn]}(\StdRep_n,\StdRep_n)=\ZZ\cdot\id$,
    \item $\Hom_{\ZZ[\Sn]}(\StdRep_{n}^\vee,\StdRep_{n}^\vee)=\ZZ\cdot\id$, and
    \item $\Hom_{\ZZ[\Sn]}(\StdRep_{n},\StdRep_{n}^\vee)=\ZZ\cdot\phi_0$ for $n\geq3$.
  \end{enumerate}
  In particular, $\StdRep_{n}$ and $\StdRep_{n}^\vee$ are not isomorphic as $\ZZ[\Sn]$-modules for $n\geq3$, since $\phi_0$ is not surjective.
\end{proposition}

\begin{proof}
  \itemnum and \itemnum:
  Extending scalars to $\QQbar$, both $\StdRep_{n}\otimes\QQbar$ and $\StdRep_{n}^{\vee}\otimes\QQbar$ become the usual standard representation of $\Sn$, which is irreducible.
  So by Schur's lemma, every $\Sn$-equivariant morphism $\phi\colon\StdRep_{n}\to\StdRep_{n}$ (\resp $\phi\colon\StdRep_{n}^{\vee}\to\StdRep_{n}^{\vee}$) becomes of the form $\phi_{\QQbar}=\lambda\cdot\id$ for some $\lambda\in\QQbar$.
  But since $\phi$ is defined integralley, we must have $\lambda\in\ZZ$ actually.
  
  \itemnum
  Using Seq.~\eqref{ses:stdrep_canonical_map}, the canonical map $\phi_0$ becomes an isomorphism after extending scalars to $\QQbar$.
  So, as above, every $\Sn$-equivariant morphism $\phi\colon\StdRep_{n}\to\StdRep_{n}^{\vee}$ becomes of the form $\phi_{\QQbar}=\lambda\cdot\phi_{0,\QQbar}$ for some $\lambda\in\QQbar$. 
  By looking at the short exact sequence~Seq.~\eqref{ses:stdrep_canonical_map} and the elementary divisors normal form of $\phi_0$, we see that $\phi_0$ corresponds to $\diag(1,\dots,1,n)$ in suitable bases of $\StdRep_n$ and $\StdRep_{n}^{\vee}$.
  For $n\geq3$ this forces $\lambda\in\ZZ$, as desired.
\end{proof}

\begin{remark}\label{rmk:canonical-map-n-two}
  For $n=2$, the map $\phi_0\in\Hom(\StdRep_2,\StdRep_{2}^\vee)$ is not a generator.
  In fact both $\StdRep_2$ and $\StdRep_{2}^\vee$ are isomorphic to the sign representation.
  Under these identifications the map~$\phi_0$ becomes $2\cdot\id$.
\end{remark}

\begin{passage}[The dual isogeny $\dualisog{\phi}_0$]
  We construct a `dual isogeny' $\dualisog{\phi}_0\colon\StdRep_n^{\vee}\to\StdRep_n$ to the canonical map $\phi_0\colon\StdRep_n\to\StdRep_n^\vee$.
  Consider the diagram
  \[\begin{tikzcd}
    \StdRep_n^{\vee} \arrow[r, "\phi_0^{-1}"] \arrow[rd, dotted] & \StdRep_n\otimesZ\QQ \arrow[r, "\cdot n"] & \StdRep_n\otimesZ\QQ \\
    & \frac{1}{n}\StdRep_n \arrow[u, hook'] \arrow[r, "\cdot n"] & \StdRep_n \arrow[u, hook'].
  \end{tikzcd}\]
  We claim that $\im(\phi_0^{-1})\subset\frac{1}{n}\StdRep_n$.
  Indeed, note that by Seq.~\eqref{ses:stdrep_canonical_map}, we have $n\StdRep_n^{\vee}\subset\im(\phi_0)\subset\StdRep_n^{\vee}$.
  So for every $y\in\StdRep_n^\vee$, there exists $x\in\StdRep_n$ such that $n\cdot y=\phi_0(x)$, \ie $y=\phi_0(\frac{1}{n}x)$.
  Now we define
  \[
    \dualisog{\phi}_0\coloneqq n\cdot\phi_0^{-1}\colon\StdRep_n^{\vee}\to\StdRep_n,
  \]
  and see that
  \[
    \dualisog{\phi}_0\circ\phi_0 = n\cdot\id = \phi_0\circ\dualisog{\phi}_0.
  \]
\end{passage}

\begin{remark}\label{rmk:dualisog-explicit-form}
  The map $\dualisog{\phi}_0\colon\StdRep_{n}^{\vee}\to\StdRep_{n}$ is explicitly given by $[(a_i)_i]\mapsto(n\cdot a_i)_i$ if $\sum_{i}{a_i}=0$, and sends $[(1,0,\dots,0)]$ to $(n-1,-1,\dots,-1)$.
\end{remark}

\begin{proposition}\label{prop:stdrep_dual_to_stdrep_equivariant_maps}
  For $n\geq2$ we have $\Hom_{\ZZ[\Sn]}(\StdRep_n^{\vee},\StdRep_n)=\ZZ\cdot\dualisog{\phi}_0$.
\end{proposition}

\begin{proof}
  Similar to the proof of \cref{prop:stdrep_to_stdrep_equivariant_maps}, we see that for every $\phi\in\Hom_{\ZZ[\Sn]}(\StdRep_n^{\vee},\StdRep_n)$ there exists some $\lambda\in\QQbar$ such that $\phi_{\QQbar}=\lambda\cdot\dualisog{\phi}_{0,\QQbar}$.
  Recall that in suitable integral bases, the map $\phi_0$ corresponds to the matrix $\diag(1,\dots,1,n)$.
  So, $\dualisog{\phi}_0$ corresponds to the matrix $\diag(n,\dots,n,1)$, which forces $\lambda\in\ZZ$ as desired.
\end{proof}

\begin{proposition}\label{prop:phi_and_phi_hat_symmetric}
  The canonical map $\phi_0$ and its dual isogeny $\dualisog{\phi}_0$ are both symmetric, \ie we have the identities
  \begin{enumerate}
    \item $\phi_0=\phi_0^{\vee}\circ\eval$, and
    \item $\dualisog{\phi}_0 = \eval^{-1}\circ(\dualisog{\phi}_0)^{\vee}$,
  \end{enumerate}
  where $\eval\colon \StdRep_n\to \StdRep_n^{\vee\vee}$, $x\mapsto(\psi\mapsto\psi(x))$ denotes the canonical ``evaluation'' isomorphism.
\end{proposition}

\begin{proof}
  \itemnum
  For $x,y\in\StdRep_n$ we have $\phi_0^{\vee}(\eval(x))(y)=\eval(x)(\phi_0(y))=\phi_0(y)(x)$.
  We want to compare this with $\phi_0(x)(y)$.
  Take $x=\tilde{\erm}_i$ and $y=\tilde{\erm}_j$, where $\tilde{\erm}_{i}\coloneqq\erm_{i}-\erm_{n}$ is the $i$-th standard basis vactor of $\StdRep_n$.
  Then $\phi_0(\tilde{\erm}_i)=\erm_{i}^{\vee}-\erm_{n}^{\vee}$, so $\phi_0(\tilde{\erm}_i)(\tilde{\erm}_j)=1$ if $i\neq j$ and $\phi_0(\tilde{\erm}_i)(\tilde{\erm}_j)=2$ if $i=j$.
  This is evidently independent of the order of $i$ and $j$.

  \itemnum
  Since $\StdRep_n$ is torsion-free, we can check the claimed equality after extending scalars to $\QQ$.
  By \rom{1} we know that $\phi_0=\phi_0^{\vee}\circ\eval$, so $\phi_{0}^{\inv}=\eval^{\inv}\circ(\phi_{0}^{\vee})^{\inv}=\eval^{\inv}\circ(\phi_{0}^{\inv})^{\vee}$, and hence $n\cdot\phi_{0}^{\inv} = \eval^{\inv}\circ(n\cdot\phi_{0}^{\inv})^{\vee}$, as desired.
\end{proof}



\section{Group cohomology of the standard representation}
\label{sec:group_coho_calculations}

In this section we calculate (in the stable range) the group cohomology $\Ho^{\bullet}(\Sn,\StdRep_n\otimesZ A)$ with arbitrary coefficients in an abelian group $A$ in terms of the group cohomology $\Ho^{\bullet}(\Sn,A)$ of the symmetric group.
While doing so, we recall some general facts about group cohomology of finite groups that are directly relevant to the calculation; for a more detailed reference see~\cite{NSW:CNF} or \cite{Bro:82}, or for a quick overview see \cite[§I.2]{Ser:GC}.

\begin{proposition}\label{prop:stdrep_h0}
  Let $n\geq2$, and let $A$ be an abelian group, then $\Ho^0(\Sn,\StdRep_n\otimes_{\ZZ}A)=A[n]$.
  Here we denote by $A[n]$ the group of $n$-torsion elements of $A$.
\end{proposition}

\begin{proof}
  Recall that $\Ho^0(\Sn,\StdRep_n\otimes_{\ZZ}A)=(\StdRep_n\otimes_{\ZZ}A)^{\Sn}$ is the submodule of $\Sn$-fixed-points in $\StdRep_n\otimes_{\ZZ}A$.
  Note that $(a_1,\dots,a_n)\in A^n$ is fixed by $\Sn$ \iff $a_1=\dots=a_n\eqqcolon a$.
  But $(a,\dots,a)\in\StdRep_n\otimes_{\ZZ}A$ \iff $n\cdot a=0$.
\end{proof}

\begin{proposition}
  Let $n\geq3$, and let $A$ be an abelian group, then we have $\Ho^0(\Sn,\StdRep_{n}^{\vee}\otimesZ A)=0$.
  For $n=2$, we have $\Ho^0(\SymGrp_{2},\StdRep_{2}^{\vee}\otimesZ A)=A[2]$.
\end{proposition}

\begin{proof}
  Recall that $\StdRep_{n}^{\vee}\otimesZ A$ is the cokernel of the diagonal map $\Delta\colon A\to A^n$.
  Consider a fixed-point $[(a_i)_i]\in(\StdRep_{n}^{\vee}\otimesZ A)^{\Sn}$.
  Let $\tau\in\Sn$ be some transposition, say $\tau=(1\,2)$ to keep the notation simple.
  Then being fixed by $\tau$ means that we have $(a_1,a_2,\dots,a_n)=(a_2,a_1,\dots,a_n)+\Delta(a)$ for some $a\in A$.
  For $n\geq3$ this entails $a_3=a_{3}+a$, which implies $a=0$.
  So the representative $(a_i)_i$ itself is already fixed by $\tau$.
  In conclusion, we see that $(a_i)_i=\Delta(a_1)$, which is zero in $\StdRep_{n}^{\vee}\otimesZ A$.
  
  For $n=2$ we can consider $(a_1,0)$ without loss of generality.
  Being fixed by $\tau$ means exactly $(a_1,0)=(0,a_1)+(a,a)$ for some $a\in A$, or equivalently $2 a_{1}=0$ and $a=a_1$.
\end{proof}

\begin{remark}
  The preceding proposition also follows from \cref{prop:stdrep_dual_coho_vanishing_and_seq} below.
\end{remark}

\begin{passage}[Restricted representations]
  Let $G$ be a finite group with some subgroup $H\subset G$, and let $N$ be a $\ZZ[G]$-module.
  Define the restricted $\ZZ[H]$-module as $\Res_{H}^{G}(N)\coloneqq N$ where $H$ acts through the action of $G$ on $N$.
  This construction becomes a functor by sending a homomorphism of $\ZZ[G]$-modules to itself.
\end{passage}

\begin{passage}[Induced representations]
  Let $G$ be a finite group with some subgroup $H\subset G$, and let $M$ be a $\ZZ[H]$-module.
  The (co)induced $\ZZ[G]$-module is defined as
  \begin{align*}
    \Ind_{H}^{G}(M)
    &\coloneqq\Hom_{\ZZ[H]}(\ZZ[G],M) \\
    &\simeq\{\phi\colon G\to M\mid \phi(h g)=h.\phi(g)\quad\text{for}\ g\in G, h\in H\}
  \end{align*}
  with $G$-action afforded by $g.\phi\coloneqq(g'\mapsto\phi(g' g))$.
  Define the canonical projection $ \pi\colon\Ind_{H}^{G}(M) \to M$ by mapping $\phi \mapsto \phi(1)$.
  It is a morphism of $\ZZ[H]$-modules.  
  
  This construction provides a right adjoint to the restriction functor $\Res_{H}^{G}$.
  In fact it is also a left adjoint to the restriction functor.
  For this reason we won't make a distinction in our terminology regarding ``induced'' and ``coinduced''.
  (The former would be strictly speaking defined as $\ZZ[G]\otimes_{\ZZ[H]}M$.)
\end{passage}

The following perspective on induced representations is useful when considering a $\ZZ[H]$-module with trivial action, and is applied in the proof of \cref{prop:permutation_rep_as_induced_rep_identifications} below.

\begin{passage}
  Let $g_1,\dots,g_n$ be a set of right coset representatives for $H\backslash G$.
  Then we have an isomorphism
  \[
    \Ind_H^G(M)\simeq\Map(H\backslash G,M),
  \]
  sending $\phi$ to the map of sets $H g_i\mapsto\phi(g_i)$.
  Its inverse is given by extending such a map of sets via the formula $\phi(hg)=h.\phi(g)$.
  The left $G$-action on $\Map(H\backslash G,M)$ becomes
  \[
    (g.\phi)(H g_i) \coloneqq h.\phi(H g_j),
  \]
  when $g_i g=h g_j$ for $g\in G$ and $h\in H$.
  
  Assume now that $M$ arises as the restriction of a $\ZZ[G]$-module with trivial action.
  Note that then the isomorphism above is independent of the choice of coset representatives~$\{g_i\}$.
  The $G$-action on $\Map(H\backslash G,M)$ above becomes just the usual action on the set of maps between a right and a left $G$-set.
  Concretely, for $\phi\in\Map(H\backslash G,M)$ and $g\in G$ this is given by
  \begin{align*}
    g.\phi
    &= x\mapsto g.\phi(x.g) \\
    &= x\mapsto \phi(x.g).
  \end{align*}
  
\end{passage}

\begin{proposition}[Shapiro's Lemma {\cite[Prop.~1.6.4]{NSW:CNF}}]\label{prop:shapiros_lemma}
  Let $G$ be a finite group with some subgroup $H\subset G$, and let $M$ be a $\ZZ[H]$-module.
  Then, for every $n\geq0$, the canonical projection~$\pi$ induces an isomorphism
  \[
    \shapiro_n\colon\Ho^n(G,\Ind_{H}^{G}(M))\isoarr\Ho^n(H,M).
  \]
\end{proposition}


\begin{passage}
  Let $G$ be a finite group with some subgroup $H\subset G$, and let $N$ be a $\ZZ[G]$-module.
  Define the following two homomorphisms of $\ZZ[G]$-modules:
  \begin{alignat}{2}
    \iota &\colon N \to \Ind_{H}^{G}(\Res_{H}^{G}(N))
    \quad\quad
    && x \mapsto(g\mapsto g.x) \\
    \nu &\colon\Ind_{H}^{G}(\Res_{H}^{G}(N)) \to N
    && \phi \mapsto \sum_{[g]\in G/H}{g.\phi(g^{-1})}
  \end{alignat}
  The significance of these maps is that $\iota$ is the unit of the restriction-coinduction adjunction, and $\nu$ is the counit of the induction-restriction adjunction.
\end{passage}

\begin{proposition}\label{prop:res_and_cor_via_shapiro}
  Let $G$ be a finite group with some subgroup $H\subset G$, and let $N$ be a $\ZZ[G]$-module.
  For $n\geq0$ we have commutative diagrams
  \[\begin{tikzcd}[trim left=(a),trim right=(a)]
    \Ho^n(G,N) \arrow[r, "\iota_{*}"] \arrow[rr, "\res_{H}^{G}", bend right=15] &|[alias=a]| \Ho^n(G,\Ind_{H}^{G}(\Res_{H}^{G}(N))) \arrow[r, "\shapiro"] & \Ho^n(H,\Res_{H}^{G}(N)),
  \end{tikzcd}\]
  and
  \[\begin{tikzcd}[trim left=-158]
    \Ho^n(H,\Res_{H}^{G}(N)) \arrow[r, "\shapiro^{-1}"] \arrow[rr, "\cor_{H}^{G}", bend right=15] & \Ho^n(G,\Ind_{H}^{G}(\Res_{H}^{G}(N))) \arrow[r, "\nu_{*}"] & \Ho^n(G,N).
  \end{tikzcd}\]
\end{proposition}

\begin{proof}
  See \cite[Prop.~1.6.5]{NSW:CNF} for a concrete proof.
  The highbrow explanation is that $\res_{H}^{G}$ and $\shapiro\circ\iota_{*}$, as well as $\nu_{*}$ and $\cor_{H}^{G}\circ\shapiro$ are morphisms of universal $\delta$-functors.
  So it is enough to check the case $n=0$, which is straightforward once one recalls the definition of corestriction in degree~$0$ as a norm map.
%
%
\end{proof}

\begin{proposition}\label{prop:permutation_rep_as_induced_rep_identifications}
  Let $A$ be an abelian group endowed with the trivial $\SymGrp_{n-1}$-action, and let $A^n$ be the permutation $\Sn$-representation.
  Then we have an isomorphism of $\ZZ[\Sn]$-modules
  \[
    \Ind_{\SymGrp_{n-1}}^{\Sn}(A)\simeq A^n,
  \]
  such that, under this isomorphism,
  \begin{enumerate}
    \item the canonical map $\pi$ corresponds to the the projection $\erm_n^\vee$ onto the $n$-th coordinate,
    \item the morphism $\iota$ corresponds to the diagonal map $\Delta\colon A\to A^n$.
    \item the morphism $\nu$ corresponds to the the sum map $\Sigma\colon A^n\to A$, and
  \end{enumerate}
\end{proposition}

\begin{proof}
  Recall that since the action on $A$ is trivial, we have $\Ind_{\SymGrp_{n-1}}^{\Sn}(A)\simeq\Map(\SymGrp_{n-1}\backslash\Sn,A)$ with $\Sn$-action given by $\sigma.\phi = x\mapsto \phi(x\circ\sigma)$.
  Now $\Sn$ acts transitively on the set $\{1,\dots,n\}$ from the right via $k.\sigma\coloneqq\sigma^\inv(k)$.
  Then $\SymGrp_{n-1}=\Stab(n)$ is the stabilizer of the element $n$, so we get an isomorphism of right $\Sn$-sets
  \[
    \SymGrp_{n-1}\backslash\Sn\isoarr\{1,\dots,n\},\quad
    \SymGrp_{n-1}\sigma\mapsto\sigma^\inv(n).
  \]
  Identifying $\Map(\{1,\dots,n\},A)$ with $A^n$, an $n$-tuple $(a_1,\dots,a_n)\in A^n$ gets acted on as
  \[
    \sigma.(a_i)_i = (a_{i.\sigma})_i = (a_{\sigma^\inv(i)})_i,
  \]
  exactly as in the permutation representation.
  
  \itemnum
  Let us now compute the maps $\pi$, $\nu$, and $\iota$ under the identifications above.
  By definition $\pi(\phi)=\phi(\id)$, for $\phi\in\Ind_{\SymGrp_{n-1}}^{\Sn}(A)$.
  Now under the identification $\SymGrp_{n-1}\backslash\Sn\simeq\{1,\dots,n\}$, the element $\id$ corresponds to the element $n$.
  Denote by $\overline{\phi}\in\Map(\{1,\dots,n\},A)$ the element induced by $\phi$.
  Then we have $\pi(\overline{\phi})=\overline{\phi}(n)$, which means that $\pi$ becomes the projection onto the $n$-th coordinate.
  
  \itemnum
  Consider the map $\iota\colon A\to \Ind_{\SymGrp_{n-1}}^{\Sn}(A)$.
  By definition we have $\iota(a)=(\sigma\mapsto \sigma.a)\eqqcolon\phi$.
  Since the action on $A$ is trivial, this is just the constant map $\const_a$ with value $a$.
  But this implies that also $\overline{\phi}=\const_a$, which corresponds to the $n$-tuple $(a,\dots,a)\in A^n$, as desired.
  
  \itemnum
  Consider the map $\nu\colon\Ind_{\SymGrp_{n-1}}^{\Sn}(A)\to A$, which is given by
  \[
    \phi
    \mapsto\sum_{[\sigma]\in \Sn/\SymGrp_{n-1}}{\sigma.\phi(\sigma^\inv)}
    =\sum_{[\sigma]\in \Sn/\SymGrp_{n-1}}{\phi(\sigma^\inv)}.
  \]
  If $\{\sigma_1,\dots,\sigma_n\}$ is a set of left-coset representatives, then $\{\sigma_1^\inv,\dots,\sigma_n^\inv\}$ is a set of right-coset representatives.
  So an element $(a_i)_i\in A^n$ corresponds to $\phi$ satisfying $\phi(\sigma_i^\inv)=a_i$.
  Now $\nu(\phi)=\sum_{i=1}^{n}{\phi(\sigma_i^\inv)}=\sum_{i=1}^{n}{a_i}$, which witnesses that $\nu$ becomes the summation map~$\Sigma$.
\end{proof}

\begin{proposition}[{\cite[Prop.~1.5.7]{NSW:CNF}}]\label{prop:cor_res_composition}
  Let $G$ be a finite group with some subgroup $H\subset G$.
  Then we have for every $\ZZ[G]$-module $N$ and $n\geq0$ the identity of endomorphism of $\Ho^n(G,N)$
  \[
    \cor_{H}^{G}\circ\res_{H}^{G}=\GrpIndex{G}{H}\cdot\id.
  \]
\end{proposition}

\begin{proposition}[Nakaoka~{\cite[Cor.~6.7]{Nak:60}}]\label{prop:nakaoka_res_isomorphism}
  Let $A$ be an abelian group endowed with trivial $\Sn$-action.
  Then the restriction map
  \[
    \res_{\SymGrp_{n-1}}^{\Sn}\colon\Ho^k(\Sn,A)\to\Ho^k(\SymGrp_{n-1},\Res_{\SymGrp_{n-1}}^{\Sn}(A))
  \]
  is an isomorphism for $k<n/2$.
\end{proposition}

\begin{proposition}\label{prop:stdrep_coho_ses}
  Let $A$ be an abelian group endowed with trivial $\Sn$-action.
  Then we have a short exact sequence of group cohomology groups
  \[
    0\to\Ho^{k-1}(\Sn,A)/n\Ho^{k-1}(\Sn,A)\to\Ho^k(\Sn,\StdRep_n\otimesZ A)\to\Ho^k(\Sn,\ZZ^n\otimesZ A)[n]\to0
  \]
  for $k<n/2$.
  When $k=n/2$, we still have the injection on the left side.
\end{proposition}


\begin{proof}
  Apply group group cohomology to the defining short exact sequence~Seq.~\eqref{ses:abvar_stdrep_def} of $\StdRep_n\otimesZ A$ to get the exact sequence
  \[
    \cdots\to\Ho^i(\Sn,\StdRep_n\otimesZ A)\to\Ho^i(\Sn,\ZZ^n\otimesZ A)\xrightarrow{\Sigma_{*}^i}\Ho^i(\Sn,\ZZ\otimesZ A)\to\cdots
  \]
  Using the identification from \cref{prop:permutation_rep_as_induced_rep_identifications}, Shapiro's isomorphism (\cref{prop:shapiros_lemma}) and \cref{prop:res_and_cor_via_shapiro}, we extend this to a commutative diagram
  \[\begin{tikzcd}[rotate90south/.style={anchor=south, rotate=270}, column sep=1em]
    \cdots \arrow[r] & \Ho^i(\Sn,\StdRep_n\otimesZ A) \arrow[r] & \Ho^i(\Sn,\ZZ^n\otimesZ A) \arrow[r, "\Sigma_{*}^i"] \arrow[d, "\shapiro"', "\sim" rotate90south] & \Ho^i(\Sn,\ZZ\otimesZ A) \arrow[r] & \cdots \\
    & & \Ho^i(\SymGrp_{n-1},\ZZ\otimesZ A). \arrow[ru, "\cor_{\SymGrp_{n-1}}^{\Sn}"'] & &
  \end{tikzcd}\]
  
  Now using the formula $\cor_{\SymGrp_{n-1}}^{\Sn}\circ\res_{\SymGrp_{n-1}}^{\Sn}=\GrpIndex{\Sn}{\SymGrp_{n-1}}=n$, \cf~\cref{prop:cor_res_composition}, and that $\res_{\SymGrp_{n-1}}^{\Sn}$ is an isomorphism for $i<n/2$, \cf~\cref{prop:nakaoka_res_isomorphism}, we get for $i=k$
  \[
    \ker(\cor_{\SymGrp_{n-1}}^{\Sn}\colon\Ho^k(\SymGrp_{n-1},A)\to\Ho^k(\Sn,A))=\Ho^k(\SymGrp_{n-1},A)[n],
  \]
  which determines $\ker(\Sigma_{*}^k)=\Ho^k(\Sn,\ZZ^n\otimesZ A)[n]$.
  For $i=k-1$ we get
  \[
    \im(\Sigma_{*}^{k-1})=\im(\cor_{\SymGrp_{n-1}}^{\Sn}\colon\Ho^{k-1}(\SymGrp_{n-1},A)\to\Ho^{k-1}(\Sn,A))=n\cdot\Ho^{k-1}(\Sn,A).
  \]
  So the long exact sequence above induces the desired short exact sequence.
\end{proof}

%

\begin{corollary}\label{cor:stdrep_coho_h1_ses}
  Let $A$ be an abelian group, and assume $n\geq3$, then we have a short exact sequence
  \[
    0\to A/n A\to\Ho^1(\Sn,\StdRep_n\otimesZ A)\to A[2][n]\to0,
  \]
  where $A[2][n]$ denotes the $n$-torsion subgroup of the $2$-torsion subgroup of $A$.
\end{corollary}

\begin{proof}
  Since the action on $A$ is trivial, we have $\Ho^0(\Sn,A)=A$.
  Taking the viewpoint that elements of the first cohomology group can be represented by crossed-homomorphisms, which are just homomorphisms when $A$ carries the trivial action, we see that
  \[
    \Ho^1(\SymGrp_{n-1},A)
    \simeq\Hom((\SymGrp_{n-1})^\ab,A)
    \simeq\Hom(\ZZ/2\ZZ,A)
    \simeq A[2].
  \]
  Now we use Shapiro's isomorphism $\Ho^1(\Sn,A^n)\isoarr\Ho^1(\SymGrp_{n-1},A)$, and conclude by applying \cref{prop:stdrep_coho_ses}.
\end{proof}

\begin{proposition}\label{prop:stdrep_dual_coho_vanishing_and_seq}
  Let $A$ be an abelian group endowed with trivial $\Sn$-action.
  Then we have for $k<n/2 -1$ the identity
  \[
    \Ho^k(\Sn,\StdRep_{n}^{\vee}\otimesZ A)=0,
  \]
  and for $k<n/2$ we have an exact sequence of group cohomology groups
  \[
    0\to\Ho^k(\Sn,\StdRep_{n}^{\vee}\otimesZ A)\to\Ho^{k+1}(\Sn,\ZZ\otimesZ A)\xrightarrow{\res}\Ho^{k+1}(\SymGrp_{n-1},\ZZ\otimesZ A).
  \]
\end{proposition}

\begin{proof}
  Similar to above, apply group group cohomology to the defining short exact cokernel sequence of $\StdRep_{n}^{\vee}\otimesZ A$, \cf\cref{par:stdrep_dual_definition}, to get the commutative diagram with exact rows
  \[\begin{tikzcd}[rotate90south/.style={anchor=south, rotate=270}, column sep=1em]
    \cdots \arrow[r] & \Ho^i(\Sn,\ZZ\otimesZ A) \arrow[rd, "\res_{\SymGrp_{n-1}}^{\Sn}"'] \arrow[r, "\Delta_{*}"] & \Ho^i(\Sn,\ZZ^n\otimesZ A) \arrow[d, "\shapiro"', "\sim" rotate90south] \arrow[r] & \Ho^i(\Sn,\StdRep_{n}^{\vee}\otimesZ A) \arrow[r] & \cdots \\
    & & \Ho^i(\SymGrp_{n-1},\ZZ\otimesZ A). & &
  \end{tikzcd}\]
  By \cref{prop:nakaoka_res_isomorphism}, the restriction map $\res_{\SymGrp_{n-1}}^{\Sn}$ is an isomorphism for $i<n/2$.
  This information for $i=k$ and $i=k+1$ yields the claims in the proposition.
\end{proof}

In order to handle the case $k=1$, we recall the following facts about Schur multipliers:

\begin{proposition}\label{prop:schur-multiplier}
  Let $G$ be a finite group, acting trivially on $\ZZ$ and $\Bbbk^{\times}$, where $\Bbbk$ is an algebraically closed field of characteristic zero.
  \begin{enumerate}
    \item
    We have $\Ho^k(G,\Bbbk^{\times})\simeq\Hom_{\ZZ}(\Ho_k(G,\ZZ),\Bbbk^{\times})$.
    In particular, the Schur multiplier $\Ho^2(G,\Bbbk^\times)$ and $\Ho_2(G,\ZZ)$ are isomorphic as abstract groups.
    \item
    For $n\geq4$, we have $\Ho^2(\Sn,\Bbbk^{\times})\simeq\ZZ/2\ZZ$.
    For $n\leq3$, we have $\Ho^2(\Sn,\Bbbk^{\times})=0$.
  \end{enumerate}
\end{proposition}

\begin{proof}
  \itemnum
  The universal coefficient theorem for group cohomology gives a short exact sequence
  \[
    0\to\Ext_{\ZZ}^1(\Ho_{k-1}(G,\ZZ),\Bbbk^{\times})\to\Ho^k(G,\Bbbk^{\times})\to\Hom_{\ZZ}(\Ho_k(G,\ZZ),\Bbbk^{\times})\to0.
  \]
  The group $\Ho_{k-1}(G,\ZZ)$ is a finite abelian group, since $G$ is a finite group and $\ZZ$ is a finitely generated abelian group.
  So it is isomorphic to a direct sum of finite cyclic groups.
  Finally, use that $\Ext_{\ZZ}^1(\ZZ/n\ZZ,\Bbbk^{\times})\simeq\Bbbk^{\times}/(\Bbbk^{\times})^{n}=1$.
  
  For the second claim, write again $\Ho_{2}(G,\ZZ)$ as a direct sum of finite cyclic groups.
  Now note that $\ZZ/n\ZZ$ is isomorphic to the group of roots of unity $\ROU_n(\Bbbk)\simeq\Hom_{\ZZ}(\ZZ/n\ZZ,\Bbbk^{\times})$.

  \itemnum
  This goes back to Schur \cite{Sch:11}.
  For reference see \cite[Thm.~2.12.3]{Kar:87}.
%
\end{proof}

\begin{proposition}\label{prop:h1-dual-stdrep-vanishing}
  Let $A$ be an abelian group endowed with trivial $\Sn$-action.
  Then we have
  \[
    \Ho^1(\Sn,\StdRep_{n}^{\vee}\otimesZ A)=0
  \]
  when $n=3$ or $n\geq5$.
  For $n=4$ we have
  \[
    \Ho^1(\SymGrp_{4},\StdRep_{4}^{\vee}\otimesZ A)\simeq A[2].
  \]
\end{proposition}

\begin{proof}
  The claim for $n\geq5$ follows directly from \cref{prop:stdrep_dual_coho_vanishing_and_seq}, which also provides the following information
  \begin{align*}
    \Ho^1(\SymGrp_{3},\StdRep_{3}^{\vee}\otimesZ A) &=\ker(\res\colon\Ho^2(\SymGrp_{3},A)\to\Ho^2(\SymGrp_{2},A)), \\
    \Ho^1(\SymGrp_{4},\StdRep_{4}^{\vee}\otimesZ A) &=\ker(\res\colon\Ho^2(\SymGrp_{4},A)\to\Ho^2(\SymGrp_{3},A)).
  \end{align*}
  
  Since the action on $A$ is trivial, we can apply the universal coefficient theorem for group cohomology to get the short exact sequence
  \[
    0\to\Ext_{\ZZ}^1(\Ho_1(\Sn,\ZZ),A)\xrightarrow{\gamma}\Ho^2(\Sn,A)\to\Hom_{\ZZ}(\Ho_2(\Sn,\ZZ),A)\to0.
  \]
  Recall that $\Ho_1(\Sn,\ZZ)\simeq(\Sn)^\ab$.
  We can describe $\gamma$ as the composition
  \[\begin{tikzcd}
    \Ext_{\ZZ}^1((\Sn)^\ab,A) \arrow[rr, "\gamma"] \arrow[dr,"f"'] & & \Ho^2(\Sn,A) \\
    & \Ho^2((\Sn)^\ab,A), \arrow[ur, "g"'] &
  \end{tikzcd}\]
  where the map $g$ is induced from $\Sn\twoheadrightarrow(\Sn)^\ab$ by functoriality, and the map $f$  sends an extension of abelian groups to itself, but viewed as a central extension of (non-necessarily abelian) groups, \cf\cite[Thm.~1.8]{Bey:82}.
  In particular, these maps are compatible with the restriction maps~$\res_{\SymGrp_{n-1}}^{\Sn}$.
    
  For $n\leq3$ we have a trivial Schur multiplier $\Ho_2(\Sn,\ZZ)=0$, so $\gamma$ is an isomorphism.
  The inclusion $\SymGrp_2\hookrightarrow\SymGrp_3$ induces an isomorphism $(\SymGrp_{2})^{\ab}\isoarr(\SymGrp_{3})^{\ab}$, and hence and isomorphism of $\Ext$-groups.
  In conclusion, the restriction map $\res\colon\Ho^2(\SymGrp_{3},A)\to\Ho^2(\SymGrp_{2},A)$ is an isomorphism, and $\Ho^1(\SymGrp_{3},\StdRep_{3}^{\vee}\otimesZ A)=0$.
  
  Now we consider the case $n=4$.
  We have the following commutative diagram with exact rows
  \[\begin{tikzcd}
    0 \arrow[r] & \Ext_{\ZZ}^1((\SymGrp_{4})^{\ab},A) \arrow[r] \arrow[d, "\simeq"', "\res"] & \Ho^2(\SymGrp_{4},A) \arrow[r] \arrow[d, "\res"] & \Hom_{\ZZ}(\Ho_2(\SymGrp_{4},\ZZ),A) \arrow[r] \arrow[d] & 0 \\
    0 \arrow[r] & \Ext_{\ZZ}^1((\SymGrp_{3})^{\ab},A) \arrow[r] & \Ho^2(\SymGrp_{3},A) \arrow[r] & \Hom_{\ZZ}(\Ho_2(\SymGrp_{3},\ZZ),A) \arrow[r] & 0.
  \end{tikzcd}\]
  From this, and the facts $\Ho_2(\SymGrp_{3},\ZZ)=0$ and $\Ho_2(\SymGrp_{4},\ZZ)\simeq\ZZ/2\ZZ$, we conclude that
  \begin{align*}
    \Ho^1(\SymGrp_{4},\StdRep_{4}^{\vee}\otimesZ A)
    &=\ker(\res\colon\Ho^2(\SymGrp_{4},A)\to\Ho^2(\SymGrp_{3},A)) \\
    &\simeq\Hom_{\ZZ}(\Ho_2(\SymGrp_{4},\ZZ),A) \\
    &\simeq\Hom_{\ZZ}(\ZZ/2\ZZ,A) \\
    &\simeq A[2].
  \qedhere
  \end{align*}
\end{proof}

\begin{theorem}\label{thm:first-group-cohomology-ab-var-with-std-rep}
  Let $A$ be the group of $\kbar$-rational points of an abelian variety.
  Then we have
  \[
    \Ho^1(\Sn,\ZZ\times(A\otimesZ\StdRep_n)\times(A\otimesZ\StdRep_n)^{\vee})\simeq
    \begin{cases}
      0 & \text{if}\ n\ \text{odd, or}\ n=2 \\
      A[2] & \text{if}\ n\ \text{even, and}\ n\neq2,4 \\
      A[2]\oplus A^{\vee}[2] & \text{if}\ n=4.
    \end{cases}
  \]
\end{theorem}


\begin{proof}
  We have $\Ho^1(\Sn,\ZZ)=\Hom(\Sn,\ZZ)=0$, since $\Sn$ acts trivially on $\ZZ$ and the latter is torsion free.
  For $n\geq3$, we have $\Ho^1(\Sn,A\otimesZ\StdRep_{n})\isoarr A[2][n]$, since the $\kbar$-rational points of an abelian variety are an ($n$-)divisible group, \ie $A/nA=0$.
  Of course, $A[2][n]$ is just $A[2]$ when $n$ is even, and $0$ when $n$ is odd.
  We saw that $\Ho^1(\Sn,A^{\vee}\otimesZ\StdRep_{n}^{\vee})=0$ when $n\geq5$ or $n=3$, and $\Ho^1(\SymGrp_{4},A^{\vee}\otimesZ\StdRep_{4}^{\vee})=A^{\vee}[2]$ when $n=4$.
  
  The case $n=2$ was treated in \cite[Prop.~4.8]{Plo:Thesis}.
  Let us briefly spell out the details.
  Recall that we have an isomorphism $\StdRep_{2}\simeq\StdRep_{2}^{\vee}$, both being isomorphic to the sign-representation.
  Now a 1-cocycle $f\colon\SymGrp_{2}\to A\otimesZ\StdRep_{2}$ is determined by the image point $f((1\,2))=(a,-a)$.
  But the equality $(1\,2).(-a/2,a/2)-(-a/2,a/2)=(a,-a)$ realized $f$ as a 1-coboundary.
\end{proof}


\section{Preliminaries on simple abelian varieties}
\label{sec:abelian-varieties-preliminaries}

In this section we review some facts about the groups of homomorphisms between an abelian variety $A$ and its dual $A^\vee$ when $\End(A)=\ZZ$, and we discuss the notion of dual polarizations.

\begin{situation}
  Let $A$ be an abelian variety over a field $k$.
  Assume for simplicity that the characteristic of $k$ is $0$.
  Let $g=\dim(A)$; we are ultimately only interested in the case $g=2$.
\end{situation}

\begin{passage}
\label{prop:lambda_zero_polarization}
  Recall that if $\End(A)=\ZZ$, then $\End(A^\vee)=\ZZ$ and $\Hom(A,A^\vee)\simeq\ZZ$.
  Indeed, dualizing morphisms gives an isomorphism $\End(A)\isoarr\End(A^\vee)$, and composition with some isogeny $\lambda'\colon A^{\vee}\to A$, \eg the dual of a polarization, provides an injection $\Hom(A,A^\vee)\hookrightarrow\End(A)$.
  
  Furthermore, whenever $\Hom(A,A^\vee)=\ZZ\cdot\lambda_0$, then $\lambda_0$ or $-\lambda_0$ is a polarization, since by the existence of polarizations some multiple of $\lambda_0$ has to be one.
  In the following, when $\lambda_0$ or $-\lambda_0$ could be a polarization, we will implicitly assume that the former is one.
\end{passage}

\begin{definition}
  \begin{enumerate}
    \item The \emph{exponent}~$\erm(\lambda)$ of a polarization $\lambda\colon A\to A^\vee$ is defined as the smallest natural number $n$ such that $\ker(\lambda)\subset\ker([n])$.
    
    \item Writing $\ker(\lambda)_{\kbar}\simeq(\ZZ/d_1\ZZ\times\dots\times\ZZ/d_g\ZZ)^2$ where $d_i\in\NN$ such that $d_i\mid d_{i+1}$, we call $(d_1,\dots,d_g)$ the \emph{type of the polarization}.
  \end{enumerate}
\end{definition}

\begin{passage}[Dual polarization]
  \Cf\cite[§2, Thm.~2.1]{BL:02}.
  Let $\lambda\colon A\to A^\vee$ be a polarization, then there exists a unique polarization $\lambda^D\colon A^\vee\to A^{\vee\vee}\simeq A$ such that $\lambda\circ\lambda^D=[\erm(\lambda)]$ and $\lambda^D\circ\lambda=[\erm(\lambda)]$.
  Let $(d_1,\dots,d_g)$ be the type of the polarization $\lambda$, and define the \emph{dual polarization} as $\dualpol{\lambda}\coloneqq d_{1}\lambda^D$.
  
  \label{prop:dual_polariazion_symmetric}
  In particular, since $\dualpol{\lambda}$ is a polarization, it is symmetric, \ie $(\dualpol{\lambda_0})^{\vee}=\eval\circ\dualpol{\lambda_0}$.
\end{passage}

\begin{remark}
  The dual polarization should not be confused with the \emph{dual isogeny} $\dualisog{\lambda}\colon A^{\vee}\to A$, which is the unique morphism satisfying $\dualisog{\lambda}\circ\lambda=[\deg(\lambda)]$ and $\lambda\circ\dualisog{\lambda}=[\deg(\lambda)]$.
  Of course, it should neither be confused with the dual homomorphism $\lambda^{\vee}\colon A^{\vee\vee}\to A^{\vee}$.
\end{remark}  

\begin{proposition}
  Assume $g=\dim(A)=2$ and $\Hom(A,A^\vee)=\ZZ\cdot\lambda_0$, then
  \begin{enumerate}
    \item $\Hom(A^\vee,A)=\ZZ\cdot\dualpol{\lambda_0}$, and
    \item $\dualpol{\lambda_0}\circ\lambda_0=[\erm(\lambda_0)]$, as well as $\lambda_0\circ\dualpol{\lambda_0}=[\erm(\lambda_0)]$, with $\erm(\lambda_0)^2=\deg(\lambda_0)$.
  \end{enumerate}
\end{proposition}

\begin{proof}
  \itemnum
  In order to keep the formulas a bit more transparent, we consider $g$ to be arbitrary in the first part of this argument.
  Let $(d_1,\dots,d_g)$ be the type of the polarization $\lambda_0$, then by \cite[Prop.~2.2]{BL:02} the type of $\lambda_0^D$ is $(1,d_g/d_{g-1},\dots,d_g/d_1)$, and the type of $(\lambda_0^D)^D$ is $(1,d_2/d_1,\dots,d_g/d_1)$.
  This lets us compute the degrees
  \begin{align*}
    \deg(\lambda_0) &=d_1^2\dots d_g^2, \\
    \deg(\lambda_0^D) &=(d_g\dots d_g)^2/(d_g\dots d_1)^2=d_g^{2g}/\deg(\lambda_0), \\
    \deg((\lambda_0^D)^D) &=(d_1\dots d_g)^2/(d_1\dots d_1)^2=\deg(\lambda_0)/d_1^{2g}.
  \end{align*}
  
  Now, $\Hom(A,A^\vee)=\ZZ\cdot\lambda_0$ implies that $\deg(\lambda_0)$ divides $\deg((\lambda_0^D)^D)=\deg(\lambda_0)/d_1^{2g}$, which shows that $d_1=1$.
  Using $g=2$, we see that
  \[
    \deg(\lambda_0^D)=d_2^4/d_2^2=\deg(\lambda_0).
  \]
  
  Writing $\Hom(A^\vee,A)=\ZZ\cdot\widetilde{\lambda}$, and applying the above arguments to $\widetilde{\lambda}$, yields $\deg(\widetilde{\lambda}^D)=\deg(\widetilde{\lambda})$.
  So $\deg(\widetilde{\lambda})\geq\deg(\lambda_0^D)$, otherwise $\deg(\widetilde{\lambda}^D)$ contradicts the minimality of $\deg(\lambda_0)$. 
  In conclusion, since $\lambda_0^D$ is already a multiple of $\widetilde{\lambda}$, we get $\widetilde{\lambda}=\lambda_0^D$.
  
  \itemnum
  Note that the arguments above show that $\dualpol{\lambda_0}=\lambda_0^D$ and $\erm(\lambda_0)^2=d_2^2=\deg(\lambda_0)$.
\end{proof}

%


\section{$\Sn$-invariant symplectic isomorphisms}

We are concerned with the computation of the $\Sn$-invariants of the group of symplectic automorphisms $\Sp(A\otimesZ\StdRep_n)$, as well as of the set of symplectic isomorphisms~$\Sp(A\otimesZ\StdRep_n,A^{\vee}\otimesZ\StdRep_n)$.
As before, we let $A$ be some abelian variety over a field of characteristic zero.

\begin{passage}
  Since $\StdRep_n\simeq\ZZ^{n-1}$ as abelian groups, we have a natural isomorphism $(A\otimesZ\StdRep_n)^\vee\simeq A^{\vee}\otimesZ\StdRep_n^\vee$ of abelian varieties.
  The $\Sn$-action on $A\otimesZ\StdRep_n$ induces a (left) action on $(A\otimesZ\StdRep_n)^\vee$ via $\sigma.\alpha\coloneqq ((\sigma^{-1}).)^{\vee}(\alpha)$.
  Of course, the $\Sn$-action on $\StdRep_n^\vee$ also induces an action on $A^{\vee}\otimesZ\StdRep_n^\vee$.
\end{passage}

\begin{proposition}\label{prop:abvar-dual-of-product}
  The natural isomorphism $\varphi\colon(A\otimesZ\StdRep_n)^\vee \isoarr A^{\vee}\otimesZ\StdRep_n^\vee$ is $\Sn$-equivariant.
\end{proposition}

\begin{proof}
  Recall that $\varphi$ is given by the pullback morphism $(\iota_1^\vee,\dots,\iota_{n-1}^\vee)$, where $\iota_j\colon A\to A^{n-1}$ is the `$j$-th coordinate embedding'. 
  
  Dualizing the sequence Seq.~\eqref{ses:abvar_stdrep_def} of abelian varieties yields the short exact sequence in the top row of the following diagram, \cf\cite[\nopp 8.9]{EGM}:
  \[\begin{tikzcd}
    0 \arrow[r] & A^\vee \arrow[r, "\Sigma^\vee"] \arrow[d, "\id"] & (A^{n})^\vee \arrow[r, "f^{\vee}"] \arrow[d, "\psi"] & (A\otimesZ\StdRep_n)^\vee \arrow[r] \arrow[d, "\varphi"] & 0 \\
    0 \arrow[r] & A^\vee \arrow[r, "\Diagonal"] & (A^{\vee})^{n} \arrow[r, "g"] & A^{\vee}\otimesZ\StdRep_n^\vee \arrow[r] & 0
  \end{tikzcd}\]
  The isomorphism $\psi$ is defined analogously to $\varphi$.
  We leave the check that both squares commute to the reader.
  
  Finally, note that $\psi$ is equivariant.
  This can be checked quickly since the $\Sn$-action only ``permutes the indices''.
  This forces $\varphi$ to be equivariant as well.
\end{proof}

\begin{proposition}
  We have $\Sn$-equivariant isomorphisms
  \begin{enumerate}
    \item $\Hom(A\otimesZ\StdRep_n,A\otimesZ\StdRep_n) \quad\quad\;\; \simeq\Hom(A,A)\otimesZ\Hom_{\ZZ}(\StdRep_n,\StdRep_n)$,
    \item $\Hom((A\otimesZ\StdRep_n)^{\vee},A\otimesZ\StdRep_n) \;\;\;\;\, \simeq\Hom(A^{\vee},A)\otimesZ\Hom_{\ZZ}(\StdRep_n^{\vee},\StdRep_n)$,
    \item $\Hom(A\otimesZ\StdRep_n,(A\otimesZ\StdRep_n)^{\vee}) \quad\; \simeq\Hom(A,A^{\vee})\otimesZ\Hom_{\ZZ}(\StdRep_n,\StdRep_n^{\vee})$,
    \item $\Hom((A\otimesZ\StdRep_n)^{\vee},(A\otimesZ\StdRep_n)^{\vee}) \simeq\Hom(A^{\vee},A^{\vee})\otimesZ\Hom_{\ZZ}(\StdRep_n^{\vee},\StdRep_n^{\vee})$.
  \end{enumerate}
\end{proposition}

\begin{proof}
  We spell out the proof for \rom{3}; the other cases are analogous.
  In view of the previous proposition, we want to show
  $
    \Hom(A\otimesZ\StdRep_n,A^{\vee}\otimesZ\StdRep_n^{\vee}) \simeq\Hom(A,A^{\vee})\otimesZ\Hom_{\ZZ}(\StdRep_n,\StdRep_n^{\vee}).
  $
  Since $\StdRep_n$ and $\StdRep_{n}^{\vee}$ are isomorphic to $\ZZ^{n-1}$ as abelian groups, we see that
  \begin{align*}
    \Hom(A\otimesZ\StdRep_n,A^{\vee}\otimesZ\StdRep_n^{\vee})
    &\simeq \Mat((n-1)\times(n-1),\Hom(A,A^{\vee})) \\
    &\simeq \Hom(A,A^{\vee})\otimesZ\Hom(\StdRep_n,\StdRep_n^{\vee}).
  \end{align*}
  
  Recall that the induced action on homomorphism spaces is given by $\sigma.f\coloneqq \sigma.\circ f\circ (\sigma^{-1})_{.}$, so the action on the right hand side is given by
  \[
    \sigma.(g\otimes h)=g\otimes (\sigma.\circ h\circ (\sigma^{-1})_{.}),
  \]
  while the action on the left hand side is given by
  \begin{align*}
    \sigma.(g\otimes h)
    &=(\sigma.\circ (g\otimes h)\circ (\sigma^{-1})_{.}) \\
    &=(\id\otimes\sigma.)\circ(g\otimes h)\circ(\id\otimes(\sigma^{-1})_{.}) \\
    &=g\otimes(\sigma.\circ h\circ (\sigma^{-1})_{.}).
    \qedhere
  \end{align*}
\end{proof}

%
%

Mukai~\cite{Muk:98}, Polishchuk~\cite{Pol:96}, and later Orlov~\cite{Orl:02}, introduced and studied the following group of ``symplectic'' isomorphisms in the context of derived equivalences of abelian varieties. 

\begin{definition}\label{def:symplectic-isomorphism}
  For a homomorphism $f\colon A\times A^{\vee}\to B\times B^{\vee}$ of abelian varieties we write
  \[
    f=\begin{pmatrix} f_1 & f_2 \\ f_3 & f_4 \end{pmatrix}
    \ \text{and}\ 
    \tilde{f}\coloneqq\begin{pmatrix} f_{4}^{\vee} & -f_{2}^{\vee} \\ -f_{3}^{\vee} & f_{1}^{\vee} \end{pmatrix},
  \]
  where we have implicitly identified $A$ with $A^{\vee\vee}$, and $B$ with $B^{\vee\vee}$.
  Denote by
  \[
    \Sp(A,B)\coloneqq\{f\colon A\times A^{\vee}\isoarr B\times B^{\vee} \mid \tilde{f}=f^{-1}\}
  \]
  the set of \emph{symplectic isomorphisms}, and define the \emph{Mukai--Polishchuk group} as $\Sp(A)\coloneqq\Sp(A,A)$.
\end{definition}

The next proposition gives a hint why one calls elements of $\Sp(A,B)$ symplectic isomorphisms.
See \cite[§1]{LT:17} for a more thorough justification for the name.

\begin{proposition}
  Assume that $\dim(A)=2$ and $\End(A)=\ZZ$.
  Write $\Hom(A,A^\vee)=\ZZ\cdot\lambda_0$, and let $d^2=\deg(\lambda_0)$.
  Then we have an isomorphism
  \[
    \Sp(A\otimesZ\StdRep_n) \simeq \left\{ \begin{pmatrix}M_1 & M_2 \\ M_3 & M_4\end{pmatrix}\in\Sp(2(n-1),\ZZ) \;\middle|\; M_3\equiv 0 \mod d \right\}.
  \]
\end{proposition}

\begin{proof}
  Let $f\in\Sp(A\otimesZ\StdRep_n)$, then by the results of~\cref{sec:abelian-varieties-preliminaries}
  \begin{alignat*}{2}
    f_1 &=\id &&\otimes g_1 \in\Hom(A,A)\otimes\Hom_\ZZ(\StdRep_n,\StdRep_n) \\
    f_2 &=\dualpol{\lambda_0} &&\otimes g_2 \in\Hom(A^\vee,A)\otimes\Hom_\ZZ(\StdRep_n^\vee,\StdRep_n) \\
    f_3 &=\lambda_0 &&\otimes g_3 \in\Hom(A,A^\vee)\otimes\Hom_\ZZ(\StdRep_n,\StdRep_n^\vee) \\
    f_4 &=\id &&\otimes g_4 \in\Hom(A^\vee,A^\vee)\otimes\Hom_\ZZ(\StdRep_n^\vee,\StdRep_n^\vee).
  \end{alignat*}
  Ignoring the $\Sn$-action for a moment, we have that $\StdRep_n\simeq\ZZ^{n-1}$ and $\StdRep_n^\vee\simeq\ZZ^{n-1}$, so we can view each $g_i$ as a matrix $M_i\in\Mat((n-1)\times(n-1),\ZZ)$.
  Since $\dualpol{\lambda_0}\circ\lambda_0=[d]$ and $\lambda_0\circ\dualpol{\lambda_0}=[d]$, multiplication in $\Sp(A\otimesZ\StdRep_n)$ is not quite matrix multiplication.
  Instead we get a group homomorphism
  \begin{align}\label{eq:sp-group-map}
    \Sp(A\otimesZ\StdRep_n) &\to\Mat(2(n-1)\times 2(n-1),\ZZ) \\
    f &\mapsto\begin{pmatrix}M_1 & M_2 \\ d\cdot M_3 & M_4\end{pmatrix}. \nonumber
  \end{align}

  Let us compute which matrix corresponds to $\widetilde{f}$.
  Recall that $g_i^\vee$ corresponds to the transposed matrix $M_i^{\trm}$; of course, when choosing a basis of $\StdRep_n$, we endow $\StdRep_n^\vee$ with the dual basis.
  Now, identifying $A$ with $A^{\vee\vee}$ implicitly, we get
  \[
    \tilde{f}
    =\begin{pmatrix} f_{4}^{\vee} & -f_{2}^{\vee} \\ -f_{3}^{\vee} & f_{1}^{\vee} \end{pmatrix}
    =\begin{pmatrix} \id^\vee \otimes g_4^\vee & -(\dualpol{\lambda_0})^\vee \otimes g_2^\vee \\ -\lambda_0^\vee \otimes g_3^\vee & \id^\vee \otimes g_1^\vee \end{pmatrix}
    =\begin{pmatrix} \id \otimes g_4^\vee & -(\dualpol{\lambda_0}) \otimes g_2^\vee \\ -\lambda_0 \otimes g_3^\vee & \id \otimes g_1^\vee \end{pmatrix},
  \]
  \cf\cref{prop:lambda_zero_polarization} and \cref{prop:dual_polariazion_symmetric}, which corresponds to the matrix
  \begin{align*}
    \begin{pmatrix} M_{4}^{\trm} & -M_{2}^{\trm} \\ -d\cdot M_{3}^{\trm} & M_{1}^{\trm} \end{pmatrix}
    &=J^{\trm}\begin{pmatrix}M_1 & M_2 \\ d\cdot M_3 & M_4\end{pmatrix}^{\!\!\trm} J
    && \text{where}\ J=\begin{pmatrix}0 & \Irm_{n-1} \\ -\Irm_{n-1} & 0\end{pmatrix}.
  \end{align*}
  This means that the condition $\widetilde{f}=f^{-1}$ singles out symplectic matrices.
  In conclusion, the map in~\labelcref{eq:sp-group-map} provides the desired isomorphism.
\end{proof}

\begin{definition}
  The \emph{Hecke congruence subgroup of level $l\in\NN\setminus\{0\}$} is defined as
  \[
    \HeckeCong{l}\coloneqq \left\{\begin{pmatrix}a_1 & a_2 \\ a_3 & a_4 \end{pmatrix} \in\SL(2,\ZZ) \;\middle|\; a_3\equiv 0\mod l \right\}.
  \]
\end{definition}

\begin{proposition}\label{prop:sp_invariants_are_hecke_cong}
  Let $n\geq3$ be an integer.
  \begin{enumerate}
    \item
    Assume that $\dim(A)=2$ and $\End(A)=\ZZ$, and let $d^2$ denote the minimal degree of a polarization of $A$.
    Then we have an isomorphism
    \[
      \Sp(A\otimesZ\StdRep_n)^{\Sn} \simeq \HeckeCong{n d}\subset\SL_2(\ZZ).
    \]
  
    \item
    When $A$ is an arbitrary abelian variety, one can associate to each polarization (or symmetric isogeny) $\lambda\colon A\to A^{\vee}$ of exponent $e$ an injective group homomorphism
    \[
      \HeckeCong{n e} \hookrightarrow \Sp(A\otimesZ\StdRep_n)^{\Sn}.
    \]
  \end{enumerate}
\end{proposition}


\begin{proof}
  \itemnum
  Write $\Hom(A,A^\vee)=\ZZ\cdot\lambda_0$, so that $d^2=\deg(\lambda_0)$.
  By \cref{prop:stdrep_to_stdrep_equivariant_maps,prop:stdrep_dual_to_stdrep_equivariant_maps} and the results of~\cref{sec:abelian-varieties-preliminaries} we can compute for $n\geq3$ the $\Sn$-fixed-points
  \begin{align*}
    \Hom(A\otimesZ\StdRep_n,A\otimesZ\StdRep_n)^{\Sn}
    &\simeq\Hom(A,A)\otimesZ\Hom_{\ZZ}(\StdRep_n,\StdRep_n)^{\Sn} \\
    &\simeq\Hom(A,A)\otimesZ\Hom_{\ZZ[\Sn]}(\StdRep_n,\StdRep_n) \\
    &\simeq(\ZZ\cdot\id)\otimesZ(\ZZ\cdot\id), \\
    \Hom(A\otimesZ\StdRep_n,(A\otimesZ\StdRep_n)^{\vee})^{\Sn}
    &\simeq(\ZZ\cdot\lambda_0)\otimesZ(\ZZ\cdot\phi_0), \\
    \Hom((A\otimesZ\StdRep_n)^{\vee},A\otimesZ\StdRep_n)^{\Sn}
    &\simeq(\ZZ\cdot\dualpol{\lambda_0})\otimesZ(\ZZ\cdot\dualisog{\phi}_0), \\
    \Hom((A\otimesZ\StdRep_n)^{\vee},(A\otimesZ\StdRep_n)^{\vee})^{\Sn}
    &\simeq(\ZZ\cdot\id)\otimesZ(\ZZ\cdot\id).
  \end{align*}
  This shows that
  \[
    \Sp(A\otimesZ\StdRep_n)^{\Sn} = \left\{ f=\begin{pmatrix}\!\!\! a_1\cdot(\id\otimes\id) & a_2\cdot(\dualpol{\lambda_0}\otimes\dualisog{\phi}_0) \\ a_3\cdot(\lambda_0\otimes\phi_0) & \!\!\! a_4\cdot(\id\otimes\id)\end{pmatrix} \;\middle|\; a_i\in\ZZ,\ \text{and}\ f^{-1}=\widetilde{f} \right\}.
  \]
  Again, since $\dualpol{\lambda_0}\circ\lambda_0=[d]$ and $\lambda_0\circ\dualpol{\lambda_0}=[d]$, as well as $\dualisog{\phi}_0\circ\phi_0=n=\phi_0\circ\dualisog{\phi}_0$, we get a group homomorphism
  \begin{align*}
    \Sp(A\otimesZ\StdRep_n)^{\Sn} &\to\GL(2,\ZZ) \\
    f &\mapsto\begin{pmatrix}a_1 & a_2 \\ n d\cdot a_3 & a_4 \end{pmatrix}.
  \end{align*}
  As above, and using \cref{prop:phi_and_phi_hat_symmetric}, we have
  \[
    \widetilde{f}
    =\begin{pmatrix}\; a_4\cdot(\id\otimes\id) & -a_2\cdot(\dualpol{\lambda_0}\otimes\dualisog{\phi}_0) \\ -a_3\cdot(\lambda_0\otimes\phi_0) & \; a_1\cdot(\id\otimes\id)\end{pmatrix}.
  \]
  Now the condition $f^{-1}=\widetilde{f}$ becomes again in matrix form $M^{-1}=J^{\trm}M^{\trm}J$, which just means $\det(M)=1$ in this instance.
  
  \itemnum
  Recall that $\lambda\circ\lambda^{\Drm}=[e]=\lambda^{\Drm}\circ\lambda$.
  Finally note that we have the following isomorphism and inclusion
   \[
    \HeckeCong{n e}\simeq
    \left\{\begin{pmatrix}\!\!\! a_1\cdot(\id\otimes\id) & a_2\cdot(\lambda^{\Drm}\otimes\dualisog{\phi}_0) \\ a_3\cdot(\lambda\otimes\phi_0) & \!\!\!\! a_4\cdot(\id\otimes\id)\end{pmatrix} \;\middle|\; a_1 a_4 - n e a_2 a_3=1 \right\}
    \subset \Sp(A\otimesZ\StdRep_n)^{\Sn}.
    \qedhere
  \]
\end{proof}

\begin{remark}
  When $n=2$, we have to replace $\HeckeCong{2d}$ by $\HeckeCong{d}\simeq\Sp(A)$ in \cref{prop:sp_invariants_are_hecke_cong}.\rom{1}, since $\phi_0\in\Hom(\StdRep_2,\StdRep_{2}^\vee)$ is not a generator, \cf\cref{rmk:canonical-map-n-two}.
  More generally, without assuming $\End(A)=\ZZ$, the action of $\SymGrp_{2}$ on $\Sp(A\otimes\StdRep_{2})$ is trivial, since any morphism of abelian varieties commutes with negation.
\end{remark}




Now we compute the $\Sn$-invariants of $\Sp(A\otimesZ\StdRep_n,A^{\vee}\otimesZ\StdRep_n)$.


\begin{proposition}\label{prop:sp_a_to_adual_invariants}
  Let $n\geq3$ be an integer.
  \begin{enumerate}
    \item
    If $A$ admits a polarization (or just a symmetric isogeny) $\lambda\colon A\to A^{\vee}$ of exponent $e$ with $\gcd(n,e)=1$, then $\Sp(A\otimesZ\StdRep_n,A^{\vee}\otimesZ\StdRep_n)^{\Sn}\neq\emptyset$ is non-empty.
    
    \item
    Conversely, assume that $\dim(A)=2$ and $\End(A)=\ZZ$, and let $d^2$ denote the minimal degree of a polarization of $A$, then we have $\Sp(A\otimesZ\StdRep_n,A^{\vee}\otimesZ\StdRep_n)^{\Sn}\neq\emptyset$ \iff $\gcd(n,d)=1$.
    
    \item
    In case $\Sp(A\otimesZ\StdRep_n,A^{\vee}\otimesZ\StdRep_n)^{\Sn}$ is non-empty, it is a (right) torsor under $\Sp(A\otimesZ\StdRep_n)^{\Sn}$.
  \end{enumerate}
%
\end{proposition}

\begin{proof}
  \rom{2}
  As above, using the results of \cref{sec:standard_rep,sec:abelian-varieties-preliminaries}, each $f\in\Sp(A\otimesZ\StdRep_n,A^{\vee}\otimesZ\StdRep_n)^{\Sn}$ is of the form
  \[
    f=
    \begin{pmatrix}
      \; a_1\cdot(\lambda_0\otimes\id) & \!\! a_2\cdot(\id\otimes\dualisog{\phi}_0) \\
      a_3\cdot(\id\otimes\phi_0) & a_4\cdot(\dualpol{\lambda_0}\otimes\id) 
    \end{pmatrix},
  \]
  for some $a_i\in\ZZ$, where $\lambda_0$ is the polarization of minimal degree.
  Using \cref{prop:phi_and_phi_hat_symmetric} and \cref{prop:lambda_zero_polarization} and \cref{prop:dual_polariazion_symmetric}, we can write $\widetilde{f}\in\Sp(A^{\vee}\otimesZ\StdRep_n,A\otimesZ\StdRep_n)$ as
  \[
    \widetilde{f}=
    \begin{pmatrix}
      \quad a_4\cdot(\dualpol{\lambda_0}\otimes\id) & -a_2\cdot(\id\otimes\dualisog{\phi}_0) \\
      -a_3\cdot(\id\otimes\phi_0) & \quad a_1\cdot(\lambda_0\otimes\id) 
    \end{pmatrix}.
  \]
  Keeping in mind $\dualpol{\lambda_0}\circ\lambda_0=[d]$ and $\lambda_0\circ\dualpol{\lambda_0}=[d]$, as well as $\dualisog{\phi}_0\circ\phi_0=n=\phi_0\circ\dualisog{\phi}_0$, the condition $\widetilde{f}=f^{-1}$, \ie $\widetilde{f}\circ f=\id=f\circ\widetilde{f}$, becomes $\det(f)\coloneqq a_1 a_4 d-a_2 a_3 n = 1$.
  Note that the last equation has a solution \iff $\gcd(n,d)=1$.
  
  \rom{1}
  Since $\gcd(n,e)=1$, there is a solution to $a_1 a_4 e-a_2 a_3 n = 1$ with $a_i\in\ZZ$.
  Now the element
  \[
    f=
    \begin{pmatrix}
      a_1\cdot(\lambda\otimes\id) & \!\! a_2\cdot(\id\otimes\dualisog{\phi}_0) \\
      \, a_3\cdot(\id\otimes\phi_0) & a_4\cdot(\lambda^{\Drm}\otimes\id) 
    \end{pmatrix}
  \]
  witnesses that $\Sp(A\otimesZ\StdRep_n,A^{\vee}\otimesZ\StdRep_n)^{\Sn}$ is non-empty.
  
  \rom{3}
  The claim about torsors follows from \cref{prop:invariants_of_torsor_is_pseudo_torsor} in view of \cref{par:sp_torsor} below.
\end{proof}

\begin{remark}
  For $n=2$, the condition in \cref{prop:sp_a_to_adual_invariants} should read $\gcd(1,d)=1$, which is vacuous.
  So we have $\Sp(A\otimesZ\StdRep_2,A^{\vee}\otimesZ\StdRep_2)^{\SymGrp_2}\neq\emptyset$, and under the assumptions of \rom{2} it is a (right) torsor under $\HeckeCong{d}$.
\end{remark}


\section{Invariant derived autoequivalences in Orlov's sequence}

We first recall Orlov's fundamental short exact sequence for derived equivalences of abelian varieties.
Afterwards we consider the sequence in the equivariant setup of the previous sections and apply group cohomology to it in order to arrive at the proof of \cref{thm:intro-main-thm-ses}.

\begin{theorem}[Orlov]\label{thm:orlov}
  Let $A$ and $B$ be two abelian varieties over an algebraically closed field of characteristic $0$, then we have a short exact sequence of groups
  \begin{equation}\label{ses:orlov_ses}
    0\to\ZZ\times A\times A^{\vee}\to \Aut(\Db(A))\xrightarrow{\gamma_A} \Sp(A)\to0,
  \end{equation}
  and a surjective map
  \[
    \gamma_{A,B}\colon\Eq(\Db(A),\Db(B))\twoheadrightarrow\Sp(A,B).
  \]
  Here $n\in\ZZ$ is mapped to the shift functor $[n]$, and $a\in A$ is mapped to the push-forward $(\trm_{a})_{*}$, where $\trm_a$ denotes translation by $a$, and $\alpha\in A^\vee$ is mapped to the twist functor $\Lca_{\alpha}\otimes-$.\footnotemark
  \footnotetext{This is different from Ploog's convention in \cite{Plo:Thesis} by a sign, where $a\in A$ is mapped to the pull-back functor $(\trm_{a})^{*}=(\trm_{-a})_{*}$.}
  
  Moreover, the maps $\gamma_A$ and $\gamma_{A,B}$ are compatible in the sense that \[\gamma_{A,B}(\Phi'\circ\Phi)=\gamma_{A,B}(\Phi')\gamma_{A}(\Phi)\] for every $\Phi\in\Aut(\Db(A))$ and $\Phi'\in\Eq(\Db(A),\Db(B))$.
\end{theorem}

\begin{proof}[Reference]
  \Cf\cite[Thm.~4.14, Prop.~4.11, Prop.~4.12]{Orl:02}, \cite[Appendix]{Orl:02a}.
  See also~\cite{Pol:96}.
\end{proof}



Now we consider Orlov's short exact sequence in the equivariant setup.

\begin{passage}[Induced action on $\Aut(\Db(X))$]
  When a group $G$ acts on a variety $X$, then $G$ acts from the left on $\Db(X)$ via $g.\Fca\coloneqq(g^\inv)^{*}\Fca$.
  Accordingly, the diagonal action of $G$ on $\Db(X\times X)$ becomes $g.\Pca=(g^\inv,g^\inv)^{*}\Pca$.
  Recall that we have for (auto)morphisms $g$ and $h$ of $X$
  \[
    \FM_{(g,h)^{*}\Pca} = h^{*}\circ\FM_{\Pca}\circ g_{*},
  \]
   \cf\cite[Ex.~1.6.(4)]{Plo:Thesis}.
  So for $F=\FM_{\Pca}\in\Aut(\Db(X))$, we get
  \begin{align*}
     g.F
    =g.\FM_\Pca
    &=\FM_{(g^\inv,g^\inv)^{*}\Pca} \\
    &=(g^\inv)^{*}\circ\FM_{\Pca}\circ g^{\inv}_{*} \\
    &=(g^\inv)^{*}\circ F\circ g^{\inv}_{*} \\
    &=(g^\inv)^{*}\circ F\circ g^{*},
  \end{align*}
  using $g^{\inv}_{*}\simeq g^{*}$ for the isomorphism $g$.
  Analogously, $G$ acts also on $\Eq(\Db(X),\Db(Y))$ by conjugation.
\end{passage}


\begin{proposition}\label{prop:orlovs_map_gamma_is_equivariant}
  Let $A$ and $B$ be abelian varieties, then the maps in Seq.~\eqref{ses:orlov_ses}
  \[
    0\to\ZZ\times(A\otimesZ\StdRep_n)\times(A\otimesZ\StdRep_n)^{\vee}
    \xrightarrow{\iota}\Aut(\Db(A\otimesZ\StdRep_n))
    \xrightarrow{\gamma_{A\otimesZ\StdRep_n}}\Sp(A\otimesZ\StdRep_n)\to0
  \]
  and, more generally, the map
  \[
    \gamma_{A\otimesZ\StdRep_n,B\otimesZ\StdRep_n}\colon
    \Eq(\Db(A\otimesZ\StdRep_n),\Db(B\otimesZ\StdRep_n))\to\Sp(A\otimesZ\StdRep_{n},B\otimesZ\StdRep_{n})
  \]
  are $\Sn$-equivariant.
\end{proposition}

\begin{proof}
  We consider the more general situation where a group $G$ acts on abelian varieties~$X$ and $X'$ by homomorphisms.
  The proposition then follows from the specialization $G=\Sn$ and $X=A\otimes\StdRep_n$ and $X'=B\otimes\StdRep_n$.
  
  Let us first treat the map $\iota$, which sends $k\in\ZZ$ to the shift functor $[k]$, a point $a\in X$ to the pushforward $(\trm_a)_{*}$ along translation by $a$, and a point $\alpha\in X^{\vee}$ to the twist $\Lca_{\alpha}\otimes -$ by the degree zero line bundle $\Lca_{\alpha}$ corresponding to $\alpha$.
  The equalities below refer to equality in $\Aut(\Db(X))$, \ie natural isomorphisms of functors.
  
  Let $g\in G$.
  On the one hand, $G$ acts trivially on $\ZZ$.
  On the other hand, we have
  \[
    g.[k]
    =(g^\inv)^{*}\circ [k]\circ g^{*}
    =(g^\inv)^{*}\circ g^{*}\circ [k]
    =[k].
  \]
  We claim that $g.(\trm_{a})_{*}=(\trm_{g(a)})_{*}$.
  Indeed, using $(\trm_{a})_{*}=\trm_{(-a)}^{*}$, we have
  \[
    g.(\trm_{a})_{*}
    =(g^\inv)^{*}\circ\trm_{(-a)}^{*}\circ g^{*}
    =(g\circ\trm_{(-a)}\circ g^{\inv})^{*}
    =(\trm_{-g(a)})^{*}
    =(\trm_{g(a)})_{*}.
  \]
  We claim that $g.(\Lca_{\alpha}\otimes-)=\Lca_{g.\alpha}\otimes-$.
  Indeed, recalling that $g.\alpha=(g^\inv)^{\vee}(\alpha)$, we have
  \begin{align*}
    g.(\Lca_{\alpha}\otimes-)
    &=(g^\inv)^{*}\circ(\Lca_{\alpha}\otimes-)\circ g^{*} \\
    &=((g^\inv)^{*}\Lca_{\alpha}\otimes(g^\inv)^{*}(-))\circ g^{*} \\
    &=(g^\inv)^{*}\Lca_{\alpha}\otimes - \\
    &=\Lca_{(g^\inv)^{\vee}(\alpha)}\otimes - \\
    &=\Lca_{g.\alpha}\otimes -.
  \end{align*}
  
  Now we treat the map $\gamma_{X,X'}\colon\Eq(\Db(X),\Db(X'))\to\Sp(X,X')$.
  By \cite[Ex.~4.5.(3)]{Plo:Thesis}\footnotemark\ we have
  \[
    \gamma_{X}(g^{*})=\begin{pmatrix}g^{-1} & 0 \\ 0 & g^{\vee}\end{pmatrix}.
  \]
  \footnotetext{Beware of a misprint in \loccit where $g^{*}$ was mistyped as $g_{*}$.}
  Now for $F\in\Eq(\Db(X),\Db(X'))$ let us write $\gamma_{X,X'}(F)=\begin{psmallmatrix}f_1&f_2\\f_3&f_4\end{psmallmatrix}$ and calculate
  \begin{align*}
    \gamma_{X,X'}(g.F)
    &=\gamma_{X,X'}((g^\inv)^{*}\circ F\circ g^{*}) \\
    &=\gamma_{X'}((g^\inv)^{*})\cdot\gamma_{X,X'}(F)\cdot\gamma_{X}(g^{*}) \\
    &=\begin{pmatrix}g & 0 \\ 0 & (g^\inv)^{\vee}\end{pmatrix}
      \begin{pmatrix}f_1 & f_2 \\ f_3 & f_4\end{pmatrix}
      \begin{pmatrix}g^{-1} & 0 \\ 0 & g^{\vee}\end{pmatrix} \\
    &=\begin{pmatrix}g\circ f_1\circ g^\inv & g\circ f_2\circ g^\vee \\ (g^\inv)^{\vee}\circ f_3\circ g^\inv & (g^\inv)^{\vee}\circ f_4\circ g^{\vee}\end{pmatrix}.
  \end{align*}
  This is exactly how $G$ acts on $\Sp(X,X')$.
  Indeed, recall again that $g\in G$ acts on a map $\phi$ between two $G$-sets by $g.\phi=g.\circ\phi\circ(g^{\inv}).$, and the action of $G$ on $X^\vee$ (\resp $X'^{\vee}$) is given by $g.\alpha=(g^\inv)^{\vee}(\alpha)$.
  Considering $\Hom(X,X')$, $\Hom(X^{\vee},X')$, $\Hom(X,X'^{\vee})$ and $\Hom(X^{\vee},X'^{\vee})$ produces exactly the formulas above.
\end{proof}

Let us briefly recall some results from non-abelian group cohomology.
We use \cite[Ch.~II]{Ber:IGH} as a reference. Alternatively, see \cite[§1]{BS:64}, \cite[§I.5]{Ser:GC} or \cite[§III.3]{Gir:CNA}. 

\begin{passage}[Non-abelian group cohomology]
  Let $G$ be a group and let $1\to A\xrightarrow{i} B\xrightarrow{p} C\to 1$ be a short exact sequence of groups with $G$-action.
  Then we have a 6-term exact sequence of pointed sets (\ie kernels coincide with images)
\filbreak
  \[\begin{tikzcd}[row sep=tiny]
    0 \arrow[r]  & \Ho^0(G,A) \arrow[r, "i_{*}"] & \Ho^0(G,B) \arrow[r, "p_{*}"] & \Ho^0(G,C) \\
    {} \arrow[r, "\delta"] & \Ho^1(G,A) \arrow[r] & \Ho^1(G,B) \arrow[r] & \Ho^1(G,C),
  \end{tikzcd}\]
  where $\Ho^0(G,-)$ is the fixed-point functor $(-)^G$.
  The pointed sets $\Ho^1(G,-)$ can be defined via 1-cocycles and 1-coboundaries, and the map $\delta$ is called the \emph{connecting map}, \cf\cite[§II.4]{Ber:IGH}.
\end{passage}


\begin{passage}
  In the case that $A$ is an abelian group, $\Ho^1(G,A)$ is also an abelian group, but the connecting map $\delta$ does not need to be a group homomorphism.
  Let us explain that instead $\delta$ will be a crossed homomorphism for the following action:
  The group $C$ acts from the right on $A$ via inner automorphisms of $B$.
  That is, for $c\in C$ and $a\in A$, choose $b\in p^\inv(c)$ and define
  \[
    a.c
    \coloneqq b^{\inv}a b
    \coloneqq i^{\inv}(b^{\inv}i(a) b).
  \]
  Using that $A$ is abelian, one checks that this is independent of the choice of preimage $b$.
  Since for $g\in G$ one has $g.(a.c)=(g.a).(g.c)$, this right action is $G$-equivariant as long as $c\in \Ho^0(G,C)=C^G$.
  By functoriality, this induces a right action of $\Ho^0(G,C)$ on $\Ho^1(G,A)$.
  Concretely, $c\in \Ho^0(G,C)$ acts on a class $[\alpha]$ given by a cocycle $(\alpha_g)_{g\in G}$ via $[\alpha].c = [(\alpha_{g}.c)_{g\in G}]$.
\end{passage}

\begin{proposition}
  Let $G$ be a group and let $0\to A\to B\to C\to 1$ be a short exact sequence of groups with $G$-action.
  Assume that $A$ is an abelian group.
  Then the connecting map $\delta\colon \Ho^0(G,C)\to\Ho^1(G,A)$ is a crossed homomorphism, \ie
  \[
    \delta(c c')=\delta(c).c'+\delta(c')
  \]
  for $c,c'\in \Ho^0(G,C)$.
\end{proposition}

\begin{proof}
  \Cf\cite[Prop.~III.3.4.1]{Gir:CNA}.
  The proof is a straightforward check once one recalls that the map $\delta$ sends an element $c\in \Ho^0(G,C)$ to the class represented by the cocycle $(\alpha_g)_{g\in G}$, where
  \[
    \alpha_g\coloneqq i^{\inv}(b^{\inv}\cdot(g.b)),
  \]
  for some choice of $b\in p^{\inv}(c)$. 
\end{proof}

\begin{theorem}[Main \cref{thm:intro-main-thm-ses}]\label{thm:fmaut_invariants_seqs}
  Let $A$ be an abelian variety of dimension $\dim(A)=2$ over an algebraically closed field $\kbar$ of characteristic zero.
  Assume $\End(A)=\ZZ$, and let $d^2$ denote the minimal degree of a polarization of $A$.
  \begin{enumerate}
    \item For $n\neq 2,4$, we have an exact sequence of groups
    \[
      0\to\ZZ\times A[n]\to \Aut(\Db(A\otimesZ\StdRep_n))^{\Sn} \to\HeckeCong{n d}\xrightarrow{\delta} A[2][n].
    \]
    \item For $n=4$, we have an exact sequence of pointed sets ($\delta$ might not be a homomorphism)
     \[
      0\to\ZZ\times A[4]\to \Aut(\Db(A\otimesZ\StdRep_{4}))^{\SymGrp_{4}} \to\HeckeCong{4 d}\xrightarrow{\delta} A[2]\times A^{\vee}[2].
    \]
    \item For $n=2$, we have an exact sequence of groups
    \[
      0\to\ZZ\times A[2]\times A^{\vee}[2]\to \Aut(\Db(A\otimesZ\StdRep_2))^{\SymGrp_{2}} \to\HeckeCong{d}\simeq\Sp(A)\to0.
    \] 
  \end{enumerate}
  We have identified above, by abuse of notation, $A$ with its group of $\kbar$-rational points $A(\kbar)$.
\end{theorem}

\begin{remark}
  The assumptions ``$\dim(A)=2$'' as well as ``$\End(A)=\ZZ$'' only serve to be able to identify $\Ho^0(\Sn,\Sp(A\otimes\StdRep_n))$ with certain Hecke congruence groups in the sequences above, and for the analysis of the connecting map.
  The characteristic zero assumption is inherited from Orlov's sequence, \cf\cref{thm:orlov}, where it eventually comes from Bondal--Orlov's criterion for Fourier--Mukai equivalences.
\end{remark}

\begin{proof}
  Apply non-abelian group cohomology to Orlov's short exact sequence
  \begin{equation}\label{ses:orlovs_ses_with_stdrep}
    0\to\ZZ\times(A\otimes\StdRep_n)\times(A\otimes\StdRep_n)^{\vee}
    \to\Aut(\Db(A\otimes\StdRep_n))
    \to\Sp(A\otimes\StdRep_n)\to0
  \end{equation}
  to get the exact sequence
  \begin{align*}
    0&\to\Ho^0(\Sn,\ZZ\times(A\otimes\StdRep_n)\times(A\otimes\StdRep_n)^{\vee})
    \to\Ho^0(\Sn,\Aut(\Db(A\otimes\StdRep_n))) \\
    &\to\Ho^0(\Sn,\Sp(A\otimes\StdRep_n))
    \xrightarrow{\delta}\Ho^1(\Sn,\ZZ\times(A\otimes\StdRep_n)\times(A\otimes\StdRep_n)^{\vee}).
  \end{align*}
  For $n=3$ and $n\geq5$, we have computed in \cref{sec:group_coho_calculations}
  \begin{align*}
    \Ho^0(\Sn,\ZZ) &= \ZZ &
    \Ho^0(\Sn,A\otimes\StdRep_n) &\simeq A[n] &
    \Ho^0(\Sn,(A\otimes\StdRep_n)^{\vee}) &= 0 \\
    \Ho^1(\Sn,\ZZ) &= 0 &
    \Ho^1(\Sn,A\otimes\StdRep_n) &\simeq A[2][n] &
    \Ho^1(\Sn,(A\otimes\StdRep_n)^{\vee}) &= 0 \\
    && \Ho^0(\Sn,\Sp(A\otimes\StdRep_n)) &\simeq \HeckeCong{n d}.
  \end{align*}
  For $n=4$ the only difference is that $\Ho^1(\SymGrp_{4},(A\otimes\StdRep_{4})^{\vee}) \simeq A^{\vee}[2]$.
  The case $n=2$ was treated by Ploog in \cite[Prop.~4.8]{Plo:Thesis}; it also follows from our remarks and computations above.
  The differences are that $\Ho^0(\SymGrp_{2},(A\otimes\StdRep_{2})^{\vee})\simeq A^{\vee}[2]$, and $\Ho^0(\SymGrp_{2},\Sp(A\otimes\StdRep_{2}))=\Sp(A) \simeq \HeckeCong{d}$, and $\Ho^1(\SymGrp_{2},A\otimes\StdRep_{2})=0$.
  
  It is left to show that, for even $n\neq4$, the connecting map $\delta\colon\HeckeCong{n d}\to A[2]$ is a group homomorphism.
  Denote the equivalence associated to $(a,\alpha)\in(A\otimesZ\StdRep_n)\times(A\otimesZ\StdRep_n)^{\vee}$ by $\Phi_{(a,\alpha)}\coloneqq(\trm_a)_{*}\circ(\Lca_{\alpha}\otimes-)$.
  For any $\Phi\in\Aut(\Db(A\otimesZ\StdRep_n))$ we have
  \[
    \Phi\circ\Phi_{(a,\alpha)}\circ\Phi^{\inv} \simeq\Phi_{\gamma(\Phi)(a,\alpha)}
  \]
  up to shifts, \cf\cite[Cor.~9.58]{Huy:FM}.
  So for $f\in\Sp(A\otimesZ\StdRep_n)$ with $\gamma(\Phi)=f$, we have up to shifts
  \[
    \Phi_{(a,\alpha)}.f\coloneqq \Phi^{\inv}\circ\Phi_{(a,\alpha)}\circ\Phi \simeq\Phi_{f^{\inv}(a,\alpha)}.
  \]
  This means that the right-action of $\Sp(A\otimesZ\StdRep_n)$ on $(A\otimesZ\StdRep_n)\times(A\otimesZ\StdRep_n)^{\vee}$ is $(a,\alpha).f=f^{\inv}(a,\alpha)$.
  Since we are eventually interested in the action on $\Ho^1(\Sn,\ZZ\times(A\otimesZ\StdRep_n)\times(A\otimesZ\StdRep_n)^{\vee})$, and we have $\Ho^1(\Sn,\ZZ)=0$, we are fine with ignoring shifts.
  Now assume $f\in\Sp(A\otimesZ\StdRep_n)^{\Sn}$, and take any class $[(a_g,\alpha_g)_g]\in\Ho^1(\Sn,(A\otimesZ\StdRep_n)\times(A\otimesZ\StdRep_n)^{\vee})$, represented by a cocycle written as $(a,\alpha)\colon\Sn\to(A\otimesZ\StdRep_n)\times(A\otimesZ\StdRep_n)^{\vee}$.
  When $n\neq4$, we have $\Ho^1(\Sn,(A\otimesZ\StdRep_n)^{\vee})=0$, so we can take $\alpha=0$.
  The action reads now
  \[
    [(a,0)].f=[(f_{1}\circ a,f_{3}\circ a)]=[(f_{1}\circ a,0)],
  \]
  where we write $f^{\inv}=\begin{psmallmatrix}f_{1}&f_{2}\\f_{3}&f_{4}\end{psmallmatrix}$.
  Since $f^{\inv}\in\Sp(A\otimesZ\StdRep_n)^{\Sn}$, we have $f_{1}=k_{1}\cdot\id$ for some $k_{1}\in\ZZ$, \cf\cref{prop:sp_invariants_are_hecke_cong}.
  We claim that $k_{1}$ is odd, since $n$ is even.
  Indeed, identifying $f^{\inv}$ with an element of $\HeckeCong{nd}$, the condition $\det(f^{\inv})=k_{1}k_{4}-nd k_{3}k_{2}=1$ forces $k_1$ to be odd since~$n$ is already even. 
  Finally, using that $\Ho^1(\Sn,A\otimesZ\StdRep_n)\simeq A[2]$ is $2$-torsion, we see that $f$ acts trivially.
  In conclusion, $\delta$ is a crossed homomorphism for the trivial action, which makes it a group homomorphism.
\end{proof}

\begin{passage}
  Let us discuss the connecting map $\delta\colon\HeckeCong{4 d}\to A[2]\times A^{\vee}[2]$ in the case $n=4$.
  Assume $f\in\Sp(A\otimesZ\StdRep_{n})^{\SymGrp_{n}}$  and continue with the notation and setup from the proof of \cref{thm:fmaut_invariants_seqs}.
  We have the action
  \[
    \begin{pmatrix}a \\ \alpha\end{pmatrix}\!.f
    =\begin{pmatrix}f_{1}&f_{2}\\f_{3}&f_{4}\end{pmatrix}\begin{pmatrix}a \\ \alpha\end{pmatrix}
    =\begin{pmatrix}f_{1}\circ a +f_{2}\circ\alpha \\ f_{3}\circ a +f_{4}\circ\alpha\end{pmatrix}.
  \]
  The following proposition explains that this action is in general non-trivial.
  In particular, when $d$ is odd, $\delta$ will \emph{not} be a homomorphism unless $\im(\delta)\subset A[2]\times 0$.
\end{passage}

\begin{proposition}\label{prop:connecting-map-analysis-n4}
  Denote by $(f_i)_{*}$ the induced maps on first group cohomology.
  We have
  \begin{enumerate}
    \item $(f_{1})_{*}=\id$, 
    and $(f_{4})_{*}=\id$, 
    \item $(f_{3})_{*}=0$, 
    \item $f_{2}=k_2\cdot(\dualpol{\lambda_0}\otimes\dualisog{\phi}_0)$ and $\ker((\dualpol{\lambda_0}\otimes\dualisog{\phi}_0)_{*}) \simeq \ker(\dualpol{\lambda_0}\colon A^{\vee}[2]\to A[2])$, for some $k_2\in\ZZ$.
  \end{enumerate}
\end{proposition}
  
\begin{proof}
  \itemnum:
  As before, the condition $\det(f^{\inv})=1$ yields that we can view $f_1$ and $f_4$ as odd integers, so that we get the identities of cohomology classes $[f_{1}\circ a]=[a]$ and $[f_{4}\circ\alpha]=[\alpha]$.
  
  \itemnum
  We have $f_{3}=k_{3}\cdot(\lambda_0\otimes\phi_0)$ for some $k_{3}\in\ZZ$, \cf\cref{prop:sp_invariants_are_hecke_cong}.
  We obtain a short exact sequence
  \[
    0\to A[n]\xrightarrow{\Delta} A\otimesZ\StdRep_{n}\xrightarrow{\id\otimes\phi_0}A\otimesZ\StdRep_{n}^{\vee}\to0
  \]
  by tensoring Seq.~\eqref{ses:stdrep_canonical_map} $0\to\StdRep_{n}\xrightarrow{\phi_0}\StdRep_{n}^{\vee}\to\ZZ/n\ZZ\to0$ with $A$, and recalling the remark on the kernel of $\phi_0$ in the proof of \cref{par:stdrep_canonical_map}.
  Applying cohomology and using our calculations in \cref{cor:stdrep_coho_h1_ses}, we get
  \[\begin{tikzcd}[rotate90center/.style={anchor=center, rotate=270}, row sep=small]
    \Ho^1(\Sn,A[n]) \arrow[r, "\Delta_{*}"] \arrow[d, phantom, "\isoarr" rotate90center] & \Ho^1(\Sn,A\otimesZ\StdRep_n) \arrow[r, "(\id\otimes\phi_0)_{*}"] \arrow[d, phantom, "\isoarr" rotate90center] & \Ho^1(\Sn,A\otimesZ\StdRep_{n}^{\vee}) \\
    \Hom(\Sn,A[n]) & A[2][n], &
  \end{tikzcd}\]
  where the first isomorphism recognizes a cocycle $(g\mapsto a(g))$ as a homomorphism, and the second isomorphism sends a cocycle $(g\mapsto(a_i(g))_i)$ to $a_n(\tau)$ for some transposition $\tau\in\SymGrp_{n-1}\subset\Sn$.
  Using $\Hom(\Sn,A[n])\simeq\Hom(\Sn/\AlterGrp_{n},A[2])$ we see that $\Delta_{*}$ is surjective.
  Hence $(\id\otimes\phi_0)_{*}=0$, which implies $(\lambda_0\otimes\phi_0)_{*}=0$.
  
  \itemnum
  As before, we can write the map $f_2$ as $f_{2}=k_2\cdot(\dualpol{\lambda_0}\otimes\dualisog{\phi}_0)$ for some $k_{2}\in\ZZ$.
  We claim that  $\ker((\dualpol{\lambda_0}\otimes\id)_{*}\colon\Ho^1(\SymGrp_{4},A^{\vee}\otimesZ\StdRep_{4})\to\Ho^1(\SymGrp_{4},A\otimesZ\StdRep_{4}))$ is isomorphic to $\ker(\dualpol{\lambda_0}\colon A^{\vee}[2]\to A[2])$.
  For this, note that the map $(\dualpol{\lambda_0}\otimes\id)_{*}$ becomes $\dualpol{\lambda_0}$ under the identifications $\Ho^1(\SymGrp_{4},A^{\vee}\otimesZ\StdRep_{4})\simeq A^{\vee}[2]$ and $\Ho^1(\SymGrp_{4},A\otimesZ\StdRep_{4})\simeq A[2]$ of \cref{cor:stdrep_coho_h1_ses}.
  
  Finally we claim that $(\id\otimes\dualisog{\phi}_0)_{*}\colon\Ho^1(\SymGrp_{4},A^{\vee}\otimesZ\StdRep_{4}^{\vee})\to\Ho^1(\SymGrp_{4},A^{\vee}\otimesZ\StdRep_{4})$ is an isomorphism.
  For notation reasons we use $A$ in place of $A^{\vee}$ in the argument.
  Using the explicit form of $\dualisog{\phi}_0$ from \cref{rmk:dualisog-explicit-form}, and that $A$ is ($n$-)divisible, one checks that we have an exact sequence
  \[
    0\to A[n]\xrightarrow{\Delta} A[n]\otimesZ\StdRep_{n} \xrightarrow{\phi_0} A\otimesZ\StdRep_{n}^{\vee}\xrightarrow{\dualisog{\phi}_0} A\otimesZ\StdRep_{n} \to0.
  \]
  As above, by tensoring Seq.~\eqref{ses:stdrep_canonical_map} with $A[n]$, we obtain
  \[
    0\to A[n]\xrightarrow{\Delta} A[n]\otimesZ\StdRep_{n} \xrightarrow{\phi_0} A[n]\otimesZ\StdRep_{n}^{\vee} \xrightarrow{\Sigma}A[n]\to0,
  \]
  where $\Sigma([(a,0,\dots,0)])=a$.
  We recognize the map $\Sigma$ as the summation map.
  We get
  \begin{align*}
    \Ho^1(\Sn,(A[n]\otimesZ\StdRep_{n})/A[n]) \xrightarrow{(\phi_{0})_{*}} \Ho^1(\Sn,A[n]\otimesZ\StdRep_{n}^{\vee}) &\xrightarrow{\Sigma_{*}} \Ho^1(\Sn,A[n]) \\
    [f\colon\Sn\to A[n]\otimesZ\StdRep_{n}^{\vee}] &\mapsto [\Sigma\circ f].
  \end{align*}
  For $n=4$, the middle and right group are isomorphic to $A[2]$, where the latter isomorphism is given by evaluation at a transposition.
  One can (tediously) check using the Coxeter--Moore presentation of the symmetric group that, for $a\in A[2]$, the assignment
  \begin{align*}
    f((1\,2)) &\coloneqq [(0,0,a,0)] \\
    f((2\,3)) &\coloneqq [(a,0,0,0)] \\
    f((3\,4)) &\coloneqq [(0,a,0,0)]
  \end{align*}
  extends to a 1-cocycle on $\SymGrp_{4}$.
  By construction, these cocycles witness the surjectivity of $\Sigma_{*}$.
  So~$\Sigma_{*}$ becomes an isomorphism and $(\phi_0)_{*}=0$.
  It follows that the map $(\phi_0)_{*}$ in the sequence
  \[
    \Ho^1(\Sn,(A[n]\otimesZ\StdRep_{n})/A[n]) \xrightarrow{(\phi_{0})_{*}}  \Ho^1(\Sn,A\otimesZ\StdRep_{n}^{\vee}) \xrightarrow{(\dualisog{\phi}_0)_{*}}  \Ho^1(\Sn,A\otimesZ\StdRep_{n})
  \]
  is the zero map as well, since it factors as the previous map $(\phi_0)_{*}$ followed by the map induced by the inclusion $A[n]\subset A$.
  We conclude that $(\dualisog{\phi}_0)_{*}$ is injective.
  For $n=4$ its domain and codomain are isomorphic to $A[2]$, so it is an isomorphism.
  %
\end{proof}

\begin{question}
  For $n\neq2,4$, we have seen that $\ker(\delta)\subset\HeckeCong{nd}$ is a normal subgroup. It can be written as the intersection of at most $4$ subgroups of index $2$, since $A[2]$ is abstractly isomorphic to $(\ZZ/2\ZZ)^4$.
  We wonder if $\ker(\delta)\subset\HeckeCong{nd}$ is a congruence subgroup.
\end{question}



\section{Invariant derived equivalences via equivariant torsors}
\label{sec:equivariant_torsors}
\label{sec:ploogs-method}

In this section we will consider torsors which are equipped with an additional action by a group $G$.
See for reference \cite[§I.2]{NSW:CNF}, \cite[§1]{BS:64} and \cite[§I.5]{Ser:GC}.
Let $A$ be a \emph{$G$-group}, \ie a group equipped with a left-action by $G$.

\begin{definition}
  A \emph{$G$-equivariant $A$-pseudo-torsor} is a set $T$ equipped with a left-action by $G$ and a free transitive right-action by $A$, which are compatible in the sense that $g.(t.a)=(g.t).(g.a)$ for every $g\in G$, $a\in A$, $t\in T$.
  Provided that $T$ is non-empty, we call it an \emph{$G$-equivariant $A$-torsor}.
\end{definition}

\begin{definition}\label{def:equivariant-map-of-torsors}
  A morphism of $G$-equivariant $A$-pseudo-torsors is a map of sets which is $G$- and $A$-equivariant.
  More generally, when $\gamma\colon A\to A'$ is a homomorphism of $G$-groups, then a map $\widetilde{\gamma}\colon T\to T'$ from a $G$-equivariant $A$-pseudo-torsor to a $G$-equivariant $A'$-pseudo-torsor is called \emph{equivariant} provided that it is $G$-equivariant and $\widetilde{\gamma}(t.a)=\widetilde{\gamma}(t).\gamma(a)$ for every $t\in T$ and $a\in A$.
\end{definition}

Whenever $G$ is the trivial group, we drop ``$G$-equivariant'' from the terminology.

\begin{proposition}\label{prop:invariants_of_torsor_is_pseudo_torsor}
  Let $T$ be a $G$-equivariant $A$-pseudo-torsor, then the fixed-point set $T^G$ is a $A^G$-pseudo-torsor.
\end{proposition}

\begin{proof}
  The right-action of $A^G$ on $T^G$ is inherited from the action of $A$ on $T$.
  It is well-defined, since for $t\in T^G$ and $a\in A^G$, we have $g.(t.a)=(g.t).(g.a)=t.a$ for every $g\in G$.
  The action is still free and transitive, since for every $t,t'\in T^G$ there exists a unique $a\in A$ such that $t'=t.a$.
  We check that $a\in A^G$
  Indeed, for every $g\in G$ we have $t'=g.t'=g.(t.a)=(g.t).(g.a)=t.(g.a)$, and thus $g.a=a$ by uniqueness of $a$.
\end{proof}

\begin{proposition}\label{prop:torsors_as_cohomology_classes}\label{prop:torsor_cohomology_class_and_fixed_points}
  We have a canonical bijection between $\Ho^1(G,A)$ and the set of isomorphism classes of $G$-equivariant $A$-torsors:
  \begin{align*}
    \Ho^1(G,A) &\leftrightarrow\{\text{$G$-equivariant $A$-torsors}\}/{\simeq} \\
      [T] & \mapsfrom T
  \end{align*}
  Moreover, a $G$-equivariant $A$-torsor $T$ has a $G$-fixed-point \iff $[T]=0$.
\end{proposition}

\begin{proof}
  \Cf\cite[Prop.~1.2.3]{NSW:CNF} or \cite[Prop.~1.8]{BS:64} for the first part.
  Regarding fixed-points:
  Let $t_{0}\in T$.
  Then the cohomology class $[T]$ is represented by the 1-cocycle $(a_g)_g$ where $g.t_{0}=t_{0}.a_g$.
  If $t_{0}\in T^{G}$, then we have $g.t_{0}=t_{0}$, which implies $a_{g}=1$ for every $g\in G$, \ie $(a_g)_g$ is the trivial 1-cocycle.
  Conversely, assume $[T]=0$, \ie $(a_g)_g$ is cohomologous to the trivial 1-cocycle.
  This means that there exists an $a\in A$ such that $a_{g}=a\cdot(g.a)^{\inv}$ for every $g\in G$.
  Finally, the element $t'_{0}\coloneqq t_{0}.a$ is a $G$-fixed-point, since
  \[
    g.t'_{0}
    =g.(t_{0}.a)
    =(g.t_{0}).(g.a)
    =(t_{0}.a_{g}).(g.a)
    =t_{0}.(a\cdot(g.a)^{\inv}\cdot(g.a))
    =t_{0}.a
    =t'_{0}.
    \qedhere
  \]
\end{proof}

\begin{proposition}\label{prop:relative_equivariant_maps_and_cohomology_classes}
  Let $\gamma\colon A\to A'$ be a homomorphism of $G$-groups.
  Assume that there exists an equivariant map $\widetilde{\gamma}\colon T\to T'$ between a $G$-equivariant $A$-torsor and a $G$-equivariant $A'$-torsor.
  Then $\gamma_{*}\colon\Ho^1(G,A)\to\Ho^1(G,A')$ sends $[T]$ to $[T']$.
\end{proposition}

\begin{proof}
  Pick some element $t_{0}\in T$, and define $t'_{0}\coloneqq\widetilde{\gamma}(t_0)\in T'$.
  Then $[T]$ is represented by the cocycle $(a_g)_g$ where $g.t_0=t_{0}.a_g$.
  Now $[T']$ is represented by the cocycle $(\gamma(a_g))_g$, since
  \[
    g.t'_{0}=g.\widetilde{\gamma}(t_0)=\widetilde{\gamma}(g.t_0)=\widetilde{\gamma}(t_{0}.a_g)=\widetilde{\gamma}(t_0).\gamma(a_g)=t'_{0}.\gamma(a_g).
  \]
  This is exactly a cocycle representing $\gamma_{*}([(a_g)_g])$, as desired.
\end{proof}

Now we specialize the discussion to torsors that are relevant to our study of Fourier--Mukai equivalences.

\begin{passage}\label{par:sp_torsor}\label{par:orlovs_map_gamma_is_relative_equivariant}
  Let $A$ and $B$ be abelian varieties, then $\Sp(A,B)$ is a pseudo-torsor under $\Sp(A)$, where the right action is afforded by function composition.
  Indeed, use that $\Sp(A,B)$ is a subset of $\Isom(A\times A^{\vee},B\times B^{\vee})$ and analogously for $\Sp(A)$, and that $(g\circ f)^{\sim}=\widetilde{f}\circ\widetilde{g}$ for $f\in\Sp(A)$ and $g\in\Sp(A,B)$.
  Similarly $\Eq(\Db(A),\Db(B))$ is a pseudo-torsor under $\Aut(\Db(A))$, where the right action is given by functor composition.
  
  Recall that the map $\gamma_{A,B}\colon\Eq(\Db(A),\Db(B))\to\Sp(A,B)$ from \cref{thm:orlov} is an equivariant map of pseudo-torsors with respect to the homomorphism $\gamma_{A}\colon\Aut(\Db(A))\to\Sp(A)$, \ie $\gamma_{A,B}(\Phi'\circ\Phi)=\gamma_{A,B}(\Phi')\circ\gamma_{A}(\Phi)$, \cf\cite[Exer.~9.41]{Huy:FM}.
%
\end{passage}

\begin{passage}
  We specialize the situation further and consider firstly the $\Sn$-equivariant $\Sp(A\otimesZ\StdRep_{n})$-torsor
  \[
    T_1\coloneqq\Sp(A\otimesZ\StdRep_{n},A^{\vee}\otimesZ\StdRep_{n}).
  \]
  The actions are indeed compatible, since $\Sn$ acts by conjugation on homomorphism sets.
  Note that $T_1$ is non-empty since $A^{\vee}\otimesZ\StdRep_{n}\simeq (A^{n-1})^{\vee}$ is the dual abelian variety of $A\otimesZ\StdRep_{n}\simeq A^{n-1}$; a concrete witness is $g=\begin{psmallmatrix}0&\id\\-\id&0\end{psmallmatrix}\in T_1$.
  
  Secondly, consider the $\Sn$-equivariant $\Aut(\Db(A\otimesZ\StdRep_{n}))$-torsor
  \[
    T_2\coloneqq\Eq(\Db(A\otimesZ\StdRep_{n}),\Db(A^{\vee}\otimesZ\StdRep_{n})).
  \]
  Again, the actions are compatible since $\Sn$ acts by conjugation, and $T_2$ is non-empty since the map $\gamma_{(A\otimesZ\StdRep_n),(A^{\vee}\otimesZ\StdRep_n)}$ is surjective.
\end{passage}


\begin{corollary}
  Let $A$ be an abelian variety.
  Then $\Eq(\Db(A\otimesZ\StdRep_{n}),\Db(A^{\vee}\otimesZ\StdRep_{n}))^{\Sn}$ is a pseudo-torsor under $\Aut(\Db(A\otimesZ\StdRep_{n}))^{\Sn}$.
\end{corollary}

\begin{proof}
  This is a direct application of \cref{prop:invariants_of_torsor_is_pseudo_torsor} to the $\Sn$-equivariant torsor $T_2$.
\end{proof}

\begin{situation}
  From now on we work over an algebraically closed field $\Bbbk$ of characteristic zero.
\end{situation}

\begin{theorem}\label{thm:invariant-equivalence}
  Let $A$ be an abelian variety over an algebraically closed field of characteristic zero.
  Let $n\geq3$ and let $\lambda\colon A\to A^{\vee}$ be a polarization (or  symmetric isogeny) of exponent $e$.
  \begin{enumerate}
    \item Assume $n$ is odd.
    Then $\gcd(n,e)=1$ implies $\Eq(\Db(A\otimesZ\StdRep_{n}),\Db(A^{\vee}\otimesZ\StdRep_{n}))^{\Sn}\neq\emptyset$.
    \item If $\dim(A)=2$ and $\End(A)=\ZZ$, then the converse implication is true when we take $\lambda$ to be the polarization of minimal degree.
  \end{enumerate}
\end{theorem}


\begin{proof}
%
  We associate to the torsor $T_1$ the cohomology class $[T_1]\in\Ho^1(\Sn,\Sp(A\otimesZ\StdRep_n))$ and to the torsor $T_2$ the class $[T_2]\in\Ho^1(\Sn,\Aut(\Db(A\otimesZ\StdRep_n)))$, \cf\cref{prop:torsors_as_cohomology_classes}.
  Then the map $\gamma_{A\otimesZ\StdRep_n}\colon\Aut(\Db(A\otimesZ\StdRep_n))\to\Sp(A\otimesZ\StdRep_n))$ induces a map of non-abelian cohomology sets
  \[
    (\gamma_{A\otimesZ\StdRep_n})_{*}\colon\Ho^1(\Sn,\Aut(\Db(A\otimesZ\StdRep_n)))\to\Ho^1(\Sn,\Sp(A\otimesZ\StdRep_n)),
  \]
  which sends $[T_2]$ to $[T_1]$.
  Indeed, we have seen by \cref{prop:orlovs_map_gamma_is_equivariant} that $\gamma_{(A\otimesZ\StdRep_n),(A^{\vee}\otimesZ\StdRep_n)}\colon T_2\to T_1$ is an $\Sn$-equivariant map which is equivariant relative to $\gamma_{A\otimesZ\StdRep_n}$, so we can apply \cref{prop:relative_equivariant_maps_and_cohomology_classes}.
  
  \itemnum
  We have seen that $\gcd(n,e)=1$ implies $T_{1}^{\Sn}\neq\emptyset$, the latter being equivalent to $[T_1]=0$, \cf\cref{prop:sp_a_to_adual_invariants,prop:torsor_cohomology_class_and_fixed_points}.
  Thus $[T_2]\in\ker((\gamma_{A\otimesZ\StdRep_n})_{*})$.
  Using the cohomology sequence associated to Seq.~\eqref{ses:orlovs_ses_with_stdrep}
  \[
    0\to\ZZ\times(A\otimes\StdRep_n)\times(A\otimes\StdRep_n)^{\vee}
    \to\Aut(\Db(A\otimes\StdRep_n))
    \to\Sp(A\otimes\StdRep_n)\to0,
  \]
  this means that $[T_2]$ is in the image of $\Ho^1(\Sn,\ZZ\times(A\otimesZ\StdRep_n)\times(A\otimesZ\StdRep_n)^{\vee})$.
  But by \cref{thm:first-group-cohomology-ab-var-with-std-rep} for $n$ odd, the latter is zero.
  So we have $[T_2]=0$ and thus $T_{2}^{\Sn}\neq\emptyset$ by \cref{prop:torsor_cohomology_class_and_fixed_points}.

  \itemnum
  We have seen in \cref{prop:sp_a_to_adual_invariants} that in this case the condition $\gcd(n,e)=1$ is equivalent to $T_{1}^{\Sn}\neq\emptyset$ and to $[T_1]=0$, \cf\cref{prop:torsor_cohomology_class_and_fixed_points}.
  Finally, $[T_2]=0$ implies $[T_1]=(\gamma_{A\otimesZ\StdRep_n})_{*}[T_2]=0$, which implies $\gcd(n,e)=1$.
\end{proof}

\begin{remark}
  \begin{enumerate}
    \item In part \rom{2} of the theorem, we can take $\lambda\colon A\to A^{\vee}$ to be the polarization of minimal degree $d$.
    Then $\erm(\lambda)^2=d$, and the numerical condtion can be read as ``$\gcd(n,d)=1$''.
    \item The assumptions on the ground field $\Bbbk$ are inherited from Orlov's theorem (\cref{thm:orlov}).
    Also the vanishing of $\Ho^1(\Sn,A(\Bbbk)\otimesZ\StdRep_n)$ in the proof above uses that $A(\Bbbk)$ is $n$-divisible, which is fine when $\Bbbk=\overline{\Bbbk}$.
  \end{enumerate}
\end{remark}

%
%

Let us recall Ploog's method to enhance an invariant derived equivalence to an equivalence of equivariant derived categories.
For reference see \cite[§§1--2]{Plo:07} or \cite[Ch.~3]{Plo:Thesis}.

\begin{passage}
  Let $G$ be a finite group acting on two smooth projective varieties $X$ and $Y$.
  We consider the following three sets of derived equivalences:
  Firstly, we have the set of derived equivalences between $\Db(X)$ and $\Db(Y)$ which commute with the $G$-action up to isomorphism.
  In terms of Fourier--Mukai kernels, this is
  \begin{multline*}
    \Eq(\Db(X),\Db(Y))^{G} \\
    \simeq \{ \Pca\in\Db(X\times Y) \mid \FM_{\Pca}\colon\Db(X)\isoarr\Db(Y),\ \text{and}\ \forall g\in G\colon (g,g)^{*}\Pca\simeq\Pca \}.
  \end{multline*}
  Secondly, we have the set of derived equivalences between $\DbG{G}(X)$ and $\DbG{G}(Y)$.
  These are represented by kernels which are endowed with an equivariant structure for the $G\times G$-action on $X\times Y$:
  \[
    \Eq(\DbG{G}(X),\DbG{G}(Y))
    \simeq \{ (\widetilde{\Pca},\widetilde{\phi})\in\DbG{G\times G}(X\times Y) \mid \FM_{(\widetilde{\Pca},\widetilde{\phi})}\colon\DbG{G}(X)\isoarr\DbG{G}(Y) \}.
  \]
  Thirdly, interpolating between the two cases above, we have the set of derived equivalences $F\colon\Db(X)\isoarr\Db(Y)$ which are endowed with an equivariant structure witnessing that $F$~``commutes coherently'' with the $G$-action.
  Again, in terms of kernels, this is
  \[
    \Eq(\Db(X),\Db(Y))^{\hrm G}
    \coloneqq \{ (\Pca,\phi)\in\DbG{\Diagonal G}(X\times Y) \mid  \FM_{\Pca}\colon\Db(X)\isoarr\Db(Y) \},
  \]
  where $\Diagonal G\subset G\times G$ denotes the diagonal subgroup.
\end{passage}

\begin{passage}
  These sets are related to each other by a `forgetful' map and and `inflation' map:
  The `forgetful' map $\forget\colon\DbG{\Diagonal G}(X\times Y)\to\Db(X\times Y)$ maps $(\Pca,\phi)\mapsto\Pca$.
  Note that the kernels in the image of the forgetful map are still $G$-invariant under the diagonal $G$-action.
  
  For the subgroup $\Diagonal G\subset G\times G$ (or more generally any pair of sub-/supergroup) there is an inflation map $\inflation_{\Diagonal G}^{G\times G}\colon\DbG{\Diagonal G}(X\times Y)\to\DbG{G\times G}(X\times Y)$ which maps $(\Pca,\phi)$ to \[\inflation_{\Diagonal G}^{G\times G}(\Pca,\phi)= \bigoplus_{[g]\in {\Diagonal G}\backslash G\times G}g^{*}\Pca\] endowed with a suitable equivariant structure, see \cite[Def.~8.2.1]{BL:94} for details.
\end{passage}


We want to apply \cite[Thm.~6]{Plo:07} not only to autoequivalences but to the sets of equivalences described above, so we spell out the following more general statement of the theorem.

\begin{theorem}[Ploog]
  Let $G$ be a finite group which acts on two smooth projective varieties~$X$ and $Y$.
  Assume in \rom{2} that $G$ acts faithfully.
  \begin{enumerate}
    \item
    We have an exact sequence of groups
    \[
      0 \to \Hom(G,\Bbbk^{\times}) \to \Aut(\Db(X))^{\hrm G} \xrightarrow{\forget} \Aut(\Db(X))^G \xrightarrow{\delta_X} \Ho^2(G,\Bbbk^{\times})
    \]
    and an exact sequence of pseudo-torsors over the respective last three terms of the sequence above, \ie the maps are equivariant in the sense of \cref{def:equivariant-map-of-torsors} and images equal kernels:
    \[
      \Eq(\Db(X),\Db(Y))^{\hrm G} \xrightarrow{\forget} \Eq(\Db(X),\Db(Y))^{G} \xrightarrow{\delta_{X,Y}} \Ho^2(G,\Bbbk^{\times}).
    \]
    
    \item
    We have an exact sequence of groups
    \[
      0 \to \Center(G) \to \Aut(\Db(X))^{\hrm G} \xrightarrow{\inflation_{\Diagonal G}^{G\times G}} \Aut(\DbG{G}(X))
    \]
    and an equivariant map of pseudo-torsors over the respective last two terms of the sequence above:
    \[
      \Eq(\Db(X),\Db(Y))^{\hrm G} \xrightarrow{\inflation_{\Diagonal G}^{G\times G}} \Eq(\DbG{G}(X),\DbG{G}(Y)).
    \]
  \end{enumerate}
\end{theorem}

\begin{remark}
  Strictly speaking we should work with isomorphism classes of objects, otherwise we would have to deal with higher groups instead of abstract groups.
\end{remark}

\begin{proof}
  The part about groups is exactly \cite[Thm.~6]{Plo:07}; the part about pseudo-torsors is essentially proven in \loccit but not spelled out as such, so we provide a few pointers:
  
  The group structure on $\Aut(\Db(X))^{\hrm G}$ and its action on $\Eq(\Db(X),\Db(Y))^{\hrm G}$ are given by convolution of Fourier--Mukai kernels $(\Pca,\phi)\star(\Pca',\phi')\coloneqq(\Pca\star\Pca',\phi\star\phi')$ with $(\phi\star\phi')_{g}\coloneqq(\phi_{g}\star\phi'_{g})$, which corresponds to composition of associated Fourier--Mukai functors by \cite[Lem.~5.(3)]{Plo:07}, and a composition of equivalences is again an equivalence.
  Now \cite[Lem.~5.(5)]{Plo:07} provides inverses for kernels of equivalences endowed with an equivariant structure, which lets one deduce that the action is free and transitive.
  Similarly the group structure on $\Aut(\Db(X))^G$ and its action on $\Eq(\Db(X),\Db(Y))^{G}$ are given by convolution, and as above one sees that the action is free and transitive.
  It is clear that $\Eq(\DbG{G}(X),\DbG{G}(Y))$ is a pseudo-torsor under $\Aut(\DbG{G}(X))$ since equivalences of categories are invertible.
  
  The description of the actions above settle that the forgetful map is equivariant.
  The inflation map is equivariant since \cite[Lem.~5.(3)]{Plo:07} implies that $\inflation(\Pca,\phi)\star\inflation(\Pca',\phi')\simeq\inflation((\Pca,\phi)\star(\Pca',\phi'))$.
  The map $\delta_{X,Y}$ is defined in \cite[Lem.~1]{Plo:07}, where also the equality $\im(\forget)=\ker(\delta_{X,Y})$ is proven.
  The proof in \cite[Thm.~6.(2)]{Plo:07} that $\delta_X$ is a group homomorphism shows more generally that $\delta_{X,Y}$ is equivariant over $\delta_X$.
\end{proof}

\begin{remark}
  When $G=\Sn$ with $n\geq3$, then we have trivial center $\Center(\Sn)=0$, so the inflation map $\inflation_{\Diagonal \Sn}^{\Sn\times \Sn}$ is injective.
  Since $\Hom(\Sn,\Bbbk^{\times})=\{\id,\sgn\}$, a equivariant structure on an invariant object is unique up to the sign representation of $\Sn$.
\end{remark}

Putting everything together, we are ready to prove the main theorem for generalized Kummer fourfolds.
In \cref{sec:equivariant-semi-homogeneous} we explain how to treat generalized Kummer varieties of dimension $2m$ with $m$ even.

\begin{theorem}[Main \cref{thm:main_kum_two_derived_eq} Part 1]\label{thm:main-theorem-generalized-kummer-derived-equivalence}
  Let $A$ be an abelian variety over an algebraically closed field $\Bbbk$ of characteristic zero, which admits a polarization (or symmetric isogeny) $\lambda\colon A\to A^{\vee}$ of exponent $e$ such that $3\notdivide e$.
  Then the derived categories $\Db(\Kummer^{2}(A))$ and $\Db(\Kummer^{2}(A^{\vee}))$ of~$4$-dimensional generalized Kummer varieties are derived equivalent.
\end{theorem}


\begin{proof}
  We apply the method of Ploog recalled above.
  \Cref{thm:invariant-equivalence} instantiated with $n=3$ provides a $\Sn$-invariant derived equivalence 
  \[
    \FM_{\Pca}\colon\Db(A\otimesZ\StdRep_{3})\isoarr\Db(A^{\vee}\otimesZ\StdRep_{3}).
  \]
  Since the Schur multiplier $\Ho^2(\SymGrp_{3},\Bbbk^{\times})$ is zero, \cf\cref{prop:schur-multiplier}, we can enhance the Fourier--Mukai kernel $\Pca$ to an equivariant object $(\Pca,\phi)\in \DbG{\Diagonal\SymGrp_{3}}((A\otimesZ\StdRep_{3})\times(A^{\vee}\otimesZ\StdRep_{3}))$. 
  Now the inflation of $(\Pca,\phi)$ along the inclusion $\Diagonal\SymGrp_{3}\subset\SymGrp_{3}\times\SymGrp_{3}$ provides a kernel for an equivalence
  \[
    \DbG{\SymGrp_{3}}(A\otimesZ\StdRep_3)\isoarr\DbG{\SymGrp_{3}}(A^{\vee}\otimesZ\StdRep_3).
  \]
  Finally, as explained in the introduction, we have $\Kummer^{2}(A)\simeq\Hilb_{\SymGrp_{3}}(A\otimesZ\StdRep_3)$ by \cite{Hai:01}, and by the derived McKay-correspondence \cite{BKR:01} we know $\Db(\Hilb_{\SymGrp_{3}}(A\otimesZ\StdRep_3))\simeq\DbG{\SymGrp_{3}}(A\otimesZ\StdRep_3)$, and similarly for $A^{\vee}$ instead of $A$.
\end{proof}


\section{Equivariant semi-homogeneous vector bundles and Orlov's construction}
\label{sec:equivariant-semi-homogeneous}

In this section we study Orlov's construction demonstrating the surjectivity of the map $\gamma\colon\Eq(A\otimes\StdRep_n,A^\vee\otimes\StdRep_n) \to \Sp(A\otimes\StdRep_n,A^\vee\otimes\StdRep_n)$ from an equivariant perspective, with the goal to construct an equivariant Fourier--Mukai kernel out of an invariant symplectic isomorphism.
The construction makes use of Mukai's theory of semi-homogeneous vector bundles, which we will briefly recall below.

\begin{passage}[Orlov's construction {\cite[Constr.~4.10]{Orl:02}}]
  Let $A$ and $B$ be abelian varieties.
  The construction consists of the following steps:
  \begin{enumerate}[label=(\arabic*)]
    \item
    Let $f=\begin{psmallmatrix}f_{1}&f_{2}\\f_{3}&f_{4}\end{psmallmatrix}\in\Sp(A,B)$ be a symplectic isomorphism and assume that $f_2\colon A^\vee \to B$ is an isogeny.
    \item
    Denote by $f_{2}^{\inv}$ the inverse isogeny of $f_2$ with rational coefficients, and subsequently define the map $g\in\Hom(A\times B,A^\vee\times B^\vee)\otimesZ\QQ$ by \[g\coloneqq\begin{pmatrix}f_{2}^{\inv} \circ f_{1} & -f_{2}^{\inv} \\ -(f_{2}^{\inv})^{\vee} & f_{4} \circ f_{2}^{\inv}\end{pmatrix}.\]
    \item
    Since $f$ is symplectic, one can check that $g$ is symmetric, \ie $g=g^\vee\circ\eval$, where $\eval$ denotes the evaluation map which identifies an abelian variety with its double dual.
    Hence $g$ is contained in the image of the injection $\NS(A\times B)\otimes\QQ\hookrightarrow\Hom(A\times B,A^\vee\times B^\vee)\otimesZ\QQ$ which associates to a line bundle $\Lca$ the map $\varphi_\Lca(x)=\trm_{x}^{*}\Lca\otimes\Lca^\vee$.
    So $g$ corresponds to an element \[\mu\coloneqq [\Lca]\otimes\frac{1}{\ell} \in\NS(A\times B)\otimes\QQ.\]
    \item[(4--6)]
    Then Orlov takes a ``simple semi-homogeneous vector bundle $\Eca$ on $A\times B$ of slope $\mu$'' and shows that this is a kernel of a derived equivalence $\Db(A)\simeq\Db(B)$ with $\gamma(\Eca)=f$, \cf\cite[Prop.~4.11]{Orl:02}.
    To make sense of this and fill in the details, we explain next what a semi-homogeneous vector bundle is and how to construct them.
  \end{enumerate}
\end{passage}

We briefly exposit Mukai's theory of semi-homogeneous vector bundles and state a few facts about them that will be useful later on in this section.
See \cite{Muk:78} for reference and details.
Let $X$ be an abelian variety.

\begin{definition}
  A vector bundle $\Eca$ on $X$ is called \emph{semi-homogeneous} if for every $x\in X$ there exists some line bundle $\Lca_x\in\Pic^0(X)$ such that $\trm_x^{*}\Eca\simeq\Eca\otimes\Lca_x$.
\end{definition}

\begin{definition}
  \begin{enumerate}
    \item
    Let $\Eca$ be a vector bundle on $X$.
    Define the \emph{slope} $\mu(\Eca)$ of $\Eca$ as \[\mu(\Eca)\coloneqq [\det(\Eca)]\otimes\frac{1}{\rank(\Eca)} \in\NS(X)\otimes\QQ.\]
    \item
    Given an element $\mu=[\Lca]\otimes\frac{1}{\ell}\in\NS(X)\otimes\QQ$, define \[\mathrm{\Phi}_\mu\coloneqq \im((\ell,\varphi_\Lca)\colon X\to X\times X^\vee),\]
    and denote the projection onto the first factor by $\pr_1\colon\mathrm{\Phi}_\mu\to X$.
    Following Mukai, we denote the kernel of $\pr_1$ by $\SigmaMukai(\mu)\coloneqq\ker(\pr_1)$.\footnotemark
    \footnotetext{We hope this doesn't lead to confusion with the summation map of the previous sections, which will not feature in this section.}
  \end{enumerate}
\end{definition}

\begin{remark}
  By \cite[Lem.~4.9]{Orl:02}, if $\Eca$ is a simple semi-homogeneous vector bundle of slope $\mu$, then $(a,\alpha)\in\mathrm{\Phi}_\mu$ \iff $\trm_a^{*}\Eca\simeq\Eca\otimes\Pca_\alpha$, where $\Pca\in\Pic(X\times X^\vee)$ is the Poincaré bundle.
\end{remark}

\begin{proposition}\label{prop:mukai-semi-homogeneous-facts}
  Let $\mu\in\NS(X)\otimes\QQ$ be given.
  \begin{enumerate}
    \item
    There exists a simple semi-homogeneous vector bundle $\Eca$ on $X$ of slope $\mu$.
    \item
    Every other such $\Eca'$ is of the form $\Eca'\simeq\Eca\otimes\Mca$ for some line bundle $\Mca\in\Pic^0(X)$.
    \item
    The rank of $\Eca$ can be computed using $\rank(\Eca)^2=\deg(\pr_1)$.
  \end{enumerate}
\end{proposition}

\begin{proof}
  \Cf\cite[Thm.~7.11]{Muk:78}.
  Part \rom{1} uses the constructions of the following proposition.
\end{proof}

\begin{proposition}\label{prop:mukai-semi-homog-construction}
  Let $\mu=[\Lca]\otimes\frac{1}{\ell}\in\NS(X)\otimes\QQ$ with $\ell>0$.
  \begin{enumerate}
    \item
    The sheaf $\Fca\coloneqq [\ell]_{*}(\Lca^{\otimes\ell})$ is a semi-homogeneous vector bundle on $X$ of slope $\mu(\Fca)=\mu$.
    \item
    Let $\Fca$ be a semi-homogeneous vector bundle on $X$ of slope $\mu(\Fca)=\mu$, then there exists a Jordan--Hölder filtration \[0=\Fca_0\subset\Fca_1\subset\dots\subset\Fca_k=\Fca\] such that each $\Eca_i=\Fca_i/\Fca_{i-1}$ is a simple semi-homogeneous vector bundle of slope $\mu(\Eca_i)=\mu$.
    Such a filtration is not necessarily unique but the associated multiset of grades pieces $\{\!\{\Eca_i\}\!\}$ is unique.
  \end{enumerate}
\end{proposition}

\begin{proof}
  \Cf\cite[Prop.~6.22, Prop.~6.15]{Muk:78}.
\end{proof}

\begin{remark}
  In fact every simple semi-homogeneous vector bundle arises as the push-forward of a line bundle under a suitable isogeny $\pi\colon Y\to X$, \cf\cite[Thm.~5.8]{Muk:78}.
\end{remark}

\begin{passage}[Continuation of Orlov's construction]
  We fill in the details of step (4) to (6) using \cref{prop:mukai-semi-homog-construction}:
  \begin{enumerate}[label=(\arabic*),start=4]
    \item
    The sheaf $\Fca\coloneqq [\ell]_{*}(\Lca^{\otimes\ell})$ is a semi-homogeneous vector bundle on $A\times B$ of slope $\mu(\Fca)=\mu$.
    \item
    Consider a Jordan--Hölder filtration $0=\Fca_0\subset\dots\subset\Fca_k=\Fca$, where each $\Eca_i=\Fca_i/\Fca_{i-1}$ is a \emph{simple} semi-homogeneous vector bundle of slope $\mu(\Eca_i)=\mu$.
    \item
    Take any of the vector bundles $\Eca_i$ as the kernel of a Fourier--Mukai functor \[\FM_{\Eca_i}\colon\Db(A)\to\Db(B),\]
    which, by Orlov, provides a desired derived equivalence mapping under $\gamma$ to $f$.
  \end{enumerate}
\end{passage}

Now we specialize our situation to the equivariant setting of the previous sections, that is, we let $A_0$ be an abelian surface and take $A=A_0\otimes\StdRep_n$ and $B=A_0^\vee\otimes\StdRep_n$ with their $\Sn$-action.
We follow the steps of Orlov's and Mukai's constructions and wish to construct on the way a suitable equivariant structure which we can carry through the argument to the final Fourier--Mukai kernel.

\begin{passage}\label{par:construction-mu}
  Ad (1):
  Let $\lambda\colon A\to A^{\vee}$ be a polarization of exponent $\erm(\lambda)$.
  Pick integers $n_3$ and $n_4$ which solve the equation $n_{4}\erm(\lambda)-n_{3}n=1$; this can be done whenever $\gcd(\erm(\lambda),n)=1$.
  Now recall the maps $\phi_0$ and $\dualisog{\phi}_0$ from \cref{sec:standard_rep} and consider the element
  \begin{equation*}\renewcommand\arraystretch{1.4}
    f\coloneqq\begin{pmatrix}\lambda & \dualisog{\phi}_0 \\ n_{3}\phi_0 & n_{4}\dualpol{\lambda}\end{pmatrix} \in\Sp(A_0\otimes\StdRep_n,A_0^\vee\otimes\StdRep_n).
  \end{equation*}
  Note that indeed $\dualisog{\phi}_0\coloneqq\id\otimes\dualisog{\phi}_0\colon A_0^\vee\otimes\StdRep_n^\vee \to A_0^\vee\otimes\StdRep_n$ is an isogeny.
  
  Ad (2):
  Recall that by construction we have $(\dualisog{\phi}_0)^\inv = \frac{1}{n}\phi_0 \in\Hom(A_0^\vee\otimes\StdRep_n,A_0^\vee\otimes\StdRep_n^\vee)\otimes\QQ$, and that $\phi_0^\vee = \phi_0$, so we get
  \begin{equation*}\renewcommand\arraystretch{1.6}
    g
    =\begin{pmatrix}\dualisog{\phi}_0^\inv\circ\lambda & -\dualisog{\phi}_0^\inv \\ -(\dualisog{\phi}_0^\inv)^\vee & n_{4}\dualpol{\lambda}\circ (\dualisog{\phi}_0)^\inv\end{pmatrix}
    =\begin{pmatrix}\frac{1}{n}\lambda\otimes\phi_0 & -\frac{1}{n}\id\otimes\phi_0 \\ -\frac{1}{n}\id\otimes\phi_0 & \frac{n_4}{n}\dualpol{\lambda}\otimes\phi_0\end{pmatrix}
    =\begin{pmatrix}\lambda & -\id \\ -\id & n_{4}\dualpol{\lambda}\end{pmatrix} \otimes\frac{1}{n}\phi_0.
  \end{equation*}

  Ad (3):
  The symmetric map $g$ corresponds to $\mu\coloneqq \frac{1}{n} [\Lca]$ where $\Lca\in\Pic((A_0\otimes\StdRep_n)\times(A_0^\vee\otimes\StdRep_n))$ is a line bundle with $\varphi_\Lca=n g$.
  The next proposition applied to $X=A_0\times A_0^\vee$ shows that we can take \[\Lca=\restr{\Lca_{0}^{\boxtimes n}}{(A_0\times A_0^\vee)\otimes\StdRep_n},\] where $\Lca_0\in\Pic(A_0\times A_0^\vee)$ satisfies $\varphi_{\Lca_0}=\begin{psmallmatrix}\lambda & -\id \\ -\id & n_{4}\dualpol{\lambda}\end{psmallmatrix}$.
  In particular, $\Lca$ becomes an $\Sn$-equivariant line bundle since the box-product $\Lca_0^{\boxtimes n} = \pr_{1}^{*}\Lca_{0}\otimes\dots\otimes\pr_{n}^{*}\Lca_{0} \in\Pic((A_0\times A_0^\vee)^{\times n})$ carries a canonical $\Sn$-equivariant structure afforded by the commutativity of the tensor product.
\end{passage}

\begin{proposition}
  Let $X$ be an abelian variety and let $\Lca_{0}\in\Pic(X)$ be a line bundle.
  Then we have the equality
  \[
    \varphi_{(\Lca_{0}^{\boxtimes n}|{X\otimes\StdRep_n})} = \varphi_{\Lca_{0}}\otimes\phi_0 \colon X\otimes\StdRep_n\to X^{\vee}\otimes\StdRep_n^\vee.
  \]
\end{proposition}

\begin{proof}
  Consider the following diagram:
  \[\begin{tikzcd}[column sep=large]
    X\otimes\StdRep_{n} \arrow[d, hook', "i"] \arrow[r, "\varphi_{i^{*}(\Lca_{0}^{\boxtimes n})}"]                              & (X\otimes\StdRep_{n})^{\vee} \arrow[r, "\sim"] & X^{\vee}\otimes\StdRep_{n}^{\vee} \\
    X^n \arrow[r, "\varphi_{\Lca_{0}^{\boxtimes n}}"] \arrow[rr, "\varphi_{\Lca_{0}}\times\cdots\times\varphi_{\Lca_{0}}"', bend right=20] & (X^n)^{\vee} \arrow[u, "i^{\vee}"] \arrow[r, "\sim"] & (X^{\vee})^n \arrow[u, two heads]
  \end{tikzcd}\]
  The left square commutes by \cite[Prop.~7.6]{EGM}; for the right square and lower triangle see \cref{sec:diag-is-dual-to-sum} and \cref{prop:abvar-dual-of-product}, which explain in particular that $\Sigma^\vee=\Delta$ and $i^\vee$ becomes the quotient projection in the definition of $\StdRep_n^\vee$.
  
  It is clear that the following square commutes:
  \[\begin{tikzcd}
    X\otimes\StdRep_{n} \arrow[d, "i", hook'] \arrow[r, "\varphi_{\Lca_{0}}\otimes\id"] & X^{\vee}\otimes\StdRep_{n} \arrow[d, hook'] \\
    X^n \arrow[r, "\varphi_{\Lca_{0}}\times\cdots\times\varphi_{\Lca_{0}}"'] & (X^{\vee})^n                              
  \end{tikzcd}\]
  Finally, the composition $X^{\vee}\otimes\StdRep_{n}\hookrightarrow(X^n)^{\vee}\twoheadrightarrow X^{\vee}\otimes\StdRep_{n}^{\vee}$ equals $\id\otimes\phi_0 \colon X^{\vee}\otimes\StdRep_{n}\to X^{\vee}\otimes\StdRep_{n}^{\vee}$ by the definition of $\phi_0$.
  Putting these facts together yields the claimed result.
\end{proof}

\begin{passage}
  Ad (4):
  Now we consider the semi-homogeneous vector bundle \[\Fca\coloneqq [n]_{*}\Lca^{\otimes n}.\]
  Note that the $\Sn$-equivariant structure on $\Lca$ induces one on $\Lca^{\otimes n}$.
  Also $\Fca$ inherits an $\Sn$-equivariant structure from~$\Lca^{\otimes n}$, since the push-forward along an $\Sn$-equivariant morphism $F$ of a sheave $\Gca$ with $\Sn$-equivariant structure $\phi$ also carries a natural $\Sn$-equivariant structure defined by $(F_{*}\phi_\sigma)\colon F_{*}\Gca\to F_{*}\sigma^{*}\Gca\simeq \sigma^{*}F_{*}\Gca$, for $\sigma\in\Sn$.
\end{passage}


\begin{passage}
  Now we are dealing with an $\Sn$-equivariant semi-homogeneous vector bundle $\Fca$ of slope~$\mu$, but we have the further requirement that it should be simple.
  We saw above that we should consider a Jordan--Hölder filtration of $\Fca$ and pick one of the graded pieces.
  But we face the problem that the equivariant structure of $\Fca$ does not readily restrict to one of its graded pieces.
  Thus a more precise study of $\Fca$ and its Jordan--Hölder filtrations is needed:
  
  We want to show that $\Fca$ admits a Jordan--Hölder filtration which is split.
  We can rearrange by \cite[Prop.~6.18]{Muk:78} a Jordan--Hölder filtration into the form
  \[
    \Fca\simeq\bigoplus_{j\in J}\Uca_{j}\otimes\Eca_{j}
  \]
  where the simple semi-homogeneous vector bundles $\Eca_j$ of slope $\mu$ are pairwise distinct, and the~$\Uca_j$ are unipotent vector bundles in the following sense.
\end{passage}

\begin{definition}
  A vector bundle $\Uca$ on a scheme $X$ is called \emph{unipotent} of it admits an increasing filtration $0=\Uca_{0}\subset\Uca_{1}\subset\dots\subset\Uca_{r}=\Uca$ such that $\Uca_{i}/\Uca_{i-1}\simeq\Oca_X$ for every $i=1,\dots,r$.
\end{definition}

\begin{proposition}\label{prop:unipotent-vb-splitting-criterion}
  Let $\Uca$ be a unipotent vector bundle on a variety $X$, and set $r=\rank(\Uca)$, then
  \begin{enumerate}
    \item
    $\dim\End(\Uca)\leq r^2$, and
    \item
    $\dim\End(\Uca) = r^2$ implies that $\Uca$ is split: $\Uca\simeq\Oca_X^{\oplus r}$.
  \end{enumerate}
\end{proposition}

\begin{proof}
  This can be proved inductively and the resulting diagram chase is left to the reader.
\end{proof}

\begin{situation}
  From now on we use the abbreviation $X\coloneqq (A_0\otimes\StdRep_n)\times(A_0^\vee\otimes\StdRep_n)$.
  So we have $\dim(X)=4(n-1)$.
  Recall that for an abelian variety $X$ of dimension $g$ we have $\#X[n]=n^{2g}$.
\end{situation}

\begin{passage}\label{sec:jh-filtration-action}
  Let $\alpha\in X^{\vee}[n]$ and let $\Pca_\alpha\in\Pic^0(X)$ be the associated line bundle.
  We have by the projection formula and the equality $[n]^{*}\Pca_\alpha \simeq \Pca_{n\alpha} \simeq \Pca_0 \simeq \Oca_X$ that
  \[
    \Fca\otimes\Pca_\alpha
    = ([n]_{*}\Lca^{\otimes n})\otimes\Pca_\alpha
    \simeq [n]_{*}(\Lca^{\otimes n}\otimes [n]^{*}\Pca_\alpha)
    \simeq [n]_{*}(\Lca^{\otimes n}) = \Fca.
  \]
  This means that $X^{\vee}[n]$ acts on the set (of isomorphism classes) $\{\Eca_j\}_{j\in J}$ via $\alpha.\Eca_j\coloneqq\Eca_{j}\otimes\Pca_\alpha$.
  (Recall that the multiset of associated graded pieces of a Jordan--Hölder filtration is unique.)
  
  Next we compute the stabilizers of this action.
  We can write each $\Eca_j$ as $\Eca_{0}\otimes\Pca_{\alpha_j}$ for some $\alpha_j\in X^{\vee}$ by \cref{prop:mukai-semi-homogeneous-facts}, so
  \[
    \Stab(\Eca_j)
    =\Stab(\Eca_{0}\otimes\Pca_{\alpha_j})
    =\Stab(\Eca_0)
    =\{\alpha\in X^{\vee}[n] \mid \Eca_{0}\simeq\Eca_{0}\otimes\Pca_{\alpha}\}.
  \]
  By \cite[§3, Prop.~7.7]{Muk:78} this is nothing else than $\SigmaMukai(\mu)\cap X^{\vee}[n]=\SigmaMukai(\mu)$.
\end{passage}

\begin{proposition}\label{prop:jh-filtration-split}
  The semi-homogeneous vector bundle $\Fca$ of slope $\mu$ constructed above admits a split Jordan--Hölder filtration, explicitly:
  \[
    \Fca\simeq\bigoplus_{\alpha\in X^{\vee}[n]/\SigmaMukai(\mu)}(\Eca_{0}\otimes\Pca_{\alpha})^{\oplus n^{2n-4}},
  \]
  where $\Eca_0$ is a simple semi-homogeneous vector bundle of slope $\mu$.
\end{proposition}

Before coming to the proof of this central proposition, we need to calculate the rank of the vector bundles $\Eca_j$.

\begin{proposition}\label{prop:simple-semi-homogeneous-rank-calculation}
  Let $\Eca$ be a simple semi-homogeneous vector bundle on $X$ of slope $\mu$ as constructed above in \cref{par:construction-mu}, then the rank of $\Eca$ is $\rank(\Eca)=n^{2n-4}$, and $\#\SigmaMukai(\mu)=n^{4n-8}$.
\end{proposition}

\begin{proof}
  By \cref{prop:mukai-semi-homogeneous-facts} we have $\rank(\Eca)^2=\deg(\pr_1)$.
  Now from the equation
  \[
    \SigmaMukai(\mu)=\ker(\pr_1)=\{(a,\alpha)\in\mathrm{\Phi}_\mu \mid a=0\} = \{(n x,\varphi_{\Lca}(x)) \mid n x=0, x\in X\}
  \]
  we get a short exact sequence
  \[
    0\to X[n]\cap\ker(\varphi_{\Lca})\to X[n]\xrightarrow{\varphi_{\Lca}}\SigmaMukai(\mu)\to 0,
  \]
  and $\rank(\Eca)^2 = \#\varphi_{\Lca}(X[n])$.
  
  Regarding the kernel, let $(a,\alpha)\in X[n]=((A_0\otimes\StdRep_n)\times(A_0^{\vee}\otimes\StdRep_n))[n]$, then the condition $\varphi_{\Lca}(a,\alpha)=0$ becomes by $\varphi_{\Lca}=n g$ and the definition of $g$:
  \begin{equation*}\begin{cases}
    \phi_0(\lambda(a)) = \phi_0(\alpha) \\
    \phi_0(a) = n_{4} \phi_{0}(\dualpol{\lambda}(\alpha))
  \end{cases}\end{equation*}
  Here we have, by abuse of notation, implicitly applied the maps $\lambda$ and $\dualpol{\lambda}$ entry-wise to tuples.
  Define $\alpha'\coloneqq\lambda(a)$.
  Since $\ker(\phi_0|A_0^{\vee}\otimes\StdRep_n)=\Diagonal(A_0^{\vee}[n])$ we can write by the first equation $\alpha=\alpha'+\Diagonal(\alpha_0)$ for some $\alpha_0\in A_0^{\vee}[n]$.
  Similarly, we see after substituting into the second equation that $n_{4}\dualpol{\lambda}(\alpha'+\Diagonal(\alpha_0)) = a+\Diagonal(a_0)$ for some $a_0\in A_0[n]$.
  After substituting $\alpha'=\lambda(a)$ the left hand side of this equals
  $
    n_{4}\erm(\lambda)a + n_{4}\dualpol{\lambda}(\Diagonal(\alpha_0)) = n_{4}\erm(\lambda)a + \Diagonal(n_{4}\dualpol{\lambda}(\alpha_0)).
  $
  Using $n_{4}\erm(\lambda)-1=n_{3} n$ we arrive at $0=n_{3}n a=\Diagonal(a_{0}-n_{4}\dualpol{\lambda}(\alpha_0))$, which always admits a solution $a_0\in A_0[n]$.
  We conclude that
  \[
    \ker(\varphi_{\Lca})\cap X[n] = (A_0\otimes\StdRep_n)[n]\times\Diagonal(A_0^{\vee}[n]).
  \]
  
  We can now calculate the cardinalities $\#(\ker(\varphi_{\Lca})\cap X[n]) = n^{4(n-1)}\cdot n^{4} = n^{4n}$, so that we finally get $\#\varphi_{\Lca}(X[n]) = n^{8(n-1)}/n^{4n} = n^{4n-8}$ and $\rank(\Eca)=n^{2n-4}$.
\end{proof}

\begin{remark}\label{rmk:jh-filtration-orbit-count}
  Note that $\rank(\Fca)=\rank([n]_{*}\Lca^{\otimes n})=\deg([n]\colon X\to X)=n^{8(n-1)}$.
  In particular, we see that our $\Fca$ is definitely not simple.
  Instead we see that the length $N$ of any Jordan--Hölder filtration of $\Fca$ will be $N=n^{8(n-1)}/n^{2n-4}=n^{6n-4}$.
  Further we see that each orbit of the action in \cref{sec:jh-filtration-action} contains exactly $n^{4n}=n^{8(n-1)}/n^{4n-8}$ elements.
\end{remark}

\begin{proof}[Proof of \cref{prop:jh-filtration-split}]
  We aim to apply the splitting criterion for unipotent vector bundles above, \cf\cref{prop:unipotent-vb-splitting-criterion}.
  For this we calculate $\dim\End(\Fca)$ in two ways:
  Firstly we abbreviate $\widetilde{\Lca}\coloneqq\Lca^{\otimes n}$ where $\Lca$ is as in the construction of $\Fca$.
  Then we have
  \[
    \Hom(\Fca,\Fca)
    \simeq \Hom([n]_{*}\widetilde{\Lca},[n]_{*}\widetilde{\Lca})
    \simeq \Hom([n]^{*}[n]_{*}\widetilde{\Lca},\widetilde{\Lca})
    \simeq \Hom(\bigoplus_{x\in X[n]}\trm_x^{*}\widetilde{\Lca},\widetilde{\Lca}),
  \]
  and for $x\in X[n]$
  \[
    \Hom(\trm_x^{*}\widetilde{\Lca},\widetilde{\Lca})
    \simeq \Hom(\widetilde{\Lca}\otimes\Pca_{\varphi_{\widetilde{\Lca}}(x)},\widetilde{\Lca})
    \simeq \Hom(\Pca_{\varphi_{\widetilde{\Lca}}(x)},\Oca_X)
    \simeq \Ho^0(X,\Oca_X)
    \simeq \Bbbk,
  \]
  since $\varphi_{\widetilde{\Lca}}=n\varphi_{\Lca}$ implies $\varphi_{\widetilde{\Lca}}(x)=\varphi_{\Lca}(n x)=\varphi_{\Lca}(0)=0$.
  We conclude that \[\dim\End(\Fca)=\#X[n]=n^{8(n-1)}.\]
  
  Secondly we take on the viewpoint that $\Fca\simeq\bigoplus \Uca_{j}\otimes\Eca_{j}$.
  Using $\Hom(\Eca_{j},\Eca_{j'})=0$ for $j\neq j'$ and $\Hom(\Eca_{j},\Eca_{j})=\Bbbk$, \cf\cite[Prop.~6.17]{Muk:78}, we get that \[\dim\End(\Fca)=\sum_{j\in J}\dim\End(\Uca_j).\]
  Define $r_j\coloneqq\rank(\Uca_j)$ and keep \cref{rmk:jh-filtration-orbit-count} in the following calculations in mind.
  By \cref{prop:simple-semi-homogeneous-rank-calculation} we have $n^{8(n-1)}=\rank(\Fca)=\sum_{j}r_{j}\rank(\Eca_j)=n^{2n-4}\sum_{j}r_{j}$, which implies that the $r_j$'s give a partition of $N=n^{6n-4}$.
  Now we want to take the action from \cref{sec:jh-filtration-action} into account.
  Pick from each orbit a representative $\Eca_j$ and denote the set of indices of these elements by $J_0\subset J$.
  Thus the $r_j$'s for $j\in J_0$ constitute a partition of \[N/(\#\text{elements in an orbit})=n^{6n-4}/n^{4n}=n^{2n-4}.\]
  Making the abbreviation $e_j\coloneqq\dim\End(\Uca_j)$, we get from the considerations above
  \[
    \dim\End(\Fca) = n^{4n}\sum_{j\in J_0}e_j
  \]
  by \cref{rmk:jh-filtration-orbit-count} and the fact that the unipotent bundles in an orbit must be isomorphic, \cf~\cite[Prop.~6.2]{Muk:78}.
  
  Finally, comparing both dimension calculations and using \cref{prop:unipotent-vb-splitting-criterion}, we calculate that 
  \[
    n^{4n-8} = \sum_{j\in J_0}e_j \leq \sum_{j\in J_0}(r_j)^2 \leq (\sum_{j\in J_0}r_j)^2 = (n^{2n-4})^2.
  \]
  We see that the inequalities have to be equalities, which forces $J_0$ to be a singleton, say for notation $J_0=\{0\}$, and also $e_0=r_0^2$.
  By \cref{prop:unipotent-vb-splitting-criterion} this means that $\Uca_0\simeq\Oca_X^{n^{2n-4}}$ as desired.
  Note that $J_0=\{0\}$ means that the action of \cref{sec:jh-filtration-action} is transitive, since $J_0$ indexes by definition the set of orbits.
\end{proof}

\begin{passage}
  Coming back to the $\Sn$-action on $X$, we already saw that $\Fca$ is $\Sn$-invariant, so pullback along $\sigma\in\Sn$ permutes the graded pieces $\Eca_j$ of a Jordan--Hölder filtration of $\Fca$.
  Taking \cref{prop:jh-filtration-split} into account, we see that for each $\sigma\in\Sn$ there is a unique $\alpha(\sigma)\in X^{\vee}[n]/\SigmaMukai(\mu)$ such that $\sigma^{*}\Eca_0\simeq\Eca_{0}\otimes\Pca_{\alpha(\sigma)}$.
\end{passage}

\begin{proposition}\label{prop:jh-filtration-invariant-summand}
  When $n\geq5$ is odd there exists $\alpha_0\in X^{\vee}[n]/\SigmaMukai(\mu)$ such that $\Eca_{0}\otimes\Pca_{\alpha_0}$ is $\Sn$-invariant.
  More precisely,
  \begin{enumerate}
    \item
    the map $\alpha\colon\Sn\to X^{\vee}[n]/\SigmaMukai(\mu)$ is a crossed homomorphism for the (right) action of $\Sn$ inherited from $X^{\vee}$, and
    \item
    when $n\geq5$ is odd, we have that $\alpha$ is a principal crossed homomorphism.
  \end{enumerate}
\end{proposition}

\begin{proof}
  \itemnum
  Let $\sigma,\tau\in\Sn$, by calculating
  \begin{multline*}
    \Eca_{0}\otimes\Pca_{\alpha(\sigma\tau)} \simeq (\sigma\tau)^{*}\Eca_0 \simeq \tau^{*}\sigma^{*}\Eca_0 \simeq \tau^{*}(\Eca_0\otimes\Pca_{\alpha(\sigma)}) \simeq \Eca_0\otimes\Pca_{\alpha(\tau)}\otimes\tau^{*}\Pca_{\alpha(\sigma)} \simeq \Eca_0\otimes\Pca_{\alpha(\tau)+\tau^{\vee}(\alpha(\sigma))}
  \end{multline*}
  we see that $\alpha(\sigma\tau) = \alpha(\tau)+\tau^{\vee}(\alpha(\sigma))$, so $\alpha$ is indeed a crossed homomorphism.
  
  \itemnum
  We know that $\SigmaMukai(\mu) = \varphi_{\Lca}(X[n]) = (\varphi_{\Lca_0}\otimes\phi_0)(X[n])$, and by definition $X=(A_0\times A_0^\vee)\otimes\StdRep_n$, so we have a cokernel sequence
  \[
    (A_0\times A_0^\vee)[n]\otimes\StdRep_n \xrightarrow{\varphi_{\Lca_0}\otimes\phi_0} (A_0\times A_0^\vee)^{\vee}[n]\otimes\StdRep_n^{\vee} \to X^{\vee}[n]/\SigmaMukai(\mu) \to 0.
  \]
  Define $\SigmaMukai_0\coloneqq\varphi_{\Lca_0}((A_0\times A_0^\vee)[n])$, then we get the exact sequence
  \[
    0 \to \SigmaMukai_0 \xrightarrow{\Diagonal} \SigmaMukai_0\otimes\StdRep_n \xrightarrow{\id\otimes\phi_0} (A_0\times A_0^\vee)^{\vee}[n]\otimes\StdRep_n^{\vee} \to X^{\vee}[n]/\SigmaMukai(\mu) \to 0,
  \]
  where we used for the leftmost term that $\ker(\id\otimes\phi_0)=\{(a,\dots,a)\mid n a=\sum_{i=1}^{n}a=0\}$ and $\SigmaMukai_0$ consists of $n$-torsion elements.
  Applying group cohomology, we get
  \begin{multline*}
     \Ho^1(\Sn,(A_0\times A_0^\vee)^{\vee}[n]\otimes\StdRep_n^{\vee})
    \to \Ho^1(\Sn,X^{\vee}[n]/\SigmaMukai(\mu)) \\
    \to \Ho^2(\Sn,(\SigmaMukai_0\otimes\StdRep_n)/\SigmaMukai_0)
    \to \Ho^2(\Sn,(A_0\times A_0^\vee)^{\vee}[n]\otimes\StdRep_n^{\vee}),
  \end{multline*}
  where the first term is zero by \cref{prop:h1-dual-stdrep-vanishing}, and the last term is zero for $n\geq7$ by \cref{prop:stdrep_dual_coho_vanishing_and_seq} and for $n=5$ by \cref{prop:h2-s5-vanishing} below.
  Next, using that $\SigmaMukai_0$ is $n$-torsion, we can consider the following sequence, which we already encountered in the proof of \cref{prop:connecting-map-analysis-n4},
  \[
    0 \to \SigmaMukai_0 \xrightarrow{\Diagonal} \SigmaMukai_0\otimes\StdRep_n \xrightarrow{\id\otimes\phi_0} \SigmaMukai_0\otimes\StdRep_n^{\vee} \xrightarrow{\text{sum}} \SigmaMukai_0 \to 0
  \]
  and apply group cohomology to get
  \begin{equation*}
    \Ho^1(\Sn,\SigmaMukai_0\otimes\StdRep_n^{\vee})
    \to \Ho^1(\Sn,\SigmaMukai_0)
    \to \Ho^2(\Sn,(\SigmaMukai_0\otimes\StdRep_n)/\SigmaMukai_0)
    \to \Ho^2(\Sn,\SigmaMukai_0\otimes\StdRep_n^{\vee}),
  \end{equation*}
  where the first term is zero for $n\geq5$ by \cref{prop:h1-dual-stdrep-vanishing}, and the last term is zero for $n\geq7$ by \cref{prop:stdrep_dual_coho_vanishing_and_seq} and for $n=5$ by \cref{prop:h2-s5-vanishing} below.
  
  Finally, using that $\SigmaMukai_0$ is abelian and has no $2$-torsion since $n$ is odd, we see that
  \[
    \Ho^1(\Sn,\SigmaMukai_0)
    \simeq \Hom(\Sn,\SigmaMukai_0)
    \simeq \Hom(\ZZ/2\ZZ,\SigmaMukai_0)
    = 0.
  \]
  In conclusion, $\Ho^1(\Sn,X^{\vee}[n]/\SigmaMukai(\mu)) = 0$, so the crossed homomorphism $\alpha\colon\Sn\to X^{\vee}[n]/\SigmaMukai(\mu)$ must be principal, \ie there exists $\alpha_0\in X^{\vee}[n]/\SigmaMukai(\mu)$ such that $\alpha(\sigma)=\sigma^{\vee}(\alpha_0)-\alpha_0$ for each $\sigma\in\Sn$.
  This means by construction of $\alpha$ that $\Eca_0\otimes\Pca_{-\alpha_0}$ is $\Sn$-invariant.
\end{proof}

\begin{remark}
  The concrete definition of $\varphi_{\Lca_0}$ above lets one calculate that $\ker(\varphi_{\Lca_0})\simeq A_0[\erm(\lambda)-1]$.
  So if in addition to $\gcd(\erm(\lambda),n)=1$ also $\gcd(\erm(\lambda)-1,n)=1$ holds, we have that the homomorphism $\varphi_{\Lca_0}\colon(A_0\times A_0^{\vee})[n]\to(A_0\times A_0^{\vee})^{\vee}[n]$ is an isomorphism.
  Then we can see that
  \[
    X^{\vee}[n]/\SigmaMukai(\mu) = ((A_0\times A_0^{\vee})^{\vee}[n]\otimes\StdRep_{n}^{\vee})/(\varphi_{\Lca_0}\otimes\phi_0)((A_0\times A_0^{\vee})[n]\otimes\StdRep_n)
  \]
  actually carries a trivial $\Sn$-action, since $\Sn$ acts trivially on $\StdRep_{n}^{\vee}/\phi_{0}(\StdRep_n)$ by \cref{par:stdrep_canonical_map}.
  This renders most of the proof of \rom{2} in the previous proposition unnecessary.
\end{remark}

\begin{proposition}\label{prop:h2-s5-vanishing}
  We have $\Ho^2(\SymGrp_5,A[5]\otimes\StdRep_5^\vee)=0$ and $\Ho^2(\SymGrp_5,\SigmaMukai_0\otimes\StdRep_5^\vee)=0$ where $A$ is an abelian variety and $\SigmaMukai_0$ is the group from above.
\end{proposition}

\begin{proof}
  Both claims will follow from the vanishing of $\Ho^3(\SymGrp_5,\ZZ/5\ZZ)$, which is classically known and can for example be deduced using the sequence $0\to\ZZ\xrightarrow{\cdot 5}\ZZ\to\ZZ/5\ZZ\to0$ from the group cohomology $\Ho^{\bullet}(\SymGrp_5,\ZZ)$ of the symmetric group.
  
  By \cref{prop:stdrep_dual_coho_vanishing_and_seq} we know that $\Ho^2(\SymGrp_5,A[5]\otimes\StdRep_5^\vee)$ injects into $\Ho^3(\SymGrp_5,A[5])$, but $A[5]$ is isomorphic to a direct sum of copies of $\ZZ/5\ZZ$.
  Similarly, the group $\SigmaMukai_0$ is a finite abelian group which is $5$-torsion, so it is also isomorphic to a direct sum of copies of $\ZZ/5\ZZ$, so we can again conclude by \cref{prop:stdrep_dual_coho_vanishing_and_seq}.
\end{proof}

\begin{passage}\label{par:summary-split-equivariant-semi-homogenous-bundle}
  To summarize so far, we have an $\Sn$-equivariant sheaf $\Fca$ which is the direct sum of simple semi-homogeneous vector bundles, each of which provides the kernel of a desired derived equivalence:
  \[
    \Fca\simeq\bigoplus_{\alpha}(\Eca_0\otimes\Pca_\alpha)^{\oplus r},
  \]
  where $\alpha\in X^{\vee}[n]/\SigmaMukai(\mu)$ and $r=n^{2n-4}$, \cf\cref{prop:jh-filtration-split}.
  For $n\geq5$ odd, at least one of the summands must be $\Sn$-invariant by \cref{prop:jh-filtration-invariant-summand}, say $\Eca_0$ for notation.
  
  Since the $\Eca_0\otimes\Pca_\alpha$ are pairwise non-isomorphic, there are no non-zero homomorphisms between them, so the $\Sn$-equivariant structure on $\Fca$ restricts to one on $\Eca_{0}^{\oplus r}$.
  Also recall that $r=n^{2n-4}$ is odd since $n$ is odd.
  Next we study such a situation more generally.
\end{passage}


\begin{proposition}\label{prop:direct-sum-equivariant}
  Let $G$ be a finite group acting on a variety $X$, and let $\Eca\in\CatCoh(X)$ be a sheaf on $X$.
  Assume that
  \begin{enumerate}
    \item
    $\Eca$ is simple, \ie $\End(\Eca)=\Bbbk$,
    \item 
    $\Eca$ is $G$-invariant, \ie $g^{*}\Eca\simeq\Eca$ for every $g\in G$, and that
    \item 
    $\Eca^{\oplus r}$ carries a $G$-equivariant structure for some integer $r\in\NN$.
  \end{enumerate}
  If every projective representation of $G$ of dimension $r$ is isomorphic to a linear representation of $G$, then $\Eca$ itself carries a $G$-equivariant structure.
\end{proposition}

\begin{proof}
  As a motivation, apply group cohomology (with trivial actions) to the sequence
  \[
    1 \to \Bbbk^{\times} \to \GL(r,\Bbbk) \to \PGL(r,\Bbbk) \to 1
  \]
  to get the exact sequence
  \[
    \Ho^1(G,\GL(r,\Bbbk))
    \to \Ho^1(G,\PGL(r,\Bbbk))
    \to \Ho^2(G,\Bbbk^{\times})
    \xrightarrow{\Diagonal} ``\Ho^2(G,\GL(r,\Bbbk))".
  \]
  The assumption concerning projective representations means that the first map is surjective, so the map $\Diagonal$ is injective.
  But $\Diagonal$ maps the class which obstructs the existence of an equivariant structure on $\Eca$ to the obstruction class for $\Eca^{\oplus r}$, which is trivial by assumption.
  
  But to make sense of this, we need to make sense of the $\Ho^2(G,\GL(r,\Bbbk))$ term.
  Instead of going through the trouble of explaining a suitable non-abelian second group cohomology construction, we instead mimic a concrete cocycle based approach from abelian group cohomology:
  
  Since $\Eca$ is $G$-invariant, we can pick isomorphisms $\lambda_g\colon\Eca\isoarr g^{*}\Eca$ for each $g\in G$, and define isomorphisms $\lambda'_g\coloneqq\lambda_{g}^{\oplus r} \colon\Eca^{\oplus r}\isoarr g^{*}\Eca^{\oplus r}$.
  The obstruction for $(\lambda_g)_g$ to give an equivariant structure is measured by $\delta_{g,h}\in\Bbbk^{\times}$ which is defined by $\lambda_{gh}=h^{*}\lambda_{g}\circ\lambda_{h}\circ\delta_{g,h}$.
  Note that $\delta_{g,h}\in\Bbbk^{\times}$ because $\Eca$ is simple.
  Similarly we have $\delta'_{g,h}\in\GL(r,\Bbbk)$ for $(\lambda'_g)_g$; actually we have $\delta'_{g,h}=(\delta_{g,h})^{\oplus r}$.
  By assumption $\Eca^{\oplus r}$ admits an equivariant structure $(\lambda''_g)_g$, so we can find elements $\varphi_g\in\GL(r,\Bbbk)$ such that $\lambda''_g=\lambda'_g\circ\varphi_g$.
  
  We claim that the obstruction $\delta''_{g,h}$ of $\lambda''$ satisfies $\delta'_{g,h}=\varphi_{gh}^{\inv}\circ\varphi_{g}\circ\varphi_{h}\circ\delta''_{g,h}$:
  Indeed we have by construction
  \begin{align*}
    \lambda''_{gh} &= h^{*}\lambda''_{g}\circ\lambda''_{h}\circ\delta''_{g,h} = h^{*}(\lambda'_{g}\circ\varphi_g)\circ\lambda'_{h}\circ\varphi_h\circ\delta''_{g,h} \\
    \lambda''_{gh} &= \lambda'_{gh}\circ\varphi_{gh} = h^{*}\lambda'_g\circ\lambda'_h\circ\delta'_{g,h}\circ\varphi_{gh},
  \end{align*}
  so we get, since the functor $h^{*}$ is linear, $\varphi_g\circ\lambda'_{h}\circ\varphi_h\circ\delta''_{g,h} = \lambda'_h\circ\delta'_{g,h}\circ\varphi_{gh}$.
  Now the matrices $\varphi_g\in\GL(r,\Bbbk)$ and $\lambda'_{h}=\diag(\lambda_h,\dots,\lambda_h)$ commute, so we get $\varphi_g\circ\varphi_h\circ\delta''_{g,h} = \delta'_{g,h}\circ\varphi_{gh}$.
  Finally we arrive at the desired $\varphi_g\circ\varphi_h\circ\delta''_{g,h} = \varphi_{gh}\circ\delta'_{g,h}$, since $\delta'_{g,h}=\diag(\delta_{g,h},\dots,\delta_{g,h})$ commutes with~$\varphi_{gh}$.
  
  As a consequence, since by assumption $\delta''_{g,h}=\id$ and $\delta'_{g,h}=\diag(\delta_{g,h},\dots,\delta_{g,h})$ is diagonal, the map $\varphi\colon G\to\GL(r,\Bbbk)$ is a projective representation of $G$ whose obstruction to being linear is exactly measured by $(\delta_{g,h})_{g,h}$.
  When the projective representation $\varphi$ comes from a linear one, the class $[\delta]\in\Ho^2(G,\Bbbk^{\times})$ becomes zero, so $\Eca$ admits an $G$-equivariant structure.
\end{proof}

\begin{proposition}\label{prop:sn-odd-projective-reprs}
  Let $G=\Sn$ be a symmetric group and assume $r\in\NN$ is odd, then every projective representation of $\Sn$ of dimension $r$ is already linear, \ie the map \[\Hom(\Sn,\GL(r,\Bbbk))\to\Hom(\Sn,\PGL(r,\Bbbk))\] is surjective.
\end{proposition}

\begin{proof}
  This follows from facts about the representation theory of finite groups and the symmetric group in particular.
  Every projective representation $\rho\colon G\to \PGL(r,\Bbbk)$ has an obstruction class in $\Ho^2(G,\Bbbk^{\times})$ which measures its failure to be linear.
  Now any projective representation $\rho$ decomposes into a direct sum of irreducible projective representations with the same obstruction class as the one of $\rho$, see \cite[Ch.~2--3]{Kar:93} for details.

  But every irreducible strictly projective representation of $\Sn$ is \emph{even} dimensional, so the odd dimensional representation $\rho$ cannot have non-trivial obstruction class.
  In fact Schur \cite{Sch:11} shows that for $n\geq4$ these irreducible projective representations are indexed by strict partitions~$(\lambda_i)_i$ of $n$, that is $\lambda_1+\dots+\lambda_\ell=n$ and $\lambda_1>\dots>\lambda_\ell$; for $n\leq3$ there are none.
  Their dimension is given by the formula
  \[
    f_{\lambda}=2^{\lfloor\frac{n-\ell}{2}\rfloor}g_{\lambda}
    \quad\text{with}\quad
    g_{\lambda}=\frac{n!}{\lambda_{1}!\cdots\lambda_{\ell}!}\prod_{i<j}\frac{\lambda_{i}-\lambda_{j}}{\lambda_{i}+\lambda_{j}}.
  \]
  The number $g_\lambda$ is in fact an integer since it counts certain ``shifted standard tableaux of shape~$\lambda$'', \cf\cite[III.8 Ex.~12]{Mac:98}.
  Finally, by the strictness of the partitions, we have $n\geq\ell+2$, so $f_\lambda$ is an even integer.
\end{proof}

\begin{theorem}[Main \cref{thm:main_kum_two_derived_eq} Part 2]\label{thm:main-theorem-generalized-kummer-derived-equivalence-part2}
  Let $n\geq5$ be an odd integer and let $A$ be an abelian variety over an algebraically closed field of characteristic zero, which admits a polarization (or symmetric isogeny) $\lambda\colon A\to A^{\vee}$ of exponent $e$ such that $\gcd(e,n)=1$.
  Then the two derived categories $\Db(\Kummer^{n-1}(A))$ and $\Db(\Kummer^{n-1}(A^{\vee}))$ of~$2(n-1)$-dimensional generalized Kummer varieties are derived equivalent.
\end{theorem}

\begin{proof}
  As summarized in \cref{par:summary-split-equivariant-semi-homogenous-bundle}, and taking \cref{prop:direct-sum-equivariant,prop:sn-odd-projective-reprs} into account, the arguments of this whole section culminate in the construction of an $\Sn$-equivariant kernel \[(\Pca,\phi)\in \DbG{\Diagonal\SymGrp_{n}}((A\otimesZ\StdRep_{n})\times(A^{\vee}\otimesZ\StdRep_{n}))\] such that $\FM_{\Pca}\colon\Db(A\otimesZ\StdRep_{n})\isoarr\Db(A^{\vee}\otimesZ\StdRep_{n})$ is a derived equivalence.
  So we can conclude exactly as in the proof of \cref{thm:main-theorem-generalized-kummer-derived-equivalence}.
\end{proof}

\section{Non-birational generalized Kummer varieties}
\label{sec:non-birat}

In this section we study birational equivalences between generalized Kummer fourfolds with the goal to exhibit examples which are not birationally equivalent but are derived equivalent according to \cref{thm:main_kum_two_derived_eq}.
We follow the strategy from \cite{Nam:02} and \cite{Oka:21}, while the calculations that lead to \cref{thm:gen-kummer-non-birat,thm:auto-group-of-gen-kummer} are original.
The interested reader is invited to consult the literature \cite{ADM:16,MMY:20,Oka:21} for further examples of non-birational but derived equivalent hyperkähler varieties.

\begin{situation}
  In this section we work over the complex numbers $\CC$, in order to consider singular cohomology of (analytifications of) varieties.
  Let $A$ be an abelian surface over $\CC$.
\end{situation}

\begin{theorem}\label{thm:gen-kummer-non-birat}
  Assume that $\End(A)=\ZZ$, and let $\deg(\lambda)=d^2$ be the degree of the polarization~$\lambda\colon A\to A^\vee$ of minimal degree.
  Assume that $d\neq1$ and either
  \begin{enumerate}
    \item
    $3$ divides $d$, and that $d/3$ is a perfect square, or
    \item
    $3$ does not divide $d$, and that the Pell equation $x^2-3d y^2=1$ has some solution with odd $y$.
  \end{enumerate}
  Then the generalized Kummer fourfolds $\Kummer^2(A)$ and $\Kummer^2(A^\vee)$ are \emph{not} birationally equivalent.
\end{theorem}

Before proving this theorem, we recall a few facts and notations about generalized Kummer varieties, which are also used in \cite[\P\P1--4]{Nam:02} and \cite[Prop.~2.2]{Oka:21}.
For further background and explanations the reader may consult~\cite[Ch.~7]{Mag:Thesis}.

\begin{passage}\label{par:pic-gen-kummer-lattice}
  We consider generalized Kummer varieties of dimension~$4$.
  \begin{enumerate}
    \item 
    We have a Hodge isometry
    \begin{equation}\label{eq:kummer-fourfold-coho}
      \Ho^2(\Kummer^2(A),\ZZ)\simeq\Ho^2(A,\ZZ)\oplus\ZZ\delta,
    \end{equation}
    where on the right hand side $\delta^2=-6$, and $2\delta=[E]$ is represented by the exceptional divisor $E\subset\Kummer^2(A)$ associated to the Hilbert--Chow morphism.
    
    In particular, we obtain from the Lefschetz theorem on $(1,1)$-classes and \eqref{eq:kummer-fourfold-coho} an isomorphism \[\NS(\Kummer^2(A))\simeq\NS(A)\oplus\ZZ\delta.\]
    Since the first Chern class $\crm_1\colon\Pic(\Kummer^2(A))\to\NS(\Kummer^2(A))$ is observed to be an isomorphism, we also have \[\Pic(\Kummer^2(A))\simeq\NS(A)\oplus\ZZ\delta.\]
    
    \item
    Taking a lattice theoretic viewpoint, recall that the lattice \[\Urm\coloneqq\ZZ\cdot \erm\oplus\ZZ\cdot \frm\] with $\erm^2=0$, $\frm^2=0$, and $\erm.\frm=1$ is called the \emph{hyperbolic plane}.
    For $k\in\ZZ$, denote by $\langle k\rangle\coloneqq\ZZ\cdot g$ the rank~$1$ lattice with $g^2=k$.
    Then one has isometries
    \begin{equation}\label{eq:isometry-h2-av-hyperbolic-lattice}
      \Ho^2(A,\ZZ)\simeq\Urm^{\oplus3}
    \end{equation}
    and \[\Ho^2(\Kummer^2(A),\ZZ)\simeq\Urm^{\oplus3}\oplus\langle -6\rangle.\]
    
    If $h\in\NS(A)$ is a primitive class, \ie it is not a multiple of any other class, with $h^2=2d$, then one can choose the isometry \eqref{eq:isometry-h2-av-hyperbolic-lattice} in such a way that $h$ corresponds to $\erm+d\frm$ in the first copy of $\Urm$ inside $\Urm^{\oplus 3}$. 
  \end{enumerate}
\end{passage}

\begin{passage}\label{par:cy-birat-iso-codim-one}
  Let $X$ and $Y$ be smooth, projective varieties, and let $f\colon X\isodasharrow Y$ be a birational map.
  \begin{enumerate}
    \item
    Assume that $X$ and $Y$ have trivial canonical bundles $\canosheaf_X\simeq\Oca_X$ and $\canosheaf_Y\simeq\Oca_Y$, \resp.
    Then $f$ is an isomorphism in codimension~$1$, so we obtain an isomorphism
    \begin{equation}\label{eq:induced-map-on-second-sing-coho}
      f^*\colon\Ho^2(Y,\ZZ)\isoarr\Ho^2(X,\ZZ).
    \end{equation}
    
    \item
    If $X$ and $Y$ are hyperkähler varieties, \eg $X=\Kummer^2(A)$ and $Y=\Kummer^2(A^\vee)$, then \eqref{eq:induced-map-on-second-sing-coho} is a Hodge isometry, where the second cohomology groups are endowed with their Beauville--Bogomolov--Fujiki form.
    
    \item
    Since $\NS(X)$ is a finitely generated abelian group, $\NS(X)\otimes_{\ZZ}\RR$ is a finite-dimensional vector space over $\RR$, which we endow with the euclidean topology.
    A divisor $D\in\Div(X)$ is called \emph{moveable} if $\codim_X(\Bl(\abs{D}))\geq2$, \ie the base locus of the complete linear system $\abs{D}$ has codimension at least $2$.
    Denote the convex cone generated by classes of movable divisors by \[\Mov(X)\subset\NS(X)\otimes_\ZZ\RR,\] and denote it closure in the euclidean topology by $\overline{\Mov}(X)$.
    
    Let $f\colon X\isodasharrow Y$ be a birational map which is an isomorphism in codimension~$1$.
    We see that if $D\in\Div(Y)$ is moveable, then $f^*D\in\Div(X)$ is also moveable.
    So the isomorphism $f^*\colon\NS(Y)\otimes_\ZZ\RR\isoarr\NS(X)\otimes_\ZZ\RR$ restricts to an isomorphism \[f^*\colon\overline{\Mov}(Y)\isoarr\overline{\Mov}(X)\] of convex cones.
    
    Recall that if $C$ is a convex cone, one calls $\RR_{\geq0}\cdot x\subset C$ an \emph{extremal ray} if for any $x_1,x_2\in C$ and $\alpha_1,\alpha_2\in\RR_{>0}$ the equation $x=\alpha_1 x_1 +\alpha_2 x_2$ implies that~$x_1,x_2\in \RR_{\geq0}\cdot x$.
    Note that $f^*$ maps extremal rays to extremal rays.
  \end{enumerate}
\end{passage}

\filbreak

\begin{passage}\label{par:gen-kummer-excep-divisor-facts}
  Denote by $E\subset\Kummer^2(A)$ the exceptional divisor of the desingularization given by the Hilbert--Chow morphism $\HilbChow\colon\Kummer^2(A)\to (A\otimes\StdRep_3)/\SymGrp_3$.
  \begin{enumerate}  
    \item
    The exceptional divisor $E\subset\Kummer^2(A)$ is effective and \emph{rigid}, \ie \[\dim\Ho^0(\Kummer^2(A),\Oca(E))=1,\] since it is the exceptional divisor of a desingularization of a normal, projective variety.
    
    \item
    Let $\widetilde{E}\to E$ be a desingularization of $E$.
    Then the Albanese of $\widetilde{E}$ is \[\Alb(\widetilde{E})\simeq A.\]
    More specifically, the singular locus of the variety $(A\otimes\StdRep_3)/\SymGrp_3 \subset\Sym^3(A)$ is given by $\Delta=\{a+a+b\mid {a\in A}, b=-2a\}$, and the morphism $A\to\Delta$ mapping $a\mapsto a+a+(-2a)$ induces an isomorphism $A\setminus A[3]\isoarr\Delta^\circ$, where $\Delta^\circ\coloneqq\Delta\setminus\{a+a+a\mid a\in A[3]\}$.
    The rational map $\widetilde{E}\to E\to\Delta\dasharrow A$ extends to a morphism $\widetilde{E}\to A$ since $\widetilde{E}$ is smooth and $A$ is an abelian variety.
    This map turns out to be an Albanese morphism (one verifies first that $E$ restricted over $\Delta^\circ$ becomes isomorphic to $\Delta^\circ\times\PP^1$).
  \end{enumerate}
\end{passage}

\begin{proof}[Proof of \cref{thm:gen-kummer-non-birat}]
  First, we follow the proof strategy from \cite{Nam:02}.
  By the assumptions on $A$, we can write $\NS(A)=\ZZ\cdot h$ with $h^2=2d$ and $d\coloneqq 3$.
  So, by \cref{par:pic-gen-kummer-lattice}, we have \[\Pic(\Kummer^2(A))\simeq \ZZ\cdot h\oplus\ZZ\cdot\delta,\] where $\delta^2=-6$ and $h.\delta=0$.
  The analogous statement holds for $A^\vee$, where we write $\check{h}$ and $\check{\delta}$ in place of $h$ and $\delta$.
  
  Let $f\colon\Kummer^2(A)\isodasharrow\Kummer^2(A^\vee)$ be a birational equivalence.
  By \cref{par:cy-birat-iso-codim-one}, $f$ is an isomorphism in codimension~$1$ and induces a Hodge-isometry $f^*\colon\Ho^2(\Kummer^2(A^\vee),\ZZ)\to\Ho^2(\Kummer^2(A),\ZZ)$, and hence it induces an isomorphism of lattices \[f^*\colon\Pic(\Kummer^2(A^\vee))\isoarr\Pic(\Kummer^2(A)).\]
  We assume that $f^*(\check{\delta})=\pm\delta$, which we verify afterwards.
  Then in fact \[f^*(\check{\delta})=\delta,\] since $2\delta$ and $2\check{\delta}$ are represented by the effective divisors $E$ and $\check{E}$, \cf\cref{par:pic-gen-kummer-lattice}.
  Since $E$ is rigid, \cf~\cref{par:gen-kummer-excep-divisor-facts}, $f\colon\Kummer^2(A)\isodasharrow\Kummer^2(A^\vee)$ restricts to a birational map \[\restr{f}{E}\colon E\isodasharrow\check{E}.\]
  Then the desingularizations $\widetilde{E}$ and $\widetilde{\check{E}}$ of $E$ and $\check{E}$, \resp, are also birationally equivalent, and since they are smooth, this induces by \cref{par:gen-kummer-excep-divisor-facts} an isomorphism of Albanese varieties \[A\simeq\Alb(\widetilde{E})\isoarr\Alb(\widetilde{\check{E}})\simeq A^\vee.\]
  This contradicts the assumption on the minimal degree of a polarization of~$A$.

  Now we verify that $f^*(\check{\delta})=\pm\delta$ (and indeed $f^*(\check{\delta})=\delta$ and $f^*(\check{h})=h$).
  We are in the situation that $f^*$ gives a lattice isometry \[\varphi\coloneqq f^*\colon\ZZ\check{h}\oplus\ZZ\check{\delta}\to\ZZ h\oplus\ZZ\delta,\] where
  \begin{empheq}[left=\empheqlbrace]{align}
    h^2 &=2d \label{eq:ns-lattice-1} \\
    \delta^2 &=-6 \\
    h.\delta &=0, \label{eq:ns-lattice-3}
  \end{empheq}
  and similarly for $\check{h}$ and $\check{\delta}$.
  We write
  \begin{empheq}[left=\empheqlbrace,box=\colorbox{black!15}]{align}
    \varphi(\check{h}) &=x h+y\delta \\
    \varphi(\check{\delta}) &=z h+w\delta
  \end{empheq}
  for some $x,y,z,w\in\ZZ$.
  Since $\varphi$ is an isometry we obtain from \eqref{eq:ns-lattice-1}-\eqref{eq:ns-lattice-3} the following equations:
  \begin{empheq}[left=\empheqlbrace]{align}
    -6 = (z h+w\delta)^2 &= 2d z^2-6w^2 \label{eq:phi-isome-eq-1} \\
    2d = (x h+y\delta)^2 &= 2d x^2-6y^2 \\
    0 = (x h+y\delta)(z h+w\delta) &= 2d xz-6yw \label{eq:phi-isome-eq-3}
  \end{empheq}

  \itemnum
  We can write $d=3d'$, since $3$ divides $d$ by assumption.
  So the system \eqref{eq:phi-isome-eq-1}-\eqref{eq:phi-isome-eq-3} becomes, after substituting $d$, dividing by $6$, and slight rearranging,
  \begin{empheq}[left=\empheqlbrace]{align}
    w^2-d'z^2 &=1 \label{eq:phi-isome-i-1} \\
    y^2-d'(x^2-1) &=0 \label{eq:phi-isome-i-2} \\
    yw-d'xz &=0 \label{eq:phi-isome-i-3}
  \end{empheq}
  Now we are already done, since the Pell equation $w^2-d'z^2=1$ has only the trivial solutions $(\pm1,0)$ when $d'$ is a perfect square.
  
  We still push the calculation further for application in \cref{thm:auto-group-of-gen-kummer} below.
  Equation \eqref{eq:phi-isome-i-3} times $w$ implies $yw^2-d'xzw=0$, and substituting \eqref{eq:phi-isome-i-1} into it yields $y(d'z^2+1)-d'xzw=0$.
  So $d'$ divides $y$, say $y=d'y'$ for some $y'\in\ZZ$.
  Then \eqref{eq:phi-isome-i-2} becomes $d'^2y'^2-d'(x^2-1)=0$, so $d'y'^2-x^2+1=0$ since $d'\neq0$.
  Also \eqref{eq:phi-isome-i-3} becomes $d'y'w-d'xz=0$, so $y'w-xz=0$.
  Collecting this together, the system \eqref{eq:phi-isome-i-1}-\eqref{eq:phi-isome-i-3} is equivalent to
  \begin{empheq}[left=\empheqlbrace]{align}
    w^2-d'z^2 &=1 \\
    x^2-d'y'^2 &=1 \\
    y'w-xz &=0 \\
    y &=d'y'.
  \end{empheq}
  Using \cref{prop:pell-double-lemma} we conclude that
  \begin{empheq}[left=\empheqlbrace,box=\colorbox{black!15}]{align}
    \varphi(\check{h}) &=x h+d'y'\delta \\
    \varphi(\check{\delta}) &= \pm(y' h+x\delta)
  \end{empheq}
  where $(x,y')$ is a solution of the Pell equation $x^2-d'y'^2=1$.
  But when $d'$ is a perfect square, then the latter Pell equation has only the trivial solutions $(\pm1,0)$, so $\varphi(\check{h})=\pm h$ and $\varphi(\check{\delta})=\pm\delta$, where the minus signs can be excluded since $h$, $\check{h}$, $\delta$, and~$\check{\delta}$ are represented by effective divisors on projective varieties.
  
  \itemnum
  From now on we assume that $\gcd(3,d)=1$.
  The system \eqref{eq:phi-isome-eq-1}-\eqref{eq:phi-isome-eq-3} becomes
  \begin{empheq}[left=\empheqlbrace]{align}
    dz^2-3(w^2-1) &=0 \label{eq:phi-isome-ii-1} \\
    d(x^2-1)-3y^2 &= 0 \\
    dxz-3yw &= 0. \label{eq:phi-isome-ii-3}
  \end{empheq}
  By \eqref{eq:phi-isome-ii-1} we know that $3|dz^2$, so $3$ divides $z$, say $z=3z'$.
  By \eqref{eq:phi-isome-ii-3} we know that $d|yw$, and by \eqref{eq:phi-isome-ii-1} we have $d|w^2-1$, so $\gcd(d,w)=1$ and $d$ divides $y$, say $y=dy'$.
  With these changes of variables, the system \eqref{eq:phi-isome-ii-1}-\eqref{eq:phi-isome-ii-3} becomes 
  \begin{empheq}[left=\empheqlbrace]{align}
    9dz'^2-3(w^2-1) &= 0 \\
    d(x^2-1)-3d^2y'^2 &= 0 \\
    3dxz'-3dy'w &= 0,
  \end{empheq}
  and after dividing by $3$ and $d$ respectively we get
  \begin{empheq}[left=\empheqlbrace]{align}
    w^2-3dz'^2 &=1 \\
    x^2-3dy'^2 &=1 \\
    xz'-y'w &=0.
  \end{empheq}
  Taking \cref{prop:pell-double-lemma} below into account, we arrive at
  \begin{empheq}[left=\empheqlbrace,box=\colorbox{black!15}]{align}
    \varphi(\check{h}) &= xh+y'd\delta \label{eq:phi-result-ii-a-1} \\
    \varphi(\check{\delta}) &= \pm(3y'h+x\delta),
  \end{empheq}
  where $(x,y')$ is a solution of the Pell equation $x^2-3dy'^2=1$.

  Now we will extract further conditions on the pairs $(x,y')$ from the fact that $\varphi$ is induced by a birational map.
  We follow the strategy of \cite[Prop.~2.2]{Oka:21}, where the non-birationality of Hilbert schemes of points on K3 surfaces is studied.
  Since $\NS(A)\simeq\ZZ$ and $\gcd(3,d)=1$, we know by \cite[Thm.~0.1]{Mor:21} that the moveable cone of $\Kummer^2(A)$ is \[\overline{\Mov}(\Kummer^2(A))=\RR_{\geq0}\cdot h\oplus\RR_{\geq0}\cdot\left(h-\frac{dy_0'}{3x_0}\delta\right),\] where $(x_0,y_0')$ is the fundamental solution of the diophantine equation $3x_0^2-dy_0'^2=3$, and similarly for $\Kummer^2(A^\vee)$.
  The latter equation implies that $y_0'$ is divisible by $3$, say~$y_0'=3y_0$.
  So \[\overline{\Mov}(\Kummer^2(A))=\RR_{\geq0}\cdot h\oplus\RR_{\geq0}\cdot\left(x_0 h-d y_0\delta\right),\] where $(x_0,y_0)$ is the fundamental solution of the Pell equation $x_0^2-3d y_0^2=1$.
  Since $\varphi$ is induced from a birational equivalence of hyperkähler varieties,
  \[\varphi\otimes\RR\colon\NS(\Kummer^2(A^\vee))\otimes\RR\to\NS(\Kummer^2(A))\otimes\RR\] maps the moveable cone to the movable cone.
  Hence $\varphi$ preserves the set $\{h,x_0 h-d y_0\delta\}$ of primitive generators of the extremal rays of the moveable cone.
  That is, either $\varphi(\check{h})=h$ or $\varphi(\check{h})=x_0 h-d y_0\delta$.
  In the first case, we get $\varphi(\check{\delta})=\delta$, as $\varphi(\check{\delta})=-\delta$ is geometrically impossible, as before.
  In the second case, comparing with \eqref{eq:phi-result-ii-a-1} yields that $(x,y')=(x_0,-y_0)$, so
  \begin{empheq}[left=\empheqlbrace]{align}
    \varphi(\check{h}) &= x_0 h - y_0 d\delta \\
    \varphi(\check{\delta}) &= \pm(3y_0 h - x_0\delta). \label{eq:phi-result-ii-b-2}
  \end{empheq}
  
  We claim that the undetermined sign in \eqref{eq:phi-result-ii-b-2} must be positive, so either $\varphi=\id$ or
  \begin{empheq}[left=\empheqlbrace,box=\colorbox{black!15}]{align}
    \varphi(\check{h}) &= x_0 h - y_0 d\delta \\
    \varphi(\check{\delta}) &= 3y_0 h - x_0\delta,
  \end{empheq}
  where $(x_0,y_0)$ is the fundamental solution of the Pell equation $x^2-3dy^2=1$.
  Indeed, consider the birational map \[f^\inv\colon\Kummer^2(A^\vee) \dasharrow \Kummer^2(A),\] which induces $\varphi^\inv=(f^\inv)^{*}$.
  By the arguments above, either $\varphi^\inv=\id$ (and thus $\varphi=\id$) or both $\varphi^\inv$ and~$\varphi$ correspond to the fundamental solution $(x_0,y_0)$, up to sign.
  There~are three cases to inspect:
  \begin{enumerate*}[label=(\alph*)]
    \item Both $\varphi$ and $\varphi^\inv$ have positive sign, \label{item:phi-isome-sign-pos-pos}
    \item $\varphi$ has positive sign and $\varphi^\inv$ has negative sign, \label{item:phi-isome-sign-pos-neg}
    \item both $\varphi$ and $\varphi^\inv$ have negative sign. \label{item:phi-isome-sign-neg-neg}
  \end{enumerate*}
  In case~\ref{item:phi-isome-sign-pos-neg}~and~\ref{item:phi-isome-sign-neg-neg} we calculate
  \begin{multline*}
    \varphi^\inv(\varphi(\check{h}))
    = \varphi^\inv(x_0 h-y_0 d\delta)
    = x_0(x_0\check{h}-y_0 d\check{\delta})+y_0 d(3y_0\check{h}-x_0\check{\delta}) \\
    = (x_0^2+3d y_0^2)\check{h} - 2x_0 y_0 d\check{\delta}
    \neq \check{h}.
  \end{multline*}
  Case \ref{item:phi-isome-sign-pos-pos} remains possible, since $\varphi^\inv(\varphi(\check{h}))=\check{h}$ and $\varphi^\inv(\varphi(\check{\delta}))=\check{\delta}$ does indeed hold.

  Next, we want to see that $\varphi=\id$ as soon as $y_0$ is odd.
  Since $\varphi=f^*$ is induced from a birational map $f\colon\Kummer^2(A)\dasharrow\Kummer^2(A^\vee)$, we have a commutative diagram
  \begin{equation}\label{diag:phi-isome-extension}
  \begin{tikzcd}[column sep=1.9ex]
    \!\! \ZZ \check{h}\oplus\ZZ\check{\delta} \arrow[d, "\varphi"] \arrow[r, phantom, "\simeq" description] & \NS(\Kummer^2(A^\vee)) \arrow[d, "f^*"] \arrow[r, hook] & {\Ho^2(\Kummer^2(A^\vee),\ZZ)} \arrow[d, "f^*"] \arrow[r, phantom, "\simeq" description] & {\Ho^2(A^\vee,\ZZ)\oplus\ZZ\delta} \arrow[r, phantom, "\simeq" description] & \Ulattice^{\oplus 3}\oplus\langle -6\rangle \arrow[d, "\varphi"] \\
    \!\! \ZZ h\oplus\ZZ\delta \arrow[r, phantom, "\simeq" description] & \NS(\Kummer^2(A)) \arrow[r, hook] & {\Ho^2(\Kummer^2(A),\ZZ)} \arrow[r, phantom, "\simeq" description] & {\Ho^2(A,\ZZ)\oplus\ZZ\delta} \arrow[r, phantom, "\simeq" description] & \Ulattice^{\oplus 3}\oplus\langle -6\rangle .
  \end{tikzcd}
  \end{equation}
  That is, the lattice isometry $\varphi\colon\ZZ \check{h}\oplus\ZZ\check{\delta}\to \ZZ h\oplus\ZZ\delta$ is the restriction of a lattice isometry $\Ulattice^{\oplus 3}\oplus\langle -6\rangle\isoarr\Ulattice^{\oplus 3}\oplus\langle -6\rangle$.
  Denote the generator of $\langle-6\rangle$ again by $\delta$, and denote the bases of the three copies of $\Ulattice$ by $\{\erm,\frm\}$, $\{\erm',\frm'\}$, and $\{\erm'',\frm''\}$ \resp.
  Recall that under the identifications in Diagram~\eqref{diag:phi-isome-extension}, we have
  \begin{align*}
    \ZZ h\oplus\ZZ\delta &\hookrightarrow \Ulattice^{\oplus 3}\oplus\langle -6\rangle \\
    h &\mapsto \erm+d\frm \\
    \delta &\mapsto \delta,
  \end{align*}
  and similarly for $\check{h}$ and $\check{\delta}$.
  For the readers convenience, we recall that $\erm^2=0$, $\frm^2=0$ and $\erm.\frm=1$, and that all directs sums in \eqref{diag:phi-isome-extension} are orthogonal direct sums.

  Now, let $(x_0,y_0)$ be (the fundamental) solution of $x^2-3d y^2=1$, and assume $\varphi(\check{\delta})=3y_0 h-x_0\delta$ and $\varphi(\erm+d\frm)=x_0(\erm+d\frm)-y_0 d\delta$.
  We study when it is possible to extend $\varphi$ to a map of lattices $\Ulattice\oplus\langle -6\rangle\to\Ulattice^{\oplus 3}\oplus\langle -6\rangle$.
  Write \[\varphi(\frm) = x \frm + y h + z\delta \;\;+\;\; v'\erm' + w'\frm' + v''\erm'' + w''\frm''\]
  for some $x,y,z,v',w',v'',w''\in\ZZ$.
  One verifies $\varphi(\frm)^2=2xy+2dy^2-6z^2 + 2v'w'+2v''w''$, so $f^2=0$ induces the equation \[xy+dy^2-3z^2+w=0 \quad\text{with}\quad w\coloneqq v'w'+v''w''.\]
  We have $\frm.\delta=0$, so $\varphi(\frm).\varphi(\delta)=3y_0 x + 6d y_0 y + 6x_0 z$ provides the equation \[y_0 x + 2d y_0 y + 2x_0 z = 0.\]
  From $\frm.h=1$ and $\varphi(\frm).\varphi(h) = x_0 x+2dx_0 y + 6d y_0 z$ we get \[x_0 x+2dx_0 y + 6d y_0 z = 1.\]
  Thus we have the system of equations
  \begin{empheq}[left=\empheqlbrace]{align}
    xy+dy^2-3z^2 + w &= 0 \label{eq:phi-extend-1} \\
    y_0 x + 2d y_0 y + 2x_0 z &= 0 \label{eq:phi-extend-2} \\
    x_0 x + 2d x_0 y + 6d y_0 z &= 1 \label{eq:phi-extend-3}
  \end{empheq}
  Equation~\eqref{eq:phi-extend-2} tells us that $x=\frac{-2}{y_0}(d y_0 y+x_0 z)$, and substituting this into \eqref{eq:phi-extend-3} yields \[1=\frac{-2 x_0}{y_0}(d y_0 y+x_0 z) + 2d x_0 y + 6d y_0 z = \left(\frac{-2x_0^2}{y_0}+6d y_0\right)z.\]
  Since $y_0\neq 0$, the latter is equivalent to $y_0 = (-2x_0^2 + 6d y_0^2) = (-2x_0^2 + 2x_0^2 - 2)z = -2z$, where we have used the Pell equation $3dy_0^2=x_0^2-1$.
  So we conclude $z=-\frac{y_0}{2}$ and see that $y_0$ has to be an even number, which is clearly in contradiction with the hypothesis that $y_0$ is odd.
\end{proof}

We used the following Lemma about Pell equations, which are diophantine equation of the form $x^2-N y^2=1$ with $N\in\NN$.
Also recall that the \emph{fundamental solution} $(x_0,y_0)$ is the solution with $x_0,y_0>0$ and $\abs{x_0}$ minimal among all solutions; it exists if $N$ is not a perfect square, otherwise only the trivial solutions $(x,y)=(\pm1,0)$ exist.

\begin{lemma}\label{prop:pell-double-lemma}
  Let $(x_1,y_1)$ and $(x_2,y_2)$ be two solutions of the Pell equation $x^2-N y^2=1$ which satisfy $x_1 y_2+x_2 y_1=0$, then $x_1=\pm x_2$ and $y_1=\mp y_2$.
\end{lemma}

\begin{proof}
  By Brahmagupta's identity, we have that $(x_1 x_2+N y_1 y_2)^2-N(x_1 y_2+x_2 y_1)^2=1$.
  So $x_1 y_2+x_2 y_1=0$ implies that $x_1 x_2+N y_1 y_2=\pm1$.
  Multiplying by $y_2$ yields $x_1 x_2 y_2+N y_1 y_2^2=\pm y_2$, and substituting the equation $x_1 y_2=-x_2 y_1$ from the hypothesis yields \[\pm y_2 = -x_2^2 y_1 + N y_1 y_2^2 = -y_1(x_2^2-N y_2^2)=-y_2.\]
  So the hypothesis $x_1 y_2+x_2 y_1=0$ becomes $x_1 y_2 \mp x_2 y_2 =0$, implying $x_1=\pm x_2$ (which is also true in the case of the trivial solutions $(\pm1,0)$).
\end{proof}

\begin{remark}
  Case~\rom{2} of \cref{thm:gen-kummer-non-birat} occurs for arbitrarily hight degree~$d^2$.
  Indeed, consider the Pell equation $x^2-3d y^2=1$ for $d\in\NN$, with fundamental solution~$(x_0,y_0)$.
  Then there are infinitely many $d\in\NN$, with $\gcd(3,d)=1$, such that $y_0=1$.
  Concretely, for any $k\in\NN$ set $d'\coloneqq 2k^2+(k+1)^2-1$ and $d''\coloneqq k^2+2(k+1)^2-1$.
  Then one calculates $(3k+1)^2-3d'=1$ and $(3k+2)^2-3d''=1$, as desired.
  Reducing modulo~$3$ one sees that for every $k$ either $d'$ or $d''$ is coprime to $3$. 
\end{remark}

\begin{remark}
  \cref{thm:gen-kummer-non-birat} remains true when replacing $A^\vee$ by any abelian surface $\check{A}$ with $\End(\check{A})=\ZZ$ and $A\not\simeq\check{A}$.
  Indeed, write $\NS(\check{A})=\ZZ\cdot \check{h}$ with $\check{h}^2=2\check{d}$.
  Then the existence of a lattice isometry $\varphi\colon\ZZ\check{h}\oplus\ZZ\check{\delta}\to\ZZ h\oplus\ZZ\delta$, as considered in the proof of \cref{thm:gen-kummer-non-birat}, implies the equality \[ -12d=\disc(\ZZ h\oplus\ZZ\delta)=\disc(\ZZ\check{h}\oplus\ZZ\check{\delta})=-12\check{d} \] of discriminants.
  So $d=\check{d}$, and the calculations in the proof of \cref{thm:gen-kummer-non-birat} are valid without any modification.
\end{remark}

\begin{theorem}\label{thm:auto-group-of-gen-kummer}
  Assume that $\End(A)=\ZZ$, and let $\deg(\lambda)=d^2$ be the degree of the polarization $\lambda\colon A\to A^\vee$ of minimal degree.
  Assume that $d\neq1$ and either
  \begin{enumerate}
    \item
    $3$ divides $d$, and that $d/3$ is a perfect square, or
    \item
    $3$ does not divide $d$, and that the Pell equation $x^2-3d y^2=1$ has some solution with odd $y$.
  \end{enumerate}
  Then we have isomorphisms \[\Bir(\Kummer^2(A))\simeq\Aut(\Kummer^2(A))\simeq A[3]\rtimes\Aut_{\mathrm{AV}}(A),\] where $\Bir(\Kummer^2(A))$ is the group of birational autoequivalences, and the action of $\Aut_{\mathrm{AV}}(A)$ on $A[3]$ is the obvious one.
\end{theorem}

\begin{proof}
  Let $n\geq 3$.
  In general, following \cite[§3.1]{BNS:11}, we have an injective homomorphism \[A\rtimes\Aut_{\mathrm{AV}}(A)\simeq\Aut(A)\hookrightarrow\Aut(\Hilb^n(A)),\] and one calls automorphisms in the image of this map \emph{natural}.
  A natural automorphism of $\Hilb^n(A)$ restricts to an automorphism of $\Kummer^{n-1}(A)$ \iff it corresponds to an element in $A[3]\rtimes\Aut_{\mathrm{AV}}(A)$; these are called again \emph{natural}.
  By \cite[Thm.~3, Cor.~5]{BNS:11}, this leads to an injective homomorphism \[A[3]\rtimes\Aut_{\mathrm{AV}}(A)\hookrightarrow\Aut(\Kummer^{n-1}(A))\] whose image consists of those automorphisms $f\colon\Kummer^{n-1}(A)\isoarr\Kummer^{n-1}(A)$ which satisfy $f(E)=E$.
  Recall here that $E\subset\Kummer^{n-1}(A)$ denotes the exceptional divisor of the Hilbert--Chow morphism, \cf\cref{par:gen-kummer-excep-divisor-facts}; there exists a class $\delta\in\NS(\Kummer^{n-1}(A))$ satisfying $2\delta=[E]$, \cf\cref{par:pic-gen-kummer-lattice}.
  Now the condition $f(E)=E$ is equivalent to $f^*(\delta)=\delta$, since $E$ is rigid, \cf\cref{par:gen-kummer-excep-divisor-facts}.
  
  Consider a birational equivalence $f\colon\Kummer^2(A)\isodasharrow\Kummer^2(A)$.
  Exactly as in the proof of \cref{thm:gen-kummer-non-birat}, we see that \[f^*\colon\NS(\Kummer^2(A))\isoarr\NS(\Kummer^2(A))\] is the identity map.
  In particular, $f^*$ fixes some ample class (\eg $k\cdot h-\delta$ for $k\gg1$), and since~$f$ is already an isomorphism in codimension~$1$, \cf\cref{par:cy-birat-iso-codim-one}, $f$ extends to an automorphism of $\Kummer^2(A)$.
  We conclude that \[\Bir(\Kummer^2(A))\simeq\Aut(\Kummer^2(A)).\]
  Finally, since $f^*(\delta)=\delta$, the automorphism $f$ must be natural.
\end{proof}


\printbibliography 

@Article{ADM:16,
  Author = {Addington, Nicolas and Donovan, Will and Meachan, Ciaran},
  Title = {Moduli spaces of torsion sheaves on {{\(K3\)}} surfaces and derived equivalences},
  FJournal = {Journal of the London Mathematical Society. Second Series},
  Journal = {J. Lond. Math. Soc., II. Ser.},
  ISSN = {0024-6107},
  Volume = {93},
  Number = {3},
  Pages = {846--865},
  Year = {2016},
  Language = {English},
  DOI = {10.1112/jlms/jdw022},
  zbMATH = {6618276},
  Zbl = {1361.14013}
}

@Online{Bal:09,
  Author = {Matthew Robert {Ballard}},
  Title = {{Equivalences of derived categories of sheaves on quasi-projective schemes}},
  Year = {2009},
  Version = {2 (Sep~2009)},
  Eprinttype = {arxiv},
  Eprint = {0905.3148},
}

@Book{Ber:IGH,
  Author = {Gr\'egory {Berhuy}},
  Title = {{An introduction to Galois cohomology and its applications}},
  Series = {{London Mathematical Society Lecture Note Series}},
  ISSN = {0076-0552},
  Volume = {377},
  ISBN = {978-0-521-73866-8/pbk},
  Pages = {xi + 315},
  Year = {2010},
  Publisher = {Cambridge University Press, Cambridge},
  Language = {English},
  MSC2010 = {12G05 12-02 11E04 11E72 12F12 16K20 16W10 20G10 20G15 20J06},
  Zbl = {1207.12003},
  DOI = {10.1017/CBO9781139107051},
  URL = {https://doi.org/10.1017/CBO9781139107051},
}

@Article{Bea:83,
  Author = {Arnaud {Beauville}},
  Title = {{Vari\'et\'es k\"ahleriennes dont la premi\`ere classe de Chern est nulle}},
  FJournal = {{Journal of Differential Geometry}},
  Journal = {{J. Differ. Geom.}},
  ISSN = {0022-040X},
  Volume = {18},
  Number = {4},
  Pages = {755--782},
  Year = {1983},
  Publisher = {International Press of Boston, Somerville, MA},
  Language = {French},
  DOI = {10.4310/jdg/1214438181},
  URL = {http://projecteuclid.org/euclid.jdg/1214438181},
  MSC2010 = {53C55 14J10},
  Zbl = {0537.53056}
}

@InCollection{Bey:82,
  Author = {F. R. {Beyl}},
  Title = {{Isoclinisms of group extensions and the Schur multiplicator}},
  BookTitle = {Groups---{S}t. {A}ndrews 1981},
  Series = {London Math. Soc. Lecture Note Ser.},
  Volume = {71},
  Pages = {169--185},
  Publisher = {Cambridge Univ. Press, Cambridge},
  Year = {1982},
  Language = {English},
  MSC2010 = {20J05 20E10 20E22},
  Zbl = {0509.20037}
}

@Article{BKR:01,
  Author = {Tom {Bridgeland} and Alastair {King} and Miles {Reid}},
  Title = {{The McKay correspondence as an equivalence of derived categories}},
  FJournal = {{Journal of the American Mathematical Society}},
  Journal = {{J. Am. Math. Soc.}},
  ISSN = {0894-0347},
  Volume = {14},
  Number = {3},
  Pages = {535--554},
  Year = {2001},
  Publisher = {American Mathematical Society (AMS), Providence, RI},
  Language = {English},
  DOI = {10.1090/S0894-0347-01-00368-X},
  URL = {https://doi.org/10.1090/S0894-0347-01-00368-X},
  MSC2010 = {14J50 18E30 14E15 19L47 14L30 14J30},
  Zbl = {0966.14028}
}

@Book{BL:94,
  Author = {Joseph {Bernstein} and Valery {Lunts}},
  Title = {{Equivariant Sheaves and Functors}},
  FSeries = {{Lecture Notes in Mathematics}},
  Series = {{Lect. Notes Math.}},
  ISSN = {0075-8434},
  Volume = {1578},
  ISBN = {3-540-58071-9},
  Year = {1994},
  Publisher = {Springer-Verlag, Berlin},
  Language = {English},
  DOI = {10.1007/BFb0073549},
  URL = {https://doi.org/10.1007/BFb0073549},
  MRCLASS = {55N91 (14M25 18E30 54B40 55N30)},
  Zbl = {0808.14038}
}

@Article{BL:02,
  Author = {Christina {Birkenhake} and Herbert {Lange}},
  Title = {{An isomorphism between moduli spaces of abelian varieties}},
  FJournal = {{Mathematische Nachrichten}},
  Journal = {{Math. Nachr.}},
  ISSN = {0025-584X; 1522-2616/e},
  Volume = {253},
  Pages = {3--7},
  Year = {2003},
  Publisher = {Wiley (Wiley-VCH), Weinheim},
  Language = {English},
  MSC2010 = {14K10 14K05},
  Zbl = {1034.14019},
  DOI = {10.1002/mana.200310041},
  URL = {https://doi.org/10.1002/mana.200310041},
}

@Article{BNS:11,
  Author = {Boissi{\`e}re, Samuel and Nieper-Wi{\ss}kirchen, Marc and Sarti, Alessandra},
  Title = {Higher dimensional {Enriques} varieties and automorphisms of generalized {Kummer} varieties},
  FJournal = {Journal de Math{\'e}matiques Pures et Appliqu{\'e}es. Neuvi{\`e}me S{\'e}rie},
  Journal = {J. Math. Pures Appl. (9)},
  ISSN = {0021-7824},
  Volume = {95},
  Number = {5},
  Pages = {553--563},
  Year = {2011},
  Language = {English},
  DOI = {10.1016/j.matpur.2010.12.003},
  zbMATH = {5888512},
  Zbl = {1215.14046}
}

@Online{BO:95,
  Author = {Alexey {Bondal} and Dmitri {Orlov}},
  Title = {{Semiorthogonal decomposition for algebraic varieties}},
  Year = {1995},
  Version = {1 (Jun~1995)},
  Eprinttype = {arxiv},
  Eprint = {alg-geom/9506012},
}

@Article{BO:01,
  Author = {Alexey {Bondal} and Dmitri {Orlov}},
  Title = {{Reconstruction of a variety from the derived category and groups of autoequivalences}},
  FJournal = {{Compositio Mathematica}},
  Journal = {{Compos. Math.}},
  ISSN = {0010-437X},
  Volume = {125},
  Number = {3},
  Pages = {327--344},
  Year = {2001},
  Publisher = {Cambridge University Press, Cambridge; London Mathematical Society, London},
  Language = {English},
  DOI = {10.1023/A:1002470302976},
  URL = {https://doi.org/10.1023/A:1002470302976},
  MSC2010 = {18E30 14F05 18F20},
  Zbl = {0994.18007}
}

@Article{Bog:74,
  Author = {F. A. {Bogomolov}},
  Title = {{The decomposition of K\"ahler manifolds with a trivial canonical class}},
  FJournal = {{Matematicheski\u{\i} Sbornik. Novaya Seriya}},
  Journal = {{Mat. Sb., Nov. Ser.}},
  Volume = {93},
  Number = {135},
  Pages = {573--575},
  Year = {1974},
  Publisher = {Izdatel'stvo Nauka, Moskva},
  Language = {Russian},
  MSC2010 = {32J99 32M99 53C55 57R20 14E05},
  Zbl = {0299.32023}
}

@Article{Bri:02,
  Author = {Tom {Bridgeland}},
  Title = {{Flops and derived categories.}},
  FJournal = {{Inventiones Mathematicae}},
  Journal = {{Invent. Math.}},
  ISSN = {0020-9910},
  Volume = {147},
  Number = {3},
  Pages = {613--632},
  Year = {2002},
  Publisher = {Springer, Berlin/Heidelberg},
  Language = {English},
  DOI = {10.1007/s002220100185},
  URL = {https://doi.org/10.1007/s002220100185},
  MSC2010 = {14E30 18E30 14F05 14J30 14J32},
  Zbl = {1085.14017}
}

@Book{Bro:82,
  Author = {Kenneth S. {Brown}},
  Title = {{Cohomology of groups}},
  Series = {{Graduate Texts in Mathematics}},
  Journal = {{Grad. Texts Math.}},
  ISSN = {0072-5285},
  Volume = {87},
  Year = {1982},
  Publisher = {Springer, New York},
  Language = {English},
  MSC2010 = {20J05 20-02 18-02 55-02 20J06},
  Zbl = {0584.20036}
}

@Article{BS:64,
  Author = {Armand {Borel} and Jean-Pierre {Serre}},
  Title = {{Th{\'e}oremes de finitude en cohomologie galoisienne}},
  FJournal = {{Commentarii Mathematici Helvetici}},
  Journal = {{Comment. Math. Helv.}},
  ISSN = {0010-2571},
  Volume = {39},
  Pages = {111--164},
  Year = {1964},
  Language = {French},
  DOI = {10.1007/BF02566948},
  URL = {https://doi.org/10.1007/BF02566948},
  MRCLASS = {14.50},
  Zbl = {0143.05901}
}

@Article{CHN:10,
  Author = {Frederick R. {Cohen} and David J. {Hemmer} and Daniel K. {Nakano}},
  Title = {{On the cohomology of Young modules for the symmetric group.}},
  FJournal = {{Advances in Mathematics}},
  Journal = {{Adv. Math.}},
  ISSN = {0001-8708},
  Volume = {224},
  Number = {4},
  Pages = {1419--1461},
  Year = {2010},
  Publisher = {Elsevier (Academic Press), San Diego, CA},
  Language = {English},
  DOI = {10.1016/j.aim.2010.01.004},
  MSC2010 = {20J06 20C30 05E15 55S12 55P47 16E40 20C20 18G40},
  Zbl = {1206.20057}
}

@Online{EGM,
  Shorthand = {EGM},
  Author = {Bas {Edixhoven} and Gerard {van der Geer} and Ben {Moonen}},
  Title = {{Abelian Varieties.}},
  Language = {English},
  Note = {Preliminary version of the first chapters},
  Url = {https://www.math.ru.nl/~bmoonen/research.html#bookabvar},
}

@Online{FH:22,
  Author = {Sarah {Frei} and Katrina {Honigs}},
  Title = {{Groups of symplectic involutions on symplectic varieties of Kummer type and their fixed loci}},
  Year = {2022},
  Version = {1 (Jul~2022)},
  Eprinttype = {arxiv},
  Eprint = {2207.14035},
}

@Manual{GAP4,
  Shorthand = {GAP},
  Organization = {{The GAP~Group}},
  Title = {{GAP -- Groups, Algorithms, and Programming, Version 4.11.1}},
  Year = {2021},
  URL = {https://www.gap-system.org},
}

@Book{Gir:CNA,
  Author = {Jean {Giraud}},
  Title = {{Cohomologie non ab\'elienne}},
  Series = {{Grundlehren der Mathematischen Wissenschaften}},
  ISSN = {0072-7830; 2196-9701/e},
  Volume = {179},
  Year = {1971},
  Publisher = {Springer, Berlin},
  Language = {French},
  MSC2010 = {18G50 18-02 18F20 14F05 55R40 14F20},
  Zbl = {0226.14011}
}

@Article{Hai:01,
  Author = {Mark {Haiman}},
  Title = {{Hilbert schemes, polygraphs and the Macdonald positivity conjecture}},
  FJournal = {{Journal of the American Mathematical Society}},
  Journal = {{J. Am. Math. Soc.}},
  ISSN = {0894-0347},
  Volume = {14},
  Number = {4},
  Pages = {941--1006},
  Year = {2001},
  Publisher = {American Mathematical Society (AMS), Providence, RI},
  Language = {English},
  DOI = {10.1090/S0894-0347-01-00373-3},
  URL = {https://doi.org/10.1090/S0894-0347-01-00373-3},
  MSC2010 = {14C05 05E05 14M05},
  Zbl = {1009.14001}
}

@Article{Hem:09,
  Author = {David J. {Hemmer}},
  Title = {{Cohomology and generic cohomology of Specht modules for the symmetric group.}},
  FJournal = {{Journal of Algebra}},
  Journal = {{J. Algebra}},
  ISSN = {0021-8693},
  Volume = {322},
  Number = {5},
  Pages = {1498--1515},
  Year = {2009},
  Publisher = {Elsevier (Academic Press), San Diego, CA},
  Language = {English},
  DOI = {10.1016/j.jalgebra.2009.05.041},
  MSC2010 = {20J06 20C30 20G05 20G10},
  Zbl = {1187.20062}
}

@Article{HLOY:03,
  Author = {Shinobu {Hosono} and Bong H. {Lian} and Keiji {Oguiso} and Shing-Tung {Yau}},
  Title = {{Kummer structures on a \(K3\) surface: An old question of T. Shioda}},
  FJournal = {{Duke Mathematical Journal}},
  Journal = {{Duke Math. J.}},
  ISSN = {0012-7094},
  Volume = {120},
  Number = {3},
  Pages = {635--647},
  Year = {2003},
  Publisher = {Duke University Press, Durham, NC; University of North Carolina, Chapel Hill, NC},
  Language = {English},
  DOI = {10.1215/S0012-7094-03-12036-0},
  URL = {https://doi.org/10.1215/S0012-7094-03-12036-0},
  MSC2010 = {14J28},
  Zbl = {1051.14046}
}

@Book{Huy:FM,
  Author = {Daniel {Huybrechts}},
  Title = {{Fourier-Mukai transforms in algebraic geometry}},
  Series = {{Oxford Mathematical Monographs}},
  ISBN = {0-19-929686-3/hbk},
  Pages = {viii + 307},
  Year = {2006},
  Publisher = {Clarendon Press, Oxford University Press, Oxford},
  Language = {English},
  MSC2010 = {14-02 14D20 14F05 18E30},
  Zbl = {1095.14002},
  DOI = {10.1093/acprof:oso/9780199296866.001.0001},
  URL = {https://doi.org/10.1093/acprof:oso/9780199296866.001.0001},
}

@Book{Huy:K3,
  Author = {Daniel {Huybrechts}},
  Title = {{Lectures on K3 surfaces}},
  Series = {{Cambridge Studies in Advanced Mathematics}},
  Journal = {{Camb. Stud. Adv. Math.}},
  Volume = {158},
  ISBN = {978-1-107-15304-2},
  Pages = {xi + 485},
  Year = {2016},
  Publisher = {Cambridge University Press, Cambridge},
  Language = {English},
  DOI = {10.1017/CBO9781316594193},
  URL = {https://doi.org/10.1017/CBO9781316594193},
  MSC2010 = {14J28 14-02 14C15 14D20 14F05 14F22 14J10 14J27 14J60 32J15},
  Zbl = {1360.14099}
}

@Article{IN:96,
  Author = {Ito, Yukari and Nakamura, Iku},
  Title = {{McKay} correspondence and {Hilbert} schemes},
  FJournal = {Proceedings of the Japan Academy. Series A},
  Journal = {Proc. Japan Acad., Ser. A},
  ISSN = {0386-2194},
  Volume = {72},
  Number = {7},
  Pages = {135--138},
  Year = {1996},
  Language = {English},
  DOI = {10.3792/pjaa.72.135},
  zbMATH = {993528},
  Zbl = {0881.14002}
}

@Book{Kar:87,
  Author = {Gregory {Karpilovsky}},
  Title = {{The Schur multiplier}},
  Series = {{London Mathematical Society Monographs. New Series}},
  Journal = {{Lond. Math. Soc. Monogr., New Ser.}},
  Volume = {2},
  ISBN = {0-19-853554-6},
  Year = {1987},
  Publisher = {Clarendon Press, Oxford},
  Language = {English},
  MSC2010 = {20-02 20J05 20C25 20D15 20J06 20C05},
  Zbl = {0619.20001}
}

@Book{Kar:93,
  Author = {Karpilovsky, Gregory},
  Title = {Group representations. {Volume} 2},
  FSeries = {North-Holland Mathematics Studies},
  Series = {North-Holland Math. Stud.},
  ISSN = {0304-0208},
  Volume = {177},
  ISBN = {0-444-88726-1},
  Year = {1993},
  Publisher = {North-Holland, Amsterdam},
  Language = {English},
  zbMATH = {205981},
  Zbl = {0778.20001}
}

@Article{KS:99,
  Author = {A. S. {Kleshchev} and J. {Sheth}},
  Title = {{On extensions of simple modules over symmetric and algebraic groups}},
  FJournal = {{Journal of Algebra}},
  Journal = {{J. Algebra}},
  ISSN = {0021-8693},
  Volume = {221},
  Number = {2},
  Pages = {705--722},
  Year = {1999},
  Publisher = {Elsevier (Academic Press), San Diego, CA},
  Language = {English},
  DOI = {10.1006/jabr.1999.8038},
  MSC2010 = {20C30 20G05 20G10 20J06},
  Zbl = {0981.20006}
}

@Article{LT:17,
  Author = {Ana Cristina {L\'opez Mart\'{i}n} and Carlos {Tejero Prieto}},
  Title = {{Derived equivalences of Abelian varieties and symplectic isomorphisms}},
  FJournal = {{Journal of Geometry and Physics}},
  Journal = {{J. Geom. Phys.}},
  ISSN = {0393-0440},
  Volume = {122},
  Pages = {92--102},
  Year = {2017},
  Publisher = {Elsevier (North-Holland), Amsterdam},
  Language = {English},
  MSC2010 = {14F05 14J10 18E30 14K10 18E25},
  Zbl = {1401.14096},
  DOI = {10.1016/j.geomphys.2017.01.010},
  URL = {https://doi.org/10.1016/j.geomphys.2017.01.010},
}

@Book{Mac:98,
  Author = {Macdonald, I. G.},
  Title = {Symmetric functions and {Hall} polynomials},
  Series = {Oxford Mathematical Monographs},
  Edition = {2nd},
  ISBN = {0-19-850450-0},
  Year = {1998},
  Publisher = {Clarendon Press, Oxford},
  Language = {English},
  zbMATH = {1181673},
  Zbl = {0899.05068}
}

@Article{MMY:20,
  Author = {Meachan, Ciaran and Mongardi, Giovanni and Yoshioka, K{\=o}ta},
  Title = {Derived equivalent {Hilbert} schemes of points on {{\(K3\)}} surfaces which are not birational},
  FJournal = {Mathematische Zeitschrift},
  Journal = {Math. Z.},
  ISSN = {0025-5874},
  Volume = {294},
  Number = {3-4},
  Pages = {871--880},
  Year = {2020},
  Language = {English},
  DOI = {10.1007/s00209-019-02281-1},
  zbMATH = {7179282},
  Zbl = {1469.14011}
}

@Article{Mor:21,
  Author = {Mori, Akira},
  Title = {Nef cone of a generalized {Kummer} 4-fold},
  FJournal = {Hokkaido Mathematical Journal},
  Journal = {Hokkaido Math. J.},
  ISSN = {0385-4035},
  Volume = {50},
  Number = {2},
  Pages = {151--163},
  Year = {2021},
  Language = {English},
  DOI = {10.14492/hokmj/2018-919},
  zbMATH = {7395070},
  Zbl = {1473.14077}
}

@Article{Muk:78,
  Author = {Mukai, Shigeru},
  Title = {Semi-homogeneous vector bundles on an abelian variety},
  FJournal = {Journal of Mathematics of Kyoto University},
  Journal = {J. Math. Kyoto Univ.},
  ISSN = {0023-608X},
  Volume = {18},
  Number = {2},
  Pages = {239--272},
  Year = {1978},
  Language = {English},
  DOI = {10.1215/kjm/1250522574},
  URL = {https://doi.org/10.1215/kjm/1250522574},
  Zbl = {0417.14029}
}

@Article{Muk:81,
  Author = {Shigeru {Mukai}},
  Title = {{Duality between \(D(X)\) and \(D(\hat X)\) with its application to Picard sheaves}},
  FJournal = {{Nagoya Mathematical Journal}},
  Journal = {{Nagoya Math. J.}},
  ISSN = {0027-7630},
  Volume = {81},
  Pages = {153--175},
  Year = {1981},
  Publisher = {Cambridge University Press, Cambridge},
  Language = {English},
  DOI = {10.1017/S002776300001922X},
  URL = {http://projecteuclid.org/euclid.nmj/1118786312},
  MSC2010 = {14K30 14C22 14K05},
  Zbl = {0417.14036}
}

@Online{Muk:98,
  Author = {Shigeru {Mukai}},
  Title = {{Abelian variety and spin representation.}},
  Year = {1998},
  URL = {https://www.kurims.kyoto-u.ac.jp/~mukai/paper/warwick13.pdf},
  Note = {Warwick},
}

@Article{Nak:60,
  Author = {Minoru {Nakaoka}},
  Title = {{Decomposition theorem for homology groups of symmetric groups}},
  FJournal = {{Annals of Mathematics. Second Series}},
  Journal = {{Ann. Math. (2)}},
  ISSN = {0003-486X; 1939-8980/e},
  Volume = {71},
  Pages = {16--42},
  Year = {1960},
  Publisher = {Princeton University, Mathematics Department, Princeton, NJ},
  Language = {English},
  Zbl = {0090.39002},
  DOI = {10.2307/1969878},
  URL = {https://doi.org/10.2307/1969878},
}

@Article{Nam:02,
  Author = {Yoshinori {Namikawa}},
  Title = {{Counter-example to global Torelli problem for irreducible symplectic manifolds}},
  FJournal = {{Mathematische Annalen}},
  Journal = {{Math. Ann.}},
  ISSN = {0025-5831},
  Volume = {324},
  Number = {4},
  Pages = {841--845},
  Year = {2002},
  Publisher = {Springer, Berlin/Heidelberg},
  Language = {English},
  DOI = {10.1007/s00208-002-0344-2},
  URL = {https://doi.org/10.1007/s00208-002-0344-2},
  MSC2010 = {53D35 57R17},
  Zbl = {1028.53081}
}

@article {Nam:02b,
  Author = {Yoshinori {Namikawa}},
  Title = {{Erratum: ``Counter-example to global Torelli problem for irreducible symplectic manifolds''}},
  FJournal = {{Mathematische Annalen}},
  Journal = {{Math. Ann.}},
  ISSN = {0025-5831},
  Volume = {324},
  Number = {4},
  Pages = {847},
  Year = {2002},
  Publisher = {Springer, Berlin/Heidelberg},
  Language = {English},
  DOI = {10.1007/s00208-002-0344-2},
  URL = {https://doi.org/10.1007/s00208-002-0344-2},
}

@Book{NSW:CNF,
  Author = {J\"urgen {Neukirch} and Alexander {Schmidt} and Kay {Wingberg}},
  Title = {{Cohomology of number fields}},
  Series = {{Grund\-leh\-ren der Mathematischen Wissenschaften}},
  ISSN = {0072-7830; 2196-9701/e},
  Volume = {323},
  Edition = {2nd ed.},
  ISBN = {978-3-540-37888-4/hbk},
  Pages = {xv + 825},
  Year = {2008},
  Publisher = {Springer, Berlin},
  Language = {English},
  MSC2010 = {11-02 11R34 11R23 11S25 18G10 18G20 20J05 11R37},
  Zbl = {1136.11001},
  DOI = {10.1007/978-3-540-37889-1},
  URL = {https://doi.org/10.1007/978-3-540-37889-1},
}

@Article{OGr:99,
  Author = {Kieran G. {O'Grady}},
  Title = {{Desingularized moduli spaces of sheaves on a \(K3\)}},
  FJournal = {{Journal f\"ur die Reine und Angewandte Mathematik}},
  Journal = {{J. Reine Angew. Math.}},
  ISSN = {0075-4102},
  Volume = {512},
  Pages = {49--117},
  Year = {1999},
  Publisher = {De Gruyter, Berlin},
  Language = {English},
  DOI = {10.1515/crll.1999.056},
  URL = {https://doi.org/10.1515/crll.1999.056},
  MSC2010 = {14J60 14J28 14E15 14J10},
  Zbl = {0928.14029}
}

@Article{OGr:03,
  Author = {Kieran G. {O'Grady}},
  Title = {{A new six-dimensional irreducible symplectic variety}},
  FJournal = {{Journal of Algebraic Geometry}},
  Journal = {{J. Algebr. Geom.}},
  ISSN = {1056-3911},
  Volume = {12},
  Number = {3},
  Pages = {435--505},
  Year = {2003},
  Publisher = {American Mathematical Society (AMS), Providence, RI; University Press, Chicago, IL},
  Language = {English},
  DOI = {10.1090/S1056-3911-03-00323-0},
  URL = {https://doi.org/10.1090/S1056-3911-03-00323-0},
  MSC2010 = {53D35 32Q15 14J60 32J18 53C55 53C26 14H40 32Q55},
  Zbl = {1068.53058}
}

@Article{Oka:21,
  Author = {Okawa, Shinnosuke},
  Title = {An example of birationally inequivalent projective symplectic varieties which are {D}-equivalent and {L}-equivalent},
  FJournal = {Mathematische Zeitschrift},
  Journal = {Math. Z.},
  ISSN = {0025-5874},
  Volume = {297},
  Number = {1-2},
  Pages = {459--464},
  Year = {2021},
  Language = {English},
  DOI = {10.1007/s00209-020-02519-3},
  zbMATH = {7303580},
  Zbl = {1472.14040}
}

@Article{Orl:97,
  Author = {D. O. {Orlov}},
  Title = {Equivalences of derived categories and {{\(K3\)}} surfaces},
  FJournal = {Journal of Mathematical Sciences (New York)},
  Journal = {J. Math. Sci., New York},
  ISSN = {1072-3374},
  Volume = {84},
  Number = {5},
  Pages = {1361--1381},
  Year = {1997},
  Language = {English},
  DOI = {10.1007/BF02399195},
  zbMATH = {1269461},
  Zbl = {0938.14019}
}

@Article{Orl:02,
  Author = {D. O. {Orlov}},
  Title = {{Derived categories of coherent sheaves on abelian varieties and equivalences between them}},
  FJournal = {{Izvestiya: Mathematics}},
  Journal = {{Izv. Math.}},
  ISSN = {1064-5632; 1468-4810/e},
  Volume = {66},
  Number = {3},
  Pages = {569--594},
  Year = {2002},
  Publisher = {IOP Publishing, Bristol; London Mathematical Society, London; Turpion, London; Russian Mathematical Society, Moscow},
  Language = {English},
  MSC2010 = {18E30 14K05},
  Zbl = {1031.18007},
  DOI = {10.1070/IM2002v066n03ABEH000389},
  URL = {https://doi.org/10.1070/IM2002v066n03ABEH000389},
}

@Online{Orl:02a,
  Author = {D. O. {Orlov}},
  Title = {{Derived categories of coherent sheaves on abelian varieties and equivalences between them}},
  Year = {2025},
  Version = {4 (Jul~2025)},
  Eprinttype = {arxiv},
  Eprint = {alg-geom/9712017},
}

@PhdThesis{Plo:Thesis,
  Author = {David {Ploog}},
  Title = {{Groups of autoequivalences of derived categories of smooth projective varieties}},
  School = {Freie Universität Berlin, Institut für Mathematik und Informatik},
  Pages = {71},
  Year = {2005},
  Language = {English},
  MSC2010 = {14J28 18E30 14J32 14L30 14F05},
  Zbl = {1090.14011},
  URL = {https://www.math.uni-bonn.de/people/huybrech/phd-ploog.ps},
}

@Article{Plo:07,
  Author = {David {Ploog}},
  Title = {{Equivariant autoequivalences for finite group actions}},
  FJournal = {{Advances in Mathematics}},
  Journal = {{Adv. Math.}},
  ISSN = {0001-8708},
  Volume = {216},
  Number = {1},
  Pages = {62--74},
  Year = {2007},
  Publisher = {Elsevier (Academic Press), San Diego, CA},
  Language = {English},
  MSC2010 = {14L30 18E30},
  Zbl = {1167.14031},
  DOI = {10.1016/j.aim.2007.05.002},
  URL = {https://doi.org/10.1016/j.aim.2007.05.002},
}

@Article{Pol:96,
  Author = {A. {Polishchuk}},
  Title = {{Symplectic biextensions and a generalization of the Fourier-Mukai transform}},
  FJournal = {{Mathematical Research Letters}},
  Journal = {{Math. Res. Lett.}},
  ISSN = {1073-2780; 1945-001X/e},
  Volume = {3},
  Number = {6},
  Pages = {813--828},
  Year = {1996},
  Publisher = {International Press of Boston, Somerville, MA},
  Language = {English},
  MSC2010 = {14K05 18E30},
  Zbl = {0886.14019},
  DOI = {10.4310/MRL.1996.v3.n6.a9},
  URL = {https://doi.org/10.4310/MRL.1996.v3.n6.a9},
}

@Online{Rei:97,
  Author = {Miles {Reid}},
  Title = {{McKay correspondence}},
  Year = {1997},
  Version = {3 (Feb~1997)},
  Eprinttype = {arxiv},
  Eprint = {alg-geom/9702016},
}

@Article{Sch:11,
  Author = {I. {Schur}},
  Title = {{\"Uber die Darstellung der symmetrischen und der alternierenden Gruppe durch gebrochene lineare Substitutionen.}},
  FJournal = {{Journal f\"ur die Reine und Angewandte Mathematik}},
  Journal = {{J. Reine Angew. Math.}},
  ISSN = {0075-4102},
  Volume = {139},
  Pages = {155--250},
  Year = {1911},
  Publisher = {De Gruyter, Berlin},
  Language = {German},
  DOI = {10.1515/crll.1911.139.155},
  MSC2010 = {20-XX},
  Zbl = {42.0154.02}
}

@Book{Ser:GC,
  Author = {Jean-Pierre {Serre}},
  Title = {{Galois cohomology}},
  Series = {{Springer Monographs in Mathematics}},
  ISSN = {1439-7382},
  Edition = {2nd printing},
  Note = {Translated from the French by Patrick Ion},
  ISBN = {3-540-42192-0/hbk},
  Pages = {x + 210},
  Year = {2002},
  Publisher = {Springer, Berlin},
  Language = {English},
  MSC2010 = {12G05 12-02 11S25 11-02},
  Zbl = {1004.12003}
}

@Article{Shc:04,
  Author = {V. V. {Shchigolev}},
  Title = {{On some extensions of completely splittable modules}},
  FJournal = {{Izvestiya: Mathematics}},
  Journal = {{Izv. Math.}},
  ISSN = {1064-5632},
  Volume = {68},
  Number = {4},
  Pages = {833--850},
  Year = {2004},
  Publisher = {IOP Publishing, Bristol; London Mathematical Society, London; Turpion, London; Russian Mathematical Society, Moscow},
  Language = {English},
  DOI = {10.1070/IM2004v068n04ABEH000499},
  URL = {https://doi.org/10.1070/IM2004v068n04ABEH000499},
  MSC2010 = {20C30},
  Zbl = {1067.20012}
}

@PhdThesis{Mag:Thesis,
  Author = {Pablo {Magni}},
  Title = {{Derived Equivalences of Generalized Kummer Varieties}},
  School = {Radboud University Nijmegen},
  Pages = {xviii+145},
  Year = {2023},
  Language = {English},
  MSC2010 = {14F08, 14J42, 14L30},
  URL = {https://hdl.handle.net/2066/295328},
}

\end{document}